\documentclass[12pt]{amsart}
\usepackage{amssymb,color}
\usepackage{amsmath} 
\usepackage{amsthm} 
\usepackage{calc} 
\usepackage{verbatim} 
\usepackage{hyperref}

\newcommand{\re}{\mathrm{Re}\,}
\newcommand{\im}{\mathrm{Im}\,}

\newtheorem{thm}{Theorem}[section]
\newtheorem{lem}[thm]{Lemma}
\newtheorem{cor}[thm]{Corollary}
\newtheorem{prop}[thm]{Proposition}

\theoremstyle{definition}
\newtheorem{defn}[thm]{Definition}

\theoremstyle{remark}

\newtheorem{rem}[thm]{Remark}
\newtheorem{eg}[thm]{Example}

\newcounter{probno}
\newcommand{\Vol}{\mathrm{Vol}\,}

\newcommand{\codim}{\mathrm{codim}\,}

\newcommand{\id}{\mathrm{id}}

\newcommand{\tto}{\dashrightarrow}

\newcommand{\norm}[2][{}]{\left\|#2\right\|_{#1}}

\newcommand{\pair}[2]{\left\langle #1,#2 \right\rangle}

\newcommand{\R}{\mathbf{R}}

\newcommand{\C}{\mathbf{C}}
\newcommand{\cp}{\mathbf{P}}
\newcommand{\Z}{\mathbf{Z}}
\newcommand{\N}{\mathbf{N}}

\newcommand{\Q}{\mathbf{Q}}

\newcommand{\supp}{\mathrm{supp}\,}

\newtheorem*{ackn}{Acknowledgment}

\newcommand{\torus}{\mathbb{T}}
\newcommand{\nvec}{v}

\renewcommand{\div}{\operatorname{Div}}

\newcommand{\vol}{\operatorname{Vol}}

\newcommand{\ch}[2][{}]{[#2]_{#1}}        % cohomology class of a (1,1) current 
\newcommand{\pcc}{\mathcal{D}_{1,1}}       % positive closed $(1,1)$ currents
\newcommand{\cov}{\rho}              % change of volume factor for a toric map
\newcommand{\growth}[1]{\operatorname{Growth}(#1)}

\newcommand{\pcce}{\mathcal{D}_{ext}}
\newcommand{\pcci}{\mathcal{D}_{int}}

\newcommand{\sfn}{\psi}
\newcommand{\twoform}{\eta}
\newcommand{\ind}{\mathrm{Ind}}
\newcommand{\exc}{\mathrm{Exc}}
\newcommand{\dtop}{d_{top}}          % topological degree
\newcommand{\ddeg}{\lambda_1} %
\newcommand{\hoo}{H^{1,1}_\R} 
\newcommand{\eltwo}{\mathcal{L}^2}

\newcommand{\trop}{\operatorname{Log}}
\newcommand{\eucl}{\operatorname{eucl}}

\newcommand{\ram}{\operatorname{Ram}}

\newcommand{\rzt}{\hat\torus}
\newcommand{\rzf}{\hat f}
\newcommand{\rzh}{\hat h}

\newcommand{\boundingcurrent}{\bar S}
\renewcommand{\N}{\mathbf{Z}_{\geq 0}}

\DeclareMathOperator{\rs}{rs}
\DeclareMathOperator{\ex}{ex}
\DeclareMathOperator{\dpsh}{DPSH}
\DeclareMathOperator{\psh}{PSH}

\newcommand{\extra}{\mathcal{E}^-}

\title[Equidistribution without stability for toric surface maps]{Equidistribution without stability \\ for toric surface maps}
\author{Jeffrey Diller}
\address{Department of Mathematics\\
  University of Notre Dame\\
  Notre Dame, IN 46556\\
  USA}
\email{diller.1@nd.edu}

\author{Roland Roeder}

\address{IUPUI Department of Mathematics\\
         402 North Blackford Street room LD270\\
         Indianapolis, Indiana 46202-3267 \\
         USA}
\email{roederr@iupui.edu}

\subjclass[2020]{37F80 (primary), 14E05, 32H50, 32U40 (secondary)}

\begin{document}
\begin{abstract}
We prove an equidistribution result for iterated preimages of curves by a large
class of rational maps $f:\cp^2\tto\cp^2$ that cannot be birationally
conjugated to algebraically stable maps.  The maps, which include recent examples with transcendental first dynamical degree,
are distinguished by the fact that they have constant Jacobian determinant relative to the natural
holomorphic two form on the algebraic torus.  Under the
additional hypothesis that $f$ has ``small topological degree'' we
also prove an equidistribution result for iterated forward images of curves.

To prove our results we systematically develop the idea of a positive closed $(1,1)$ current and 
its cohomology class on the inverse limit of all toric
surfaces.   This, in turn, relies upon a careful study of positive closed
$(1,1)$ currents on individual toric surfaces.
This framework may be useful in other contexts.

%We prove an equidistribution result for preimages of curves by a large class of
%rational maps $f:\cp^2\tto\cp^2$ that cannot be birationally conjugated to
%algebraically stable maps.  The maps are distinguished by the fact that they
%have constant Jacobian determinant relative to the natural holomorphic two form
%on the algebraic two torus.  They include recent examples of plane rational
\end{abstract}

\maketitle
\markboth{\today}{\today}

\section{Introduction}
\label{sec:intro}

It is a central problem concerning the dynamics of a rational map $f:\cp^k\tto\cp^k$ on complex projective space $\cp^k$ to understand the asymptotic behavior of preimages $f^{-n}(V)$ of a suitably general proper subvariety $V\subset\cp^k$ as $n\to\infty$.  It was shown in \cite{RuSh97} that when the degree of $f^{-n}(V)$ tends to $\infty$ quickly enough with $n$, this behavior is independent of $V$.  Indeed one hopes that after normalizing, the preimages converge in the sense of integration currents to some invariant closed positive current whose support tracks the set of points on which the dynamics of $f$ is exponentially expanding in at least $\codim V$ directions.  The case $k=1$ was settled by Brolin \cite{Bro65} for polynomials $f:\C\to\C$ and separately by Lyubich \cite{Ly83} and by Freire, Lopes and Ma\~ne \cite{FLM83} for general rational maps $f:\cp^1\to\cp^1$.   They showed when $\deg f\geq 2$ that preimages $f^{-n}(z)$ of a non-exceptional point $z\in\cp^1$ equidistribute with respect an $f$-invariant measure whose support coincides with the Julia set of $f$.

Increasing the dimensions of both the domain and the subvariety by one, we encounter an additional difficulty.  For our purposes, the \emph{degree} $\deg(f)$ of a dominant rational map $f:\cp^2\tto\cp^2$ will be the degree of the preimage $f^{-1}(\ell)$ of a general line $\ell\subset\cp^2$.   One calls $f$ \emph{algebraically stable} if $\deg(f^n) = (\deg f)^n$ for all $n\in\N$.  Sibony \cite{Sib99} showed that if $f$ is algebraic stable with $\deg f\geq 2$, then there exists a positive closed current $T^*$ of bidegree $(1,1)$ such that for almost every curve $C\subset\cp^2$, one has weak convergence
$$
\frac{1}{\deg(f^n)} f^{-n}(C) \to (\deg C)\cdot T^*.
$$
In contrast to the one dimensional case, however, the algebraic stability condition can fail.  This is related (see \S\ref{subsec:ddegs}) to the fact that a rational map $f:\cp^2\tto\cp^2$ need not be continuously defined at all points.  The notion of algebraic stability generalizes to rational maps $f:X\tto X$ on arbitrary smooth projective surfaces, and many rational maps on $\cp^2$ that fail to be algebraically stable, nevertheless admit algebraically stable models in the sense that they lift to algebraically stable maps on rational surfaces $X\to \cp^2$ obtained by well-chosen finite sequences of point blowups (see \cite{Bir22} for a discussion).  It turns out, moreover, that Sibony's result continues to hold in this case (see Corollary 2.11 in \cite{DDG10}).  

Unfortunately there exist rational maps $f:\cp^2\tto\cp^2$ that not only fail to be algebraically stable but further fail to admit any stable model at all.  Favre \cite{Fav03} observed, for instance, that if $A=\begin{pmatrix} a & b \\ c & d\end{pmatrix}$ is an integer matrix whose eigenvalues $r e^{\pm 2\pi i\theta}$ are non-zero with irrational arguments $\pm\theta$, then the associated `monomial' map $f(x_1,x_2):= (x_1^a x_2^b, x_1^c x_2^d)$, cannot be conjugated by birational change of surface $\varphi:X\tto \cp^2$ to an algebraically stable map $f_X:X\tto X$.  While it is not especially difficult (see Theorem \ref{thm:monomialequidistribution} below) to understand equidistribution for preimages of curves $C\subset\cp^2$ by a monomial map, the first author and Lin \cite{DiLi16} later generalized Favre's observation to the larger and more diverse class of `toric' rational maps.  Here we call $f:\cp^2\tto\cp^2$ \emph{toric} if $f^*\eta = \cov(f)\eta$, where $\eta := \frac{dx_1\wedge dx_2}{x_1x_2}$ denotes the natural holomorphic two form on the algebraic torus $\torus = (\C^*)^2$ and $\cov(f)\in\C^*$ is a non-zero constant.  Aside from some recent skew-product examples by Birkett \cite{Bir22}, all known examples of plane rational maps that do not admit stable models are toric.  The purpose of this article is to show for toric surface maps that one can obtain equidistribution results for curves even when stable models are unavailable.

In order to state our main results, we briefly describe the central
construction on which they are based.  It is observed in \cite{DiLi16} that a
toric map $f:\cp^2\tto\cp^2$ is best analyzed by allowing the domain and range
of $f^n$ to vary through sequences of increasingly elaborate compactifications
of the algebraic torus $\torus := (\C^*)^2$.  Inspired by \cite{BFJ08} and
\cite{Can11} we take this idea to its logical extreme.  That is, we declare
$X\succ Y$ for smooth projective toric surfaces $X$ and $Y$ if the birational
map $\pi_{XY}:X\tto Y$ extending $\id:\torus\to\torus$ is a morphism; i.e.~ if
$X$ is canonically obtained from $Y$ by a finite sequence of point blowups.
We then consider the map $\rzf:\rzt\tto\rzt$ induced by $f$ on the inverse
limit $\rzt$ of all toric surfaces subject to this ordering.

As we explain in \S\ref{sec:rztoric}, the compact space $\rzt$ comes close to being a complex
surface: it contains $\torus$ as an open dense subset, and the complement
$\rzt\setminus\torus$ consists mainly of \emph{poles}, i.e. $\torus$-invariant
curves $C_\tau$ indexed by rational rays $\tau\subset\R^2$.  Likewise, the map
$\rzf:\rzt\tto\rzt$, which we define and study in \S\ref{sec:maps}, behaves
much like a rational map on a projective surface.  It is well-defined and
continuous off a finite \emph{persistently indeterminate} set
$\ind(\rzf)\subset\rzt$.  The image $\rzf(C)$ of each curve $C\subset\rzt$ by
$\rzf$ is either a curve or, if $C\subset\exc(\rzf)$ is one of finitely many
\emph{persistently exceptional} curves, a point.  The sets $\ind(\rzf)$ and
$\exc(\rzf)$ are empty precisely when the toric map $f$ is monomial.  More
generally, we call a toric map $f:\cp^2\tto\cp^2$ \emph{internally stable} if
$\rzf^n(\exc(\rzf)) \cap \ind(\rzf) = \emptyset$ for all $n\in\N$.  For toric
maps, this condition is far weaker than algebraic stability. It is satisfied by
all monomial maps and by the toric maps with transcendental dynamical degrees
exhibited in \cite{BDJ20}.  The work of \cite{HPmemoir} on dynamics of Newton's
method in two complex variables is an early instance of a construction similar
to the one we describe here, where the domain of a map is blown up infinitely
many times in order to understand its dynamics.

We prove in \S\ref{sec:monomial} that monomial maps are distinguished among toric rational maps $f:\cp^2\tto\cp^2$ by the fact that they have minimal \emph{first dynamical degree}
\begin{equation}
\label{eqn:planeddeg}
\ddeg(f) := \lim_{n\to\infty} \deg(f^n)^{1/n} \in [\dtop(f)^{1/2},\deg(f)].
\end{equation}
Here $\dtop(f)$ is the \emph{topological degree} of $f$, given by $\# f^{-1}(x)$ for a typical point $x$.

\begin{thm}
\label{thm:mainC}
If $f:\cp^2\tto\cp^2$ is an internally stable toric map such that $\ddeg(f)^2 = \dtop(f)$, then either 
\begin{itemize}
 \item $f$ is a \emph{shifted monomial map}, given by $x\mapsto yh(x)$ for some monomial map $h$ and fixed factor $y\in\torus$; or
 \item $\ddeg(f)$ is a positive integer and there exists a `distinguished' coordinate system in which $f$ has the skew product form 
$$
f(x_1,x_2) = (t x_1^{\pm\ddeg(f)},g(x_1)x_2^{\pm\ddeg(f)})
$$
for some $t\in\C^*$ and rational function $g:\cp^1\to\cp^1$.  
\end{itemize}
\end{thm}

\noindent Distinguished coordinates are introduced in \S\ref{sec:rztoric}.  We invite the reader to compare Theorem \ref{thm:mainC} with the earlier characterization by Favre and Jonsson (see \cite[Theorem C]{FaJo11}) of \emph{polynomial} maps of $\C^2$ with minimal first dynamical degree.

To state the next result, we recall a simple criterion from \cite{DiLi16} for identifying toric rational maps $f:\cp^2\tto\cp^2$ that admit algebraically stable models.  The matrix $A$ underlying a monomial map can be generalized (Theorem \ref{thm:tropmap}) to an arbitrary toric map $f$ by its `tropicalization', a continuous, positively homogeneous and piecewise linear map $A_f:\R^2\to\R^2$.  It governs the action of the induced map $\rzf$ on poles via $\rzf(C_\tau) = C_{A_f(\tau)}$.  In all known examples, $A_f$ is actually a homeomorphism, and so acts on rays by a circle homeomorphism.  The map $f$ then admits a stable model if and only if the rotation number of the circle homeomorphism is rational.  Our main equidistribution result concerns the complementary case.

\begin{thm}
\label{thm:mainA}
Let $f:\cp^2\tto\cp^2$ be a toric rational map whose tropicalization is a homeomorphism with irrational rotation number.  If $f$ is internally stable and not equal to a shifted monomial map, then $\ddeg(f) > 1$ and there exists a positive closed $(1,1)$ current $T^*$ on $\cp^2$ with the following properties.
\begin{itemize}
 \item $T^*$ does not charge curves.
 \item $f^*T^* = \ddeg(f) T^* + D$ where $D$ is an effective $\R$-divisor.
 \item For each curve $C\subset\cp^2$, there exists $m(C)>0$ such that 
 $
 \frac{1}{\deg(f^n)} (f^n)^* C \to m(C) T^*.
 $
\end{itemize}
\end{thm}

To prove this theorem we follow \cite{BFJ08} in \S\ref{sec:tcurrents} by introducing a space $\hoo(\rzt)$ of toric cohomology classes together with linear pullback and pushforward operators $\rzf^*,\rzf_*:\hoo(\rzt)\to\hoo(\rzt)$.  Indeed we go one step further by introducing a space of \emph{toric currents} $\pcc(\rzt)$ that naturally represent toric classes.  Curves in $\rzt$ are instances of positive toric currents.  

Implementing these ideas requires a substantial amount of preliminary work,
found in \S\ref{sec:currents}, concerning positive closed currents of bidegree
$(1,1)$ on fixed toric surfaces.  In particular we associate to any (not
necessarily $\torus$-invariant) positive closed $(1,1)$ current $T$ on $\torus$
a convex function $\sfn_T:\R^2\to\R$ and prove in Theorem
\ref{thm:lineargrowth} that $T$ extends trivially to a positive closed current
on some/any compact toric surface if and only if $\sfn_T$ has at most linear
growth.  In particular, any such current $T \in \pcc(\torus)$ corresponds to an element
of $\pcc(\rzt)$ that we call \emph{internal} to distinguish it from currents whose supports contain poles of $\rzt$.  An internal
current $T \in \pcc(\rzt)$ will be called \emph{homogeneous} if 
\begin{itemize}
 \item $\sfn_T$ is positively homogeneous, satisfying $\sfn(t\nvec) = t\sfn(\nvec)$ for all $t\geq 0$ and $\nvec\in\R^2$; and 
 \item $T = dd^c (\sfn_T\circ \trop)$, where $\trop:\torus\to \R^2$ is the `tropicalization' map, given in distinguished local coordinates by $\trop(x_1,x_2) = (-\log|x_1|,-\log|x_2|)$.  
\end{itemize}
Homogeneous currents serve as canonical representatives for their classes in $\hoo(\rzt)$.

\begin{thm}
\label{thm:mainB} 
The continuous linear map $T\mapsto \ch{T}$ associating each toric current to its cohomology class restricts to a homeomorphism from the cone of positive homogeneous currents in $\pcc(\rzt)$ onto the cone of nef classes in $\hoo(\rzt)$.
\end{thm}

With this setup, the proof of Theorem \ref{thm:mainA} proceeds as follows.  In \S\ref{sec:actions} we define pushforward and pullback by $\rzf$ on $\hoo(\rzt)$ and $\pcc(\rzt)$, following \cite{BFJ08}.  Internal stability of $f$ implies that $(\rzf^n)^* = (\rzf^*)^n$ for all $n\in\Z_{\geq 0}$.  The main result of \cite{BFJ08} then gives us that Theorem \ref{thm:mainA} holds on the cohomological level.  That is, there exists a cohomology class $\alpha^*\in\hoo(\rzt)$ such that for any nef class $\alpha\in\hoo(\rzt)$, we have
$\frac{f^{n*}\alpha}{\deg(f^n)} \to m(\alpha) \alpha^*$ for some constant $m(\alpha)> 0$.  

In \S\ref{sec:invcurrents}, we conclude the proof Theorem \ref{thm:mainA} by passing from cohomology classes to currents.  If $C\subset\rzt$ is a curve that is \emph{internal} in the sense that $C\cap\torus \neq \emptyset$, then the normalized pullbacks $C_n = \frac{1}{\ddeg(f)^n} f^{n*}C$, $n\geq 0$, are also internal, and their cohomology classes $\alpha_n$ are nef.  If $\bar T_n\in\pcc(\rzt)$ is the canonical homogeneous representative of $\ddeg(f)^{-n}\rzf^{n*}\alpha$, then we may write $C_n = \bar T_n + dd^c \varphi_n$ for some relative potential $\varphi_n$ on $\torus$.  The currents $\bar T_n$ converge to the canonical representative $\bar T$ of $\alpha^*$ by Theorem \ref{thm:mainB}.  So modulo some translation of statements on $\rzt$ back to $\cp^2$, it suffices for establishing Theorem \ref{thm:mainA} to prove for each toric surface $X$ that the potentials $\varphi_n$ converge in $L^1(X)$.  

The key to this convergence is Theorem \ref{thm:volshrink}, which provides a weak but sufficient estimate on the way iterates of a toric surface map can shrink volumes of open subsets.  It's proof, which occupies most of \S\ref{sec:volumes}, takes advantage of the fact that the tropicalization of $f$ also governs the ramification of $\rzf$ about poles and therefore the rate at which orbits of most points can escape $\torus$.
With Theorem \ref{thm:volshrink}, convergence of the potentials $\varphi_n$ follows from the internal stability hypothesis and rewriting $\lim\varphi_n$ as a telescoping series involving pullbacks of potentials for $f^*\bar T_n - \ddeg(f) \bar T_{n+1}$.  A virtue of this approach is that it guarantees equidistribution for preimages of \emph{all} (instead of almost all) curves in $\cp^2$.  The article \cite{FAVRE_JONSSON}, which deals with holomorphic endomorphisms of $\cp^2$, is a precedent: using volume estimates it shows for such maps that equidistribution of preimages fails for at most finitely many curves.

In the case when the toric map $f$ has \emph{small topological degree}, i.e. when $\dtop(f)<\ddeg(f)$, more can be shown: forward images of curves also equidistribute to a positive closed $(1,1)$ current, and both forward and backward equilibrium currents have additional geometric structure.  

\begin{thm}
\label{thm:mainD}
Suppose that $f:\cp^2\tto\cp^2$ is an internally stable toric rational map with small topological degree $\dtop(f)<\ddeg(f)$.  Then the backward equidistribution current $T^*$ for $f$ is laminar.  There exists, moreover, a positive closed $(1,1)$ current $T_*$ on $\cp^2$ such that
\begin{itemize}
 \item $T_*$ is woven and does not charge curves;
 \item $f_*T_* = \ddeg(f)T_* + D$, where $D$ is an effective $\R$-divisor;
 \item for any curve $C\subset\cp^2$, there exists $m(C)>0$ such that 
 $\frac{1}{\deg(f^n)} f^n_*C \rightarrow m(C)T_*$.
\end{itemize}
\end{thm}

Roughly speaking, a positive closed $(1,1)$ current is said to be laminar if it can be expressed as a sum of (non-closed) positive $(1,1)$ currents $T_n$, each given in some coordinate polydisk $P\subset\cp^2$ as an average of currents of integration over disjoint graphs of holomorphic functions with respect to some non-negative Borel measure.  We say more about laminarity below but refer the reader to \cite{DDG10} \S~2.4 and \S~3.3 and the references therein for a thorough discussion.  The notion of wovenness is similar, except that the graphs are not required to be disjoint.

The results above suggest several further issues, and we intend to return to at least some of them in future work.  First of all, the condition in Theorem \ref{thm:mainA} that the toric map $f$ be internally stable is convenient but likely unnecessary.  Theorem D in \cite{DiLi16} implies that if one is willing to allow for domains a bit more general than toric surfaces, then \emph{any} toric map becomes internally stable.  It seems likely that the inverse limit constructions and results concerning $\rzt$, $\rzf$, $\pcc(\rzt)$ and $\hoo(\rzt)$ will adapt to the more general setting, though the details will certainly be more complicated.

It would also be desirable to know whether there exist toric surface maps $f:\cp^2\tto\cp^2$ for which $A_f$ is not a homeomorphism.  If no such maps exist, then we have a tricotomy.  Either $f$ is monomial, $f$ satisfies the hypotheses of Theorem \ref{thm:mainA}, or by Theorem F in \cite{DiLi16}, $f^2$ admits a stable model.  In particular, by applying Corollary 2.1 from \cite{DDG10} to the last case, we have a more or less complete understanding of equidistribution for preimages of curves by toric surface maps.

Finally, in the case when $f$ has small topological degree, one would like to interpret the wedge product $T^*\wedge T_*$ as an $f$-invariant measure of maximal entropy $\log\ddeg(f)$ as has been done e.g. for polynomial diffeomorphisms $f:\C^2\to\C^2$ (see \cite{BeSm91}) or for many surface maps that admit stable models (see \cite{DDG11}).  In the case when $\lambda_1(f)$ is transcendental, as in \cite{BDJ20}, this would imply that there exist rational surface maps whose entropy exponentiates to a transcendental number. 

The literature concerning equidistribution in holomorphic and rational dynamics
is by now very large.  Sources that we have not already mentioned include, but
are hardly limited to,
\cite{FoSi95,Dil96,FAVRE_GUEDJ,FAVRE_THESIS,GUEDJ_VOL1,GUEDJ_VOL2,DS,TAFLIN,PROTIN,BLR}.
We encourage the reader to consult with these to learn more about the specific
issues we address here as well as ones we do not touch, such as subvarieties of
higher codimension, meromorphic correspondences, and speed of convergence
problems.  

We note the very interesting alternative treatment of `homogeneous' positive closed currents on smooth toric varieties by Babaee and Huh \cite{BaHu17}, which they used to disprove a variant of the Hodge conjecture.  Babaee \cite{Bab23} has recently used the same approach to construct invariant currents for some simple (in particular, algebraically stable) monomial maps.  {We also stress that we are certainly not the first to make use of the inverse limit of the set of toric varieties in a given dimension.  See for instance the recent paper \cite{Bot22} which considers (using different terminology and notation) the inverse limit $\rzt$ of all toric surfaces and the corresponding vector space $\hoo(\rzt)$, conceived of as equivalence classes of `toric $b$-divisors' (a subset of what we here call toric currents).  On the strictly cohomological level, it was shown much more generally in \cite{FuSt97} that the inverse limit of the cohomology rings over all toric varieties of fixed dimension $d$ is isomorphic to the McMullen polytope algebra in $\R^d$.}

Note also that the dynamics of toric rational mappings has been studied from the perspective of itegrable systems and cluster algebras; see \cite{MO} and the references therein.

To close we briefly summarize the contents of each section that follows. 
\begin{itemize}
\item[\S\ref{sec:background}] supplies background concerning complex surfaces, rational maps and currents. 
\item[\S\ref{sec:rztoric}] reviews toric surfaces and introduces the inverse limit $\rzt$ of such surfaces. 
\item[\S\ref{sec:maps}] introduces and proves the foundational facts about toric maps.
\item[\S\ref{sec:volumes}] proves the key volume estimate Theorem \ref{thm:volshrink}.
\item[\S\ref{sec:currents}] analyzes the relationship between positive closed $(1,1)$ currents on $\torus$ and those on toric surfaces.
\item[\S\ref{sec:tcurrents}] introduces the spaces of toric classes and toric currents on $\rzt$.
\item[\S\ref{sec:actions}] introduces and analyzes the pullback and pushforward actions induced by a toric map on toric classes and currents.  
\item[\S\ref{sec:monomial}] gives the proof of Theorem \ref{thm:mainC} as well as the statement and proof of an equidistribution result for monomial maps that parallels Theorem \ref{thm:mainA}.
\item[\S\ref{sec:invcurrents}] provides the proofs of Theorems \ref{thm:mainA} and \ref{thm:mainD}.
\end{itemize}

\medskip
\begin{ackn}
  We thank Romain Dujardin, Vincent Guedj and Duc-Viet Vu for helpful answers
to questions that came up during the writing of this paper, Misha Lyubich for
his useful comments following a talk about our work, and Farhad Babaee, Mattias Jonsson,
Charles Favre and Simion Filip for helpful questions and comments about a draft
of this article.  {We are also indebted to the anonymous referee for several constructive commnets.  
Finally, both authors gratefully acknowledge support from the NSF, the
first by grants DMS-1954335, DMS-2246893 and the second by grant DMS-2154414.}
\end{ackn}

\section{Background}
\label{sec:background}

Except where otherwise noted, a \emph{surface} in this article will be a smooth complex two dimensional projective variety $X$, and a \emph{curve} on $X$ will be an irreducible one dimensional complex algebraic subset $C\subset X$.  In this section, we give a quick summary of facts and terminology that we will rely on below concerning surfaces, currents and rational maps.

\subsection{Divisors and currents} For each surface $X$, we let $\div(X)$ denote the vector space of $\R$-divisors on $X$ and $\pcc^+(X)$ denote the cone of real positive closed $(1,1)$ currents, endowing the latter with the weak topology.  We further let $\pcc(X) := \pcc^+(X) -\pcc^+(X)$ be the vector space generated by $\pcc^+(X)$.  This is non-standard, since it excludes many closed $(1,1)$ currents, but it suits our purposes for reasons that we will explain shortly.  We extend the weak topology on $\pcc^+(X)$ to a topology on $\pcc(X)$ by declaring that $(T_j)\subset \pcc(X)$ converges if and only if we can write $T_j = T_j^+ - T_j^-$, where $(T_j^+)$, $(T_j^-)\subset \pcc^+(X)$ are both weakly convergent sequences.  We freely conflate divisors with their associated integration currents.  

Siu's theorem \cite{Siu74} implies that any current in $\pcc(X)$ that is supported on a finite union of curves is a divisor.  Complementing this fact, the Skoda-El Mir Theorem tells us that if $C\subset X$ is a curve and $T \in \pcc^+(X \setminus C)$, then $T$ is the restriction of a positive closed $(1,1)$ current on $X$ precisely when the measure $T\wedge \omega$ has finite mass for some/any K\"ahler form $\omega$ on $X$.  Under this condition, if $T$ is expressed as a form with measure coefficients, then extending each measure by zero to $C$ gives the \emph{trivial extension} $\tilde T\in\pcc^+(X)$ of $T$.
See, e.g., \cite[Theorem 2.3]{DemBook} for more details.  We will freely apply the Skoda-El Mir Theorem to (non-positive) currents
$T \in \pcc(X \setminus C)$ by writing $T = T^+ - T^-$ with $T^+$ and $T^- \in \pcc^+(X \setminus C)$ and checking the finite mass hypothesis for both $T^+$ and $T^-$.

We let $\hoo(X)\subset H^2(X,\R)$ denote the set of cohomology classes $\ch{T}$ of currents $T\in\pcc(X)$.  In the sequel all our surfaces will be rational, in which case $\hoo(X)$ and $H^2(X,\R)$ coincide, and every class in either cohomology group is represented by an $\R$-divisor.  By the $dd^c$-lemma from K\"ahler geometry, two currents $S,T\in \pcc(X)$ are cohomologous if and only if there exists $\varphi\in L^1(X)$, locally given as a difference of plurisubharmonic functions, such that $dd^c \varphi = T-S$.  Recall that if $T_j = T + dd^c \varphi_j\in\pcc^+(X)$ is a sequence of \emph{positive} closed currents, all representing the same cohomology class $[T]\in\hoo(X)$, and the relative potentials $\varphi_j$ are normalized so that $\int_X \varphi\,dV = 0$ relative to some volume form on $X$, then $(T_j)$ converges weakly if and only if $\varphi_j$ converges in $L^1(X)$.
  
We let $(\alpha\cdot\beta) = (\alpha\cdot\beta)_X \in\R$ denote the intersection pairing between cohomology classes $\alpha,\beta\in \hoo(X)$.  A class $\alpha$ is \emph{nef} if $(\alpha\cdot D) \geq 0$ for all effective $D\in\div(X)$.  Every nef class $\alpha$ satisfies $\alpha^2 := (\alpha\cdot\alpha)\geq 0$.  If $\alpha \in\hoo(X)\setminus\{0\}$ has non-negative self-intersection and $\alpha^\perp\subset\hoo(X)$ denotes the $(\cdot,\cdot)$-orthogonal complement of $\alpha$, then the Hodge Index Theorem says that $\beta^2\leq 0$ for all $\beta\in\alpha^\perp$ with equality if and only if $\beta$ is a multiple of $\alpha$.  
  
\subsection{Rational maps}
\label{subsec:ratmaps}
If $f:X\tto Y$ is a rational map between surfaces $X$ and $Y$, we let $\ind(f)\subset X$ denote the finite set of \emph{indeterminate} points, where $f$ cannot be continuously defined.  For each curve $C\subset X$, we let $f(C) := \overline{f(C\setminus\ind(f))}\subset Y$ denote the strict transform of $C$ by $f$.  This is either another curve or a single point. In the latter case we call $C$ \emph{exceptional} for $f$.  We will always assume that our rational maps are \emph{dominant}, i.e. that $f(X\setminus I(f))$ contains an open subset of $Y$.  In this case, there are only finitely many exceptional curves for $f$, and we let $\exc(f)$ denote their union.

The restriction $f|_C:C\to f(C)$ of $f$ to any curve $C$ is a holomorphic map that is well-defined even on $\ind(f)\cap C$.  A point $p\in X$ is indeterminate for $f$ precisely when the set $f(p):=\{f|_C(p):C\subset X\text{ is a curve containing }p\}$ is not just a point but rather a finite, connected union of curves.

The main reason for our non-standard definition of $\pcc(X)$ is that it allows
(see \cite{DDG10} \S~1.2) us to associate to any rational surface map $f:X\tto
Y$ continuous linear pushforward and pullback maps $f_*:\pcc(X)\to\pcc(Y)$ and
$f^*:\pcc(Y)\to\pcc(X)$, on closed $(1,1)$ currents.  Both maps preserve
divisors and positivity.  If $T = dd^c\varphi$ is cohomologous to zero, then
$f^*T = dd^c(\varphi\circ f)$ and $f_*T = dd^c f_*\varphi$ are, too.  Here $f_*
\varphi(q) := \sum_{f(p) = q} \varphi(p)$, where the preimages are counted with multiplicity.  Hence $f^*$ and $f_*$ descend to
compatible linear maps $f_*:\hoo(X)\to \hoo(Y)$ and $f^*:\hoo(Y)\to\hoo(X)$ on
cohomology classes.  These are adjoint with respect to intersection: for all
$\alpha\in\hoo(Y)$, $\beta\in\hoo(X)$ we have
\begin{equation}
\label{eqn:projformula}
(f^*\alpha\cdot\beta)_X = (\alpha\cdot f_*\beta)_Y.
\end{equation}
It follows that $f^*$ and $f_*$ also preserve nef classes.  Below (see Theorem
\ref{thm:pushpull2} and its consequence in the proof of Theorem \ref{thm:mainC}) we make
use of the following additional fact from \cite{DiFa01}.  Recall that the
\emph{topological degree} $\dtop(f)$ of $f:X\tto Y$ is the number of preimages
of a general point $p\in X$.

\begin{thm}
\label{thm:pushpull1}
Let $f:X\tto Y$ be a dominant rational map between surfaces $X$ and $Y$.  Then the linear operator $\extra_f:\hoo(Y)\to\hoo(Y)$ defined by $\extra_f(\alpha) = f_*f^*\alpha - d_{top}(f) \alpha$ satisfies the following for any $\alpha\in\hoo(Y)$.
\begin{itemize}
 \item $\extra_f(\alpha)$ is represented by a divisor supported on $f(\ind(f))$.
 \item $\extra_f(\alpha)$ is effective whenever $\alpha$ is nef.
 \item $\extra_f(\alpha) = 0$ if and only if $(\alpha\cdot C) = 0$ for all $C\subset f(\ind(f))$.
 \item $(\extra_f(\alpha)\cdot \alpha) \geq 0$ with equality if and only if $\extra_f(\alpha) = 0$.
\end{itemize}
\end{thm}

In the special case of birational morphisms $\pi:X\to Y$, i.e. finite compositions of point blowups, pullback is an intersection isometry, satisfying $(\pi^*\alpha\cdot \pi^*\beta)_X = (\alpha\cdot \beta)_Y$ for all $\alpha,\beta\in\hoo(Y)$.    In particular if $\div(\exc(\pi))\subset\div(X)$ denotes those divisors supported on $\pi$-exceptional curves, then $(\cdot,\cdot)_X$ is negative definite on $\div(\exc(\pi))$ and we have an orthogonal decomposition $\hoo(X)\cong \pi^*\hoo(Y)\oplus \div(\exc(\pi))$.

A central difference between rational maps and morphisms is that pushforward and pullback by rational maps do not necessarily respect composition. 

\begin{prop}[See \cite{DiFa01}, Proposition 1.13]
\label{prop:stablecomp}
Suppose that $g:X\tto Y$ and $f:Y\tto Z$ are rational maps between surfaces. Then the following are equivalent.
\begin{itemize}
 \item $(f\circ g)^* = g^* \circ f^*$ (on currents and/or classes).
 \item $(f\circ g)_* = f_* \circ g_*$ (on currents and/or classes).
 \item $g(\exc(g)) \cap \ind(f) = \emptyset$.
\end{itemize}
\end{prop}

\subsection{Rational self-maps and dynamical degree(s)}
\label{subsec:ddegs}

When a rational map $f:X\tto X$ has the same domain and range, the broad
features of its dynamics are governed by two numerical invariants.  The
\emph{topological degree} of $f$ is the number of preimages $\dtop(f) :=
\# f^{-1}(p)\geq 1$ of a general point $p\in X$.  Of course, $\dtop(f^n) =
\dtop(f)^n$ for all $n\in\Z_{\geq 0}$.  The \emph{(first) dynamical degree}
$\ddeg(f)$ of $f$ similarly tracks the growth of preimages of curves, but the
definition is more elaborate.

\begin{thm}[See \cite{DiFa01}, Proposition 1.18 and Remark 5.2] 
\label{thm:ddegfacts}
If $f:X\tto X$ is a rational self-map of a surface $X$ and $D,D'\in \div(X)$ are ample divisors, then
\begin{equation}
\label{eqn:ddeg}
\ddeg(f) := \lim_{n\to\infty} (f^{n*} D\cdot D')^{1/n}
\end{equation}
is well-defined, independent of the choice of $D$ and $D'$ and satisfies $\ddeg(f)^2\geq \dtop(f)$.  If, moreover, $\varphi:X\tto Y$ is birational, then
$\ddeg(\varphi\circ f\circ \varphi^{-1}) = \ddeg(f).$
\end{thm}

In the case of plane rational maps $f:\cp^2\tto\cp^2$, one checks that the formula for $\ddeg(f)$ reduces to \eqref{eqn:planeddeg}.  The dynamical degree is difficult to compute in general and can even take transcendental values (see \cite{BDJ20}), but a further commonly satisfied condition on $f$ makes it easier to understand.

\begin{defn} A rational self-map $f:X\tto X$ is \emph{algebraically stable} if $f^n(\exc(f))\cap\ind(f) = \emptyset$ for all $n\geq 1$.
\end{defn}

\noindent In light of Proposition \ref{prop:stablecomp} algebraic stability is equivalent to the condition that $(f^*)^n = (f^n)^*$ for all $n\in\N$, and in this case, one shows (see \cite{DiFa01} \S~5) that $\ddeg(f)$ is the (always real and positive) largest eigenvalue of $f^*$.

\subsection{Lelong numbers and rational maps}

We recall (see e.g. \cite{DemBook}, Chapter 3)  that the \emph{Lelong number} of a current $T\in\pcc^+(X)$ at a point $p\in X$ is a non-negative real number $\nu(T,p)$ that is positive if and only if $T$ is `maximally concentrated' at $p$.  If, for instance, $T$ is the current of integration over a divisor $D = \sum c_j C_j\in\div(X)$, then $\nu(T,p)$ is positive at every $p\in\supp D$, and equal to $c_j$ at general points of the component $C_j$.  If, on the other hand, $T\in\pcc^+(X)$ does not charge any curves in $X$, then the set $\{p\in X:\nu(T,p) > c\}$ is discrete and closed for any $c>0$.

We will repeatedly use the following fact established in \cite{Fav99}.

\begin{lem}
\label{lem:favrebd}
Let $g:X\tto Y$ be a dominant rational map between complex surfaces $X$ and $Y$ and $T$ be a positive closed $(1,1)$ current on $Y$.  Then there exists a constant $C>0$ such that for any $p\in X\setminus\ind(g)$, we have
$$
\nu(T,g(p)) \leq \nu(p,g^*T) \leq C\nu(T,g(p)).
$$
If $g$ is locally finite near $p$, then it suffices to take $C$ to be the local topological degree of $g$ about $p$.
\end{lem}

\section{Toric surfaces}
\label{sec:rztoric}

Let $\torus \cong (\C^*)^2$ denote the two dimensional complex algebraic torus.  For our purposes, a surface $X$ is \emph{toric} if it is `marked' with an embedding $\torus\hookrightarrow X$ such that the natural action of $\torus$ on itself extends holomorphically to an action on all of $X$.  In this section we review some needed facts about such surfaces.  A much more comprehensive, and by now standard, treatment of toric varieties may be found in \cite{Ful93}.  

Following common convention, we let $N\cong{\Z^2}$ denote the dual of the character lattice for $\torus$ and $N_\R := N\otimes \R\cong\R^2$.  If $\torus_\R\subset\torus$ denotes the maximal compact subgroup, a real two dimensional torus, then we may fix a group isomorphism $\torus/\torus_\R\to N_\R$ (unique up to multiplication by a non-zero real number) and let $\trop:\torus \to N_\R$ denote the quotient `tropicalization' map.  Haar measure on $\torus_\R$ extends to a $\torus$-invariant holomorphic two form on $\torus$ that we will denote by $\eta$ and use below to regularize closed currents and their potentials.  

Any toric surface $X$ {(smooth and projective, and therefore normal and compact, per our conventions for surfaces at the beginning of \S\ref{sec:background})} is given by its \emph{fan}, i.e. a finite partition 
$$
\Sigma(X) = \{0\}\cup \Sigma_1(X) \cup \Sigma_2(X),
$$
of $N_\R(X)$ into open $0$, $1$ and $2$-dimensional cones, each of which is `rational' in the sense that its closure is generated by elements of $N$.  In particular, the curves $C_\tau\subset X\setminus\torus$ are indexed by the rays $\tau\in\Sigma_1(X)$.  Since they are simple poles for the $2$-form $\eta$, we will call these curves \emph{poles} of $X$.  All other curves $C\subset X$, i.e. those for which $C\cap\torus\neq\emptyset$ will be \emph{internal}.  We likewise call points in $X$ \emph{internal} if they lie in $\torus$ and \emph{external} if they do not. 

We call $D\in\div(X)$ an \emph{external} divisor if it is supported entirely on poles, writing $\pcce(X)$ for the set of all external divisors.  For example, when regarded as a meromorphic two form on $X$, the divisor of $\eta$ is external equal to $-\sum_{\tau\in\Sigma_1(X)} C_\tau$. 
One associates to any external divisor $D = \sum_{\tau\in\Sigma_1(X)} c_\tau C_\tau\in\div(X)$ its \emph{support function} $\sfn_D:N_\R\to \R$, which is uniquely specified by declaring $\sfn_D(\nvec) = c_\tau$ when $\nvec\in N$ is the primitive vector that generates a ray $\tau\in \Sigma_1(X)$ and then extending linearly and continuously to each cone in $\Sigma(X)$.  By definition $\sfn_D$ is non-negative if and only if $D\geq 0$ is effective.  We further have that
\begin{itemize}
 \item $D\in\pcce(X)$ is principal if and only if $\sfn_D$ is linear;
 \item $D\in\pcce(X)$ is nef if and only if $\sfn_D$ is convex\footnote{Recall that a divisor is nef if its intersection with any effective divisor is non-negative};
 \item every $D\in \div(X)$ is linearly equivalent to an external divisor.
\end{itemize}
Combining all three of these assertions we see that any nef divisor $D\in\div(X)$ is linearly equivalent to an \emph{effective} external divisor, i.e.\ one with a \emph{non-negative} convex support function.

Poles $C_{\tau_1},C_{\tau_2}\subset X$ intersect if and only if $\tau_1$ and $\tau_2$ are adjacent rays in $\Sigma_1(X)$.  The unique point of intersection is then invariant by the action of $\torus$, and we denote it by $p_\sigma$, where $\sigma\in\Sigma_2$ is the sector bounded by $\tau_1$ and $\tau_2$.  Smoothness of $X$ at $p_\tau$ is equivalent to the condition that the primitive elements $\nvec_1,\nvec_2\in N$ generating $\tau_1$ and $\tau_2$ form a basis for $N$.  The generators of the cone dual to $\sigma$ in the character lattice for $\torus$ give an isomorphism of algebraic groups $\torus \to (\C^*)^2$, and this extends to a `distinguished' holomorphic coordinate system $(x_1,x_2):X\setminus \bigcup_{\tau\neq\tau_1,\tau_2} C_\tau\to \C^2$ about $p_\sigma$ in which $C_{\tau_j} = \{x_j=0\}$. 

If $p_{\sigma'} = C_{\tau_1'}\cap C_{\tau_2'}\in X$ is another $\torus$-invariant point, and $\nvec_1',\nvec_2'\in N$ are the primitive vectors generating the rays $\tau_1',\tau_2'$, then the transition between distinguished coordinates about $p_{\sigma}$ to those about $p_{\sigma'}$ is given by the birational monomial map $(x'_1,x_2') = h_A(x_1,x_2) := (x_1^{A_{11}}x_2^{A_{12}},x_1^{A_{21}}x_2^{A_{22}})$ where $A\in\operatorname{GL}(2,\Z)$ is the matrix for changing basis from $\{\nvec_1,\nvec_2\}$ to $\{\nvec_1',\nvec_2'\}$.
In any case, we normalize $\trop$ and $\eta$ so that in (any) distinguished coordinates, we have   
\begin{equation}
\label{eqn:twoform}
\trop(x_1,x_2) = -\log|x_1|\nvec_1-\log|x_2|\nvec_2 \quad\text{and}\quad
\eta = \pm\frac{1}{4\pi^2}\frac{dx_1\wedge dx_2}{x_1x_2}.
\end{equation}
In what follows, we let $X^\circ := X\setminus \{p_\sigma:\sigma\in\Sigma_2(X)\}$ denote the complement of the $\torus$-invariant points of $X$.  Similarly, if $C_\tau\subset X$ is an pole, we let $C^\circ_\tau = C_\tau\cap X^\circ \cong \C^*$ be the complement of the two $\torus$-invariant points of $C_\tau$.  

\subsection{Inverse limits of toric surfaces}

If $X$ and $Y$ are toric surfaces, we let $\pi_{XY}:X\tto Y$ be the birational map extending $\id:\torus\to\torus$ via the markings of $X$ and $Y$.  For each $\tau\in \Sigma_1(Y)\cap \Sigma_1(X)$, we have $\pi_{XY}(C_\tau) = C_\tau$, and for each $\tau\in\Sigma_1(X)\setminus\Sigma_1(Y)$ we have $\pi_{XY}(C_\tau) = p_\sigma$, where $\sigma\in\Sigma_2(Y)$ is the unique sector containing $\tau$.  We write $X\succ Y$, saying $X$ \emph{dominates} $Y$ if $\pi_{XY}$ is a morphism, i.e. $\Sigma_1(Y)\subset\Sigma_1(X)$.  In this case $\pi_{XY}$ is a homeomorphism over the complement $Y^\circ$ of the $\torus$-invariant points of $Y$.

Since for any two toric surface there is a third toric surface dominating both, it follows that the set of all toric surfaces forms a directed set with the partial order $\succ$.  Following e.g. \cite{BFJ08,Can11}, we may thus consider the inverse limit $\rzt$ of all toric surfaces $X$.  A point $p\in\rzt$ is given by a collection $\{p_X\in X\}$ consisting of one point from every toric surface $X$ and subject to the compatibility condition $\pi_{XY}(p_X) = p_Y$ whenever $X\succ Y$.  

We give $\rzt$ the product topology, declaring that $p_j \to p\in\rzt$ if for every $X$, we have $p_{j,X} \to p_X$.  {It follows from Tychonoff's Theorem that, as the inverse limit of compact Hausdorff spaces, $\rzt$ is compact.}  The $\torus$-action on individual toric surfaces $X$ is compatible with the ordering $\succ$ and so ascends to a continuous action of $\torus$ on $\rzt$.  We let $\rzt^\circ\subset\rzt$ denote the points which are \emph{not} fixed by this action.  

For any particular toric surface $X$, we let $\pi_{\rzt X}:\rzt\to X$ denote the continuous surjection that assigns a point $p$ to its representative $p_X\in X$.  By construction $\pi_{\rzt X}$ is a homeomorphism over $X^\circ$ whose inverse gives a canonical inclusion $X^\circ\hookrightarrow \rzt^\circ$.  These inclusions are compatible for different toric surfaces, we henceforth we regard $X^\circ$ as subset of $\rzt^\circ$; in particular $\torus\subset\rzt^\circ$ and $C^\circ_\tau\subset \rzt^\circ$ for every rational ray $\tau\subset N_\R$.  The following is more or less immediate from the discussion above.

\begin{prop}  
\label{prop:mfldpts}
The complement $\rzt^\circ$ of the $\torus$-invariant points of $\rzt$ is given by $\rzt^\circ = \bigcup_X X^\circ$, where the union is over all toric surfaces.  In particular $\rzt^\circ$ is a (non-algebraic) complex surface.  Furthermore,
\begin{enumerate}
 \item $\bigcap_X X^\circ = \torus$; and
 \item $\rzt^\circ\setminus\torus = \sqcup_{\tau}{C_\tau^\circ}$, where the union is over rational rays $\tau\subset N_\R$.
\end{enumerate}
\end{prop}

\noindent We will call points in $\rzt^\circ$ \emph{realizable}, saying that a toric surface $X$ realizes $p \in \rzt^\circ$ if $p\in X^\circ$.  As with points in toric surfaces, we continue to call points in $\torus \subset \rzt$ internal and points in $\rzt\setminus\torus$ external.  Internal points are realized in every toric surface.  A realizable external point $p\in C_\tau^\circ$ is realized in every toric surface $X$ for which $\tau\in\Sigma_1(X)$.  

The non-realizable, i.e. $\torus$-invariant, points of $\rzt$, are accounted for as follows.
\begin{prop}
For each rational ray $\tau\subset N_\R$, the closure $C_\tau :=
\overline{C_\tau^\circ}$ contains exactly two distinct $\torus$-invariant
points.  For any toric surface $X$ with $\tau\in\Sigma_1(X)$, the restriction
of $\pi_{\rzt X}$ to $C_\tau$ is a homeomorphism onto the pole of $X$ with
the same name.

On the other hand, there is a bijective correspondence between
$\torus$-invariant points $p_\tau \in \rzt$ not contained in poles of $\rzt$ and
irrational rays $\tau \subset N_\R$.  That is, $p = p_\tau$ is represented in
each toric surface $X$ by the $\torus$-invariant point $p_\sigma$ determined by
$\tau\subset \sigma\in\Sigma_2(X)$.

\end{prop}

% \noindent From now on, we call a $\torus$-invariant point $p\in\rzt$ \emph{rational} if it is contained in some (necessarily unique) pole $C_\tau\subset\rzt$  and \emph{irrational} otherwise.  

We declare a (necessarily closed) subset $C\subset\rzt$ to be a curve if $C_X:=\pi_{\rzt X}(C)\subset X$ is a curve for sufficiently dominant toric surfaces $X$.  That is, for any toric surface $Z$, there exists $Y\succ Z$ such that $C_X$ is a curve for all $X\succeq Y$.  We say that such surfaces $X$ \emph{realize} $C$, and will typically drop the subscript where there is little risk of confusion, letting $C$ also denote its representative in $X$.  

Hence poles $C_\tau\subset \rzt$ are curves realized by any surface $X$ for which $\tau\in\Sigma_1(X)$.  All other curves $C\subset \rzt$ are \emph{internal} with non-trivial intersection $C\cap\torus$.  An internal curve $C$ is realized by every toric surface, and its representatives $C_X\subset X$ are themselves internal curves.  There is, however, a significant distinction to be made among these representatives.  The following can be proved inductively by repeatedly blowing up $\torus$-invariant points on an internal curve. 

\begin{prop} 
\label{prop:internalcurve}
If $C\subset\rzt$ is internal, then on sufficiently dominant toric surfaces $X$, the representative $C_X$ contains no $\torus$-invariant points.  In this case $\pi_{\rzt X}:C\to C_X$ is a homeomorphism, and if $Y$ is another sufficiently dominant surface, the restriction $\pi_{XY}|_{C_X}:C_X\to C_Y$ is an isomorphism.
\end{prop}

\noindent When there are no $\torus$-invariant points in $C_X$, we say $X$ \emph{fully realizes} $C$.  In this case, the intersection number $(C_X\cdot C'_X)_X$ between $C_X$ and the representative of any other curve $C'\subset \rzt$ does not depend on $X$.  We therefore write $(C\cdot C') := (C_X\cdot C'_X)_X$ without the subscripts.

\begin{cor}
\label{cor:balanceformula}
Let $C\subset\rzt$ be an internal curve and $X$ be a toric surface fully realizing $C$.  Then $C$ meets a pole $C_\tau\subset\rzt$ only if $\tau\in\Sigma_1(X)$, and  
$$
\sum_{\tau\in\Sigma_1(X)} (C\cdot C_\tau)\,\nvec_\tau = 0 \in N_\R,
$$
where $\nvec_\tau$ generates $N\cap\tau$.  In particular $C$ meets at least two poles $C_{\tau_1},C_{\tau_2}\subset \rzt$, and if there are no others, then $\tau_2 = -\tau_1$.
\end{cor}

\begin{proof}
Note that $\torus$ contains no compact curves, so $C$ meets at least one external curve in $X$.  Hence all but the displayed formula in the corollary is clear.  The formula itself holds because $(C\cdot D)_X = 0$ for every principal divisor $D\in\div(X)$, and external divisors are principal if and only if they have linear support functions $\sfn_D:N_\R\to\R$.  
\end{proof}

To close this section, we point out that the only curves in $\rzt$ that contain $\torus$-invariant points of $\rzt$ are poles.  Hence $\torus$-invariant points that are not contained even in poles, i.e. those indexed by irrational rays $\tau\subset N_\R$, play very little role in what follows.

\section{Toric maps}
\label{sec:maps}

Here we recall and extend facts about a class of rational self-maps on toric surfaces that was studied at length in \cite{DiLi16}.  Recall from \S~\ref{sec:rztoric} that $\eta$ denotes the holomorphic two form on $\torus$ that restricts to Haar measure on $\torus_\R$.

\begin{defn} A \emph{toric map} is a rational map $f:\torus\tto\torus$ such that  $f^*\twoform = \cov(f)\twoform$ for some constant $\cov(f)\in\C^*$.  We call $\cov(f)$, the \emph{determinant} of $f$.
\end{defn}

Theorem \ref{thm:tropmap} below makes clear that the determinant of a toric map is always an integer and, furthermore, equal to $\pm 1$ if $f$ is birational.  One checks that any monomial map $h_A:\torus\to\torus$ is toric with $\cov(h_A) = \det A$.  A map $f:\torus\to\torus$ is a \emph{translation} if it is given by $f(x) = xy$ for some $y\in \torus$.  A translation is toric with $\cov(f) = 1$ and extends to an automorphism on each toric surface.  We will refer to the (still toric) composition of a translation and a monomial map as a \emph{shifted monomial map}.  Many shifted monomial maps are birational, but it turns out that there are birational toric maps beyond just these.

\begin{eg}
\label{eg:inv}
The map $g:\cp^2\tto\cp^2$ given in affine coordinates by
$$
g(x_1,x_2) = \left(x_1\frac{1-x_1+x_2}{x_1+x_2-1},x_2\frac{1+x_1-x_2}{x_1+x_2-1}\right)
$$
is a toric involution with determinant $\cov(g) = 1$. 
\end{eg}

\noindent See \cite{Bla13} for a simple presentation of the group of birational
toric maps.  A composition of toric maps is toric, so they include the
semi-group generated by monomial maps and birational toric maps.  We do not
know whether there are any toric maps outside this semigroup.  

If $f:\torus\tto\torus$ is toric and $X$ and $Y$ are toric surfaces, then we let $f_{XY}:X\tto Y$ denote the unique extension of $f$ to a rational map between $X$ and $Y$. If $X'\succ X$ and $Y'\succ Y$ are other toric surfaces, then $f_{X'Y'}$ and $f_{XY}$ are compatible with transitions; i.e. 
\begin{equation}
\label{eqn:commutes}
f_{XY}\circ\pi_{X'X} = \pi_{Y'Y}\circ f_{X'Y'},
\end{equation}
on the Zariski open subset of $X'$ where both sides are defined.  
We therefore obtain a partially defined `rational' map $\rzf:\rzt\tto\rzt$.

Specifically, let $p\in\rzt$ be a point such that for any toric surface $Y$ there exists another toric surface $X$ such that $p_X\notin\ind(f_{XY})$.  Then we let $\rzf(p)\in\rzt$ be the point for which 
$\rzf(p)_Y = f_{XY}(p_X)$.  Compatibility guarantees that this is independent of the choice of $X$.  The map $\rzf$ acts similarly on curves $C\subset\rzt$.  That is, for any toric surface $Y$, we set $\rzf(C)_Y = \rzf(C_X)$ where $X$ is a toric surface sufficiently dominant that $C_X$ is a curve.  Indeed the restrictions $f_{XY}|_{C_X}$ also give us a holomorphic restriction $\rzf|_C:C\to\rzf(C)$.
The actions of $\rzf$ on curves and points are consistent in the sense that if $C\subset\rzt$ is a curve and $p\in C$ is a point at which $\rzf$ is well-defined, then $\rzf(p) = (\rzf|_C)(p)\in \rzf(C)$.

The assumption that $f$ is toric allows us to say much more about $\rzf$.  We recall the following collection of observations from \cite{DiLi16}.

\begin{thm}
\label{thm:janliandme1}
Let $f:\torus\tto\torus$ be a toric rational map and $X$ and $Y$ be toric surfaces.
\begin{enumerate}
 \item For each internal curve $C\subset X$, either $f_{XY}(C)$ is also an internal curve and $f_{XY}$ is unramified about $C$; or $C\subset\exc(f_{XY})$ and $f_{XY}(C)$ is an external point of $Y$.
 \item For each pole $C_\tau\subset X$, either $f_{XY}(C_\tau)$ is a pole of $Y$ or $C_\tau\subset \exc(f_{XY})$ and $f_{XY}(C_\tau)$ is a $\torus$-invariant point of $Y$.  
 \item For sufficiently dominant $Y'\succ Y$, all curves in $\exc(f_{XY'})$ are internal, and the image of each lies in ${Y'}^\circ$.
\end{enumerate}
\end{thm}

It follows from the first conclusion of this theorem that the collection of internal exceptional curves of $f_{XY}$ is independent of $X$ and $Y$.  We let $\exc(\rzf)\subset\rzt$ denote the union of these and call each of them \emph{persistently} exceptional for $\rzf$.  Theorem \ref{thm:janliandme1} and our definition of $\rzf:\rzt\tto\rzt$ directly yield the following.

\begin{cor}
\label{cor:curveimages}
Let $f:\torus\tto\torus$ be a toric rational map.  Then $\exc(\rzf)$ consists of finitely many internal curves
and the image of each is a point in $\rzt^\circ\setminus \torus$.  If $C\not\subset \exc(\rzf)$ is some other curve in $\rzt$, then $\rzf(C)$ is also a curve, and $\rzf(C)$ is internal (respectively, a pole) if and only if $C$ is.
\end{cor}

Focusing on images of points rather than curves, we also have the following from \cite{DiLi16}.

\begin{thm}
\label{thm:janliandme2}
Let $f:\torus\tto\torus$ be toric and $X$ and $Y$ be toric surfaces and $p\in X$ be a point.
\begin{itemize}
 \item If $p\notin \exc(f_{XY}) \cup \ind(f_{XY})$ is internal, then $f_{XY}(p)$ is internal, and $f_{XY}$ is a local isomorphism about $p$.
 \item If $p\notin \exc(f_{XY})\cup\ind(f_{XY})$ is external, then so is $f_{XY}(p)$, and $p$ is $\torus$-invariant if and only if $f_{XY}(p)$ is.
 \item If $p\in \ind(f_{XY})$ is not $\torus$-invariant, then $f_{XY}(p)$ is a finite union of internal curves.
 \item For sufficiently dominant toric surfaces $X'\succ X$, no points of $\ind(f_{X'Y})$ are $\torus$-invariant.
\end{itemize}
\end{thm}

It follows from the last two conclusions of this theorem that $\rzf$ is well-defined and continuous off a finite subset $\ind(\rzf) \subset\rzt^\circ$, whose elements we call \emph{persistently indeterminate points} for $\rzf$.
We define the image of each $p\in\ind(\rzf)$ by $\rzf(p) = f_{X'Y}(p)$ for $X'$ sufficiently dominant.

\begin{cor}
\label{cor:ptimages}
Let $f:\torus\tto\torus$ be toric.  Then for each $p\in\ind(\rzf)$, the image $\rzf(p)$ is a finite union of internal curves.  If $p\in \rzt\setminus(\ind(\rzf)\cup\exc(\rzf))$, then $p$ is internal (or external or $\torus$-invariant) if and only if $\rzf(p)$ is, and in the case where $p$ is internal $\rzf$ is a local isomorphism about $p$.
\end{cor}

The following result will be used below to prove Proposition \ref{PROP:TROPICAL_APPROXIMATION}

\begin{lem}\label{LEM:IMAGES_INDET_PTS} For any toric map $f$, we have $\ind(\rzf)\subset \exc(\rzf)$.  If, moreover, $\alpha \in \torus\cap\ind(\rzf)$ is internal, then
\begin{align}\label{EQN:IMAGE_OF_INDET_PT_IN_TORUS}
\rzf(\alpha) \setminus \torus \subset \rzf(\exc(\rzf)).
\end{align}
If on the other hand $\alpha \in\ind(\rzf)\setminus \torus$ is external, contained in the pole $C_\tau$, then
\begin{align}\label{EQN:IMAGE_OF_INDET_PT_ON_POLE}
\rzf(\alpha) \setminus \torus \subset \rzf(\exc(\rzf)) \cup \rzf|_{C_{\tau}}(\alpha).
\end{align}
\end{lem}

\begin{proof}
Fix a toric surface $X$ sufficiently dominant that $X$ realizes $\ind(\rzf)$ and fully realizes each curve in $\rzf(\ind(\rzf))$.  Projection $\pi_{\rzt X}:\rzt \to X$ restricts to a homeomorphism on small neighborhoods of $\alpha$ and $\rzf(\alpha)$, so we can conflate $\rzf$ and $f_{XX}$: if $\Gamma\subset X\times X$ is the graph of $f_{XX}$ and $\pi_1,\pi_2:\Gamma\to X$ are projections onto domain and range, then $\rzf = f_{XX} = \pi_2\circ\pi_1^{-1}$ on a neighborhood $U\ni\alpha$ chosen small enough that $\ind(\rzf)\cap U = \{\alpha\}$ and that any persistently exceptional curve that meets $U$ also contains $\alpha$.  Let $\beta \in \rzf(\alpha)\setminus\torus$ be an external point.  

Supposing first that $\alpha\in\ind(\rzf)\cap\torus$, we may assume $U\subset\torus$.  We have   
$$
\beta \in \rzf(\alpha) = \pi_2(\pi_1^{-1}(\alpha)) \subset \pi_2(\pi_1^{-1}(U)).
$$
But since $U\subset\torus$ and $f$ is toric, we have that $\pi_2(\pi_1^{-1}(U))\setminus \torus \subset \rzf(\exc(f)) \cup \rzf(\ind(\rzf)) \setminus\torus$
is finite.  In particular $\beta$ does not lie in the interior of $\pi_2(\pi_1^{-1}(U))$.  As $\pi_1^{-1}(U)$ is open in $\Gamma$, it follows that there is curve $E\subset\Gamma$ that meets but is not entirely contained in $\pi_1^{-1}(U)$ such that $\pi_2$ contracts $E$ to $\beta$.  It follows that $\pi_1(E)\subset X$ is internal and contracted by $\rzf$ to $\beta$.  So $\beta = \rzf(\exc(\rzf))$, i.e. \eqref{EQN:IMAGE_OF_INDET_PT_IN_TORUS} holds.  And since $\pi_1(E)$ meets $U$, we have $\alpha \in E\subset\exc(\rzf)$.
 
Now suppose $\alpha\in \ind(\rzf)\setminus\torus$ is external, contained in the pole $C_\tau$.  We can assume in this case that $U\setminus\torus \subset C_\tau^\circ$ and that $\alpha$ is the only possible preimage in $U$ of $\beta$ by $f|_{C_\tau}$.  The argument from the previous paragraph shows again that if $\beta\in \pi_2(\pi_1^{-1}(U))$ is not an interior point, then $\beta\in \rzf(E)\cap C^\circ_{\tau'}$ for some persistently exceptional curve $E$ containing $\alpha$ and some pole $C_{\tau'}\subset\rzt$.  If, on the other hand, $\beta$ is an interior point of $\pi_2(\pi_1^{-1}(U))$, it contains a relative neighborhood $W$ of $\beta$ in $C_{\tau'}$.  Since $f$ is toric, there is finite subset $S\subset W$ such that $\rzf^{-1}(W\setminus S)\cap\torus = \emptyset$, and therefore $\rzf^{-1}(W\setminus S)\subset C^*_{\tau}$.  Hence $A_f(\tau) = \tau'$ and $W \subset \rzf|_{C_\tau}(U\cap C_\tau)$.  Necessarily then $\beta = f|_{C_\tau}(p)$ for some $p\in U\cap C_\tau$, and by our choice of $U$, we conclude that $p = \alpha$.  Hence \eqref{EQN:IMAGE_OF_INDET_PT_ON_POLE} holds.  In particular, by Corollary \ref{cor:balanceformula} we see that there is at least one point of $\rzf(\alpha)$ different from $f|_{C_\tau}(\alpha)$.  So again $\alpha\in \exc(\rzf)$
\end{proof}

Let $C\subset \rzt$ be a curve that is not persistently exceptional for $\rzf$.  Compatibility with transition maps implies that the \emph{ramification} $\ram(\rzf,C) := \ram(f_{XY},C)\in\Z_{>0}$ of $f_{XY}$ about $C$ is the same for all toric surfaces $X$ and $Y$ that realize $C$ and $\rzf(C)$.  Corollary \ref{cor:ptimages} further implies that $\ram(\rzf,C) = 1$ for internal $C$.  One can say much more about ramification and other aspects of the behavior of $\rzf$ on poles.

\begin{thm}[See \cite{DiLi16}, Theorem 6.18]
\label{thm:tropmap}
If $f:\torus\tto\torus$ is a toric map, then there exists a finite set $\Sigma_1(f)$ of rational rays $\tau\in N_\R$ and a continuous self-map $A_f:N_\R\to N_\R$ with the following properties.
\begin{enumerate}
\item $A_f$ is `integral', i.e. $A_f(N)\subset N$.
\item If $\sigma\subset N_\R$ is a sector that omits all $\tau\in\Sigma_1(f)$, the restriction $A_f|_\sigma$ is linear with $\det A_f|_\sigma = \pm\rho(f)$.
\item For any curve $C_\tau\subset \rzt\setminus\torus$, we have $\rzf(C_{\tau}) = C_{A_f(\tau)}$.
\item If $\nvec,\nvec'\in N$ are the primitive vectors generating the rational rays $\tau, A_f(\tau) \subset N_\R$, then the ramification of $\rzf$ about the pole $C_\tau$ is given by
$$
\ram(\rzf,C_{\tau}) = \frac{\norm{A_f(\nvec)}}{\norm{\nvec'}},
$$
and $\rzf:C^\circ_\tau\to C^\circ_{A_f(\tau)}$ is a covering of degree $|\rho(f)|/\ram(\rzf,C_\tau)$.  Hence $\rzf$ has local topological degree $\rho(f)$ on a neighborhood of any pole.
\item $A_f:N_\R\setminus\{0\}\to N_\R\setminus\{0\}$ is a covering map with degree $\dtop(f)/|\rho(f)|$.
\end{enumerate}
\end{thm}

We will call the map $A_f$ in Theorem \ref{thm:tropmap} the \emph{tropicalization} of the toric map $f$.  Tropicalization is functorial, i.e. if $g:\torus\tto\torus$ is also toric, then $A_{f\circ g} = A_f\circ A_g$.

\begin{proof}
Let $X$ and $Y$ be toric surfaces, with $X$ sufficiently dominant that $\ind(f_{XY}) = \ind(\rzf)$ contains no $\torus$-invariant points, and $X$ fully realizes each persistently exceptional curve.
  
Fix a sector $\sigma\in\Sigma_2(X)$ bounded by rays $\tau_1,\tau_2\in\Sigma_1(X)$.  Then $f_{XY}$ is holomorphic about $p_\sigma$ with $f_{XY}(p_\sigma) = p_{\sigma'}$ for some sector $\sigma'\in \Sigma_2(Y)$.  Our setup also ensures for each $j=1,2$ that $f_{XY}$ either contracts $C_{\tau_j}$ to $p_{\sigma'}$ or has image $f_{XY}(C_{\tau_j})$ equal to one of the poles in $Y$ that contain $p_{\sigma'}$.  In either case, we may assume that the rays $\tau'_1,\tau'_2$ bounding $\sigma'$ are ordered so that $f_{XY}(C_{\tau_j}) \subset C_{\tau_j'}$.  It follows that if we work in distinguished coordinates about $p_\sigma\in X$ and $p_\sigma'\in Y$, then $f_{XY}$ is given by
\begin{equation}
\label{eqn:localcoordmap}
(y_1,y_2) = f_{XY}(x_1,x_2) = (x_1^a x_2^b f_1(x_1,x_2),x_1^cx_2^d f_2(x_1,x_2))
\end{equation}
where $a,d>0$ and $b,c\geq 0$ are integers and $f_1,f_2$ are rational functions that are holomorphic near $(0,0)$ and do not vanish identically on either axis $\{x_j=0\}$.  Since $p_\sigma$ is not contained in a persistently exceptional curve, it further follows that $f_1,f_2$ are non-vanishing on a neighborhood of $(0,0)$.  The two form $\eta$ has the same expression in any distinguished coordinate system, so we find after some further computation, that the equation $f^*\eta = \rho(f)\eta$ becomes
$$
(ad-bc + e(x_1,x_2))\frac{dx_1\wedge dx_2}{x_1x_2} = \pm\rho(f)\frac{dx_1\wedge dx_2}{x_1x_2},
$$
where $e(x_1,x_2)$ is a rational function that vanishes at $(0,0)$.  Hence $e$ vanishes identically, and we see that $ad-bc = \pm\rho(f)$.  

Now if $\nvec_1,\nvec_2\in N$ are the primitive vectors generating $\tau_1,\tau_2$, then any ray $\tau\subset\overline\sigma$ is generated by a primitive vector $\nvec = \alpha_1\nvec_1 + \alpha_2\nvec_2$ for non-negative integers $\alpha_1,\alpha_2$.  That is $C_\tau$ corresponds to the one parameter subgroup $\gamma_\nvec:t\mapsto t^\nvec := (t^{\alpha_1},t^{\alpha_2})$ of $(\C^*)^2$.  A little further computation shows that
$$
f(\gamma(\nvec)) = m(t)t^{A\nvec}
$$
where $A=\begin{pmatrix} a & b \\ c & d \end{pmatrix}$ and $m$ is holomorphic and non-vanishing near $0$.  As was observed in the appendix of \cite{DiLi16}, this implies that $f(C_\tau) = C_{A\tau}$.  Hence $A_f$ is given on $\sigma$ by a linear transformation with integer matrix $A$ satisfying $\det A =\pm\rho(f)$.  The first three conclusions follow with $\Sigma_1(f) := \Sigma_1(X)$.  Note that continuity of $A_f$ follows from the fact that the finitely many closed sectors $\overline\sigma$ bounded by rays in $\Sigma_1(X)$ cover $N_\R$.

To prove statement (4), let $C_{\tau'} = f(C_\tau)$.  We refine our initial setup, requiring in addition that $\tau\in\Sigma_1(X)$ and $\tau'\in\Sigma_1(Y)$.  We choose $\sigma\in\Sigma_2$ so that $\tau_1 = \tau$ is a bounding ray for $\sigma$ and again let $p_{\sigma'} = f_{XY}(p_\sigma)$.  Then $\tau'$ must be a bounding ray for $\sigma'$, and we index so that $\tau' = \tau_1'$.  In this case, we have in distinguished coordinates that $f_{XY}(\{x_1=0\} = f_{XY}(C_{\tau_1}) = C_{\tau_1'} = \{y_1=0\}$.  Hence in \eqref{eqn:localcoordmap}, the exponent $c$ vanishes, and one sees that 
$$
|\rho(f)| = |\det A| = |ad| = \ram(f_{XY},\{x_1=0\})\cdot \deg(f_{XY}|_{x_1=0}) = \ram(\rzf,C_{\tau_1})\cdot \deg(\rzf|_{C_{\tau_1}}).
$$  
Statement (4) follows from this and the fact that $(0,1)$ is the primitive generator for the rays indexing both $\{x_1=0\}$ and $\{y_1=0\}$.

Since general points in $X$ have $\dtop(f)$ preimages, and the local topological degree about any pole is $|\rho(f)|$, we see that every pole in $\rzt$ has exactly $\dtop(f)/|\rho(f)|$ preimages under $\rzf$.  That is, every rational ray in $N_\R$ has exactly $\dtop/|\rho(f)|$ preimages under $A_f$.  Linearity of $A_f$ fails only about rational rays, so we conclude that every irrational ray in $N_\R$ has $\dtop(f)/|\rho(f)|$ preimages, too.  Hence $A_f:N_\R\setminus\{0\}\to N_\R\setminus\{0\}$ is a covering of degree $\dtop(f)/|\rho(f)|$.
\end{proof}

Before proceeding we make a further observation concerning the proof of Theorem \ref{thm:tropmap}.  Namely, in \eqref{eqn:localcoordmap}, the functions $f_1,f_2$ have no poles or zeroes on the axes $\{x_j=0\}$, so if no pole or zero of $f_j$ meets $(\C^*)^2$, then $f_j$ is constant.  Such poles or zeroes are precisely the internal curves of $\exc(f_{XY})$, i.e. they constitute the persistent exceptional set of the induced map $\rzf:\rzt\tto\rzt$.  This gives us the following further conclusion.

\begin{cor}
\label{cor:noexceptions}
If $f:\torus\tto\torus$ is toric, then $\rzf:\rzt\tto\rzt$ has no persistently exceptional curves if and only if it is a shifted monomial map.
\end{cor}

Let us return to the examples introduced at the beginning of this section, identifying persistently exceptional curves, indeterminate points and tropicalizations for each of them.  Monomial maps $h_A:\torus\to\torus$ are self-coverings of $\torus$.  Hence $\exc(\hat h_A) = \ind(\hat h_A) = \emptyset$, and one sees directly from the arguments in the proof of Theorem \ref{thm:tropmap} that the tropicalization of $h_A$ is $A_{h_A} = A$.   

One checks that the internal, and therefore persistently exceptional curves, of the birational toric map $g:\torus\to\torus$ in Example \ref{eg:inv} are the lines joining the points $[1,1,0],[1,0,1],[0,1,1]\in\cp^2$ and that, as a self-map of $\cp^2$, $g$ sends each of these lines to the intersection of the other two.  As $g$ is an involution, it follows that $\ind(\hat g) = \{[1,1,0],[0,1,0],[0,0,1]\}$.  One further checks that in distinguished coordinates $(x_1,x_2)$ about any of the three $\torus$-invariant points $[1,0,0],[0,1,0],[0,0,1]\in\cp^2$, we have $g(x_1,x_2) = (x_1f_1,x_2f_2)$, where the $f_j$ are holomorphic near $(0,0)$.  Hence $A_g = \id$.  These facts were all observed in \cite{DiLi16} and \cite{BDJ20}. 

Compositions of the form $g\circ h_A$ furnished the central examples in \cite{BDJ20}, where it was shown that the first dynamical degree of a rational surface map can be transcendental.  So we revisit these here.

\begin{eg}
\label{eg:main}
Let $f := g\circ h_A:\torus\tto\torus$, where $g$ is the involution in Example \ref{eg:inv} and $h_A$ is monomial.  Let $\rzf:\rzt\tto \rzt$ be the induced map. Then
\begin{itemize}
 \item $A_f = A$;
 \item $\ind(\rzf) = h_A^{-1}(\ind(\hat g))$ consists of three free points, one in each pole $C_{\tau}$, $A\tau\in\Sigma_1(\cp^2)$, and $\rzf_A(\ind(\rzf)) = \exc(\hat g)$;
 \item $\exc(\rzf) = \hat h_A^{-1}(\exc(\hat g))$ consists of three internal curves, is compatible with any toric surface containing $\ind(\rzf)$, and has image $\rzf(\exc(\rzf)) = \hat g(\exc(\hat g)) = \ind(\hat g)$;
\end{itemize}
\end{eg}

\begin{defn}
A toric map $f:\torus\tto \torus$ is \emph{internally stable} if the induced map $\rzf:\rzt\tto\rzt$ satisfies $\rzf^n(\exc(\rzf))\cap \ind(\rzf) = \emptyset$ for all $n\in\N$.
\end{defn}

It follows from definitions that if $f$ is internally stable then
$$
\ind(\rzf^n) = \bigcup_{j=0}^{n-1} \rzf^{-j}(\ind(\rzf))\quad\text{and}\quad
\exc(\rzf^n) = \bigcup_{j=0}^{n-1} \rzf^{-j}(\exc(\rzf)).
$$

One way to verify internal stability is to show (if possible) that if $C_\tau\subset\rzt$ is a pole containing the image of a persistently exceptional curve, then no pole in its forward orbit $\rzf^n(C_\tau)$, $n\geq 0$ contains a persistently indeterminate point.  This works in particular for the maps from \cite{BDJ20}.

\begin{cor}
Suppose the map $f$ in Example \ref{eg:main} has matrix $A=\begin{pmatrix} \re\xi & -\im\xi \\ \im\xi & \re\xi \end{pmatrix}$ corresponding to multiplication by a Gaussian integer $\xi = |\xi|e^{2\pi i \theta}$ with $\theta\notin\Q$.  Then $f$ is internally stable on $\rzt$. 
\end{cor}

\begin{proof}
Since $\theta$ is irrational and all three rays in $\Sigma_1(\cp^2)$ meet the horizontal axis at rational angles, we have for any $\tau,\tau'\in \Sigma_1(\cp^2)$ that $A^n\tau \neq A^m\tau'$ for any $n\neq m\in\Z$.  Hence the stability assertion follows from the fact that 
$\ind(\rzf)$ consists of free points in poles $C_{A^{-1}\tau}$, $\tau\in\Sigma_1(\cp^2)$, whereas for each $n\geq 1$, $\rzf^n(\exc(\rzf))$ consists of free points in poles $C_{A^{n-1}\tau}$, $\tau\subset\Sigma_1(\cp^2)$.
\end{proof}

The next result allows us to use the tropicalization of a toric map $f$ to bound the rate at which orbits $(f^n(p))$ escape $\torus$.

\begin{prop}\label{PROP:TROPICAL_APPROXIMATION}
If $f:\torus\tto \torus$ is toric, then for any neighborhood $U$ of $\rzf(\exc(\rzf))$, there exists $R > 0$ such that 
\begin{align}\label{EQN:DESIRED_CONCLUSION}
\|\trop\circ f(p) \| \leq \|A_f \circ \trop(p) \| + R \quad \mbox{or} \quad  f(p) \in U
\end{align}
for each $p\in\torus\setminus\exc(\rzf)$.  Similarly, for any neighborhood $V$ of $\ind(\rzf)\setminus\torus$, there exists $R>0$ such that 
\begin{equation}\label{EQN:DESIRED_CONCLUSION2}
\|\trop\circ f(p)\| \geq \|A_f\circ\trop(p)\| - R\quad\mbox{or}\quad p\in V.
\end{equation}
\end{prop}

\begin{proof}
Note that by Corollary \ref{cor:ptimages} and Lemma \ref{LEM:IMAGES_INDET_PTS}, the restriction $p\in\torus\setminus\exc(\rzf)$ implies that $\rzf(p) = f(p)\in\torus$.
In order to compare the behavior of $f$ with that of $A_f$, we choose a toric surface $Y$ sufficiently dominant that it realizes all points in $\ind(\rzf)$ and $\rzf(\exc(\rzf))$ and fully realizes all curves in $\exc(\rzf)$ and $\rzf(\ind(\rzf))$.  Then we choose $X\succ Y$ so that, additionally, $\ind(f_{XY}) = \ind(\rzf)$, i.e. so that no $\torus$-invariant point in $X$ is indeterminate for $f_{XY}$.  

Let $\sigma\subset \Sigma_2(X)$ be any sector and $\tau_1,\tau_2\in\Sigma_1(X)$ be its bounding rays.  Let $(x_1,x_2)$ be distinguished coordinates about $p_\sigma$ such that $x_j$ vanishes along $C_{\tau_j}$, $j=1,2$.  Since $p_\sigma\notin \ind(f_{XY})$, we have
$A_f(\sigma) \subset \sigma'$ for some $\sigma'\in\Sigma_2(Y)$.  Therefore if $(y_1,y_2)$ are distinguished coordinates for $Y$ about $p_{\sigma'}$, we have as in the proof of Theorem \ref{thm:tropmap} that
$$
(y_1,y_2) = f_{XY}(x_1,x_2) = (x_1^a x_2^b f_1(x_1,x_2),x_1^cx_2^d f_2(x_1,x_2))
$$
where $a,d>0$ and $b,c\geq 0$ are integers and $f_1,f_2:\C^2\tto\C$ are rational functions
that are holomorphic near $(0,0)$ and do not vanish identically on either axis
$\{x_j=0\}$.  If $\nvec_j\subset\tau_j$, $j=1,2$ generate $\tau_j\cap N$ and $\nvec_j'\subset\tau_j'$ generate $\tau_j'\cap N$, then $A_f:\overline{\sigma}\to\overline{\sigma'}$ is given as before by the integer matrix 
$A=\begin{pmatrix} a & b \\ c & d\end{pmatrix}$ relative to the bases $\nvec_1,\nvec_2$ and $\nvec'_1,\nvec'_2$.   From \eqref{eqn:twoform}, one therefore computes for $x = (x_1,x_2)\in\C^2$ that
\begin{align}\label{EQN:GOOD_TROP_APPROXIMATION1}
\|\trop\circ f(x) -A_f \circ \trop(x) \| = \| - \log|f_1(x)| \nvec_1'    - \log|f_2(x)| \nvec_2' \|.
\end{align}
The zeros and poles of $f_1$ and $f_2$ are just the persistently exceptional curves of $\rzf$.
Hence if $\Omega$ is any neighborhood of $\exc(\rzf)$, and $P\subset X$ is any bounded coordinate polydisk about $(0,0)$, we have that both \eqref{EQN:DESIRED_CONCLUSION} and \eqref{EQN:DESIRED_CONCLUSION2} hold for some $R>0$ on $P\cap\torus\setminus{\Omega}$.  Indeed $X$ is covered by finitely many such polydisks, so both equations hold on all of $\torus\setminus\Omega$. 

Since $\rzf(\exc(\rzf)) \cap\torus = \emptyset$ and $\rzf$ is continuous on $\exc(\rzf)\setminus\ind(\rzf)$, we have $\lim_{p\to q} \norm{\trop\circ f(p)} = \infty$ uniformly for all $q$ in any compact subset of $\exc(\rzf)\setminus\ind(\rzf)$.  So \eqref{EQN:DESIRED_CONCLUSION2} holds for some $R = R(V) > 0$ outside any given neighborhood $V\subset\rzt$ of $\ind(\rzf)$, as asserted.  

Completing the proof of \eqref{EQN:DESIRED_CONCLUSION} is trickier.  Since $\rzf^{-1}(U) = f_{XY}^{-1}(U)$ contains a neighborhood of $\exc(\rzf)\setminus\ind(\rzf)$, it suffices to check the desired bound on a small enough neighborhood $U_\alpha$ of each of the finitely many points $\alpha \in\ind(\rzf)$.  When $\alpha\in\torus$ is internal, we can arrange \eqref{EQN:DESIRED_CONCLUSION} trivially, by setting $R>\trop(\alpha)$ and choosing a neighborhood whose points all satisfy $\trop(p) < R$.  So we can assume 
$\alpha\in\ind(\rzf)$ is external.

Revisiting our choices of distinguished coordinates above, we pick $\sigma\in\Sigma_2(X)$ and $\tau_j\in \Sigma_1(X)$ so that $C_{\tau_1}$ is the unique pole containing $\alpha$, i.e. $\alpha = (0,x_2)$ for some $x_2\in\C^*$.  Since $Y$ fully realizes $\rzf(\alpha)$, we have that $A_f(\tau_1)\in\Sigma_1(Y)$ is one of the rays, say $\tau_1'$, bounding $\sigma_1'$.  Hence in coordinates $\beta:= \rzf|_{C_{\tau_1}}(\alpha) = f_{XY}(0,x_2) = (0,y_2)$.  Given $\epsilon>0$, let $B_\beta(\epsilon)\subset Y$ be the open distinguished coordinate ball about $\beta$.  Then by Lemma \ref{LEM:IMAGES_INDET_PTS}, there exists $\delta>0$ and $R>0$ such that if $B_\alpha(\delta)\subset X$ is the corresponding ball about $\alpha$, then 
$$
f_{XY}(B_\alpha(\delta))\cap\{\trop < R\}\subset U\cup B_\beta(\epsilon),
$$ 
where $U$ is the open set in \eqref{EQN:DESIRED_CONCLUSION}.  So it suffices to verify the bound in \eqref{EQN:DESIRED_CONCLUSION} only for $p\in W := B_\alpha(\epsilon)\cap f_{XY}^{-1}(B_\beta(\epsilon))$.  Since $\beta\notin \rzf(\exc(\rzf))$, we can assume $\epsilon$ is small enough that $W\setminus\{\alpha\}$ is disjoint from all persistently exceptional curves. And as $\beta = (0,y_2)$ with $y_2 \neq 0$, we can further assume that $|f_1|$, $|f_2|$ and $|1/f_2|$ are all bounded above on $W\setminus\{\alpha\}$.   To conclude, we must show that $|1/f_1|$ is likewise bounded on $W\setminus\{\alpha\}$.

Let $\pi:\tilde X \to X$ be a birational morphism that resolves the indeterminacy of $f_{XY}$ at $\alpha$, lifting $f_{XY}$ to a holomorphic map $\tilde f:\pi^{-1}(B_\alpha(\delta))\to Y$.  Hence $\tilde W := \pi^{-1}(B_\alpha(\delta))\cap \tilde f^{-1}(B_\beta(\epsilon))$ is an open set such that $\pi(\tilde W) = W$.  The corresponding lift $\tilde f_1:\pi^{-1}(B_\alpha(\epsilon))\to\cp^1$ of $f_1$ is then well-defined everywhere, so its zeros and poles are disjoint.  As $W\setminus\{\alpha\}$ avoids persistently exceptional curves, the set $\{\tilde f_1 = 0\}\cap\tilde W$ is compact, equal to a finite union of curves in $\pi^{-1}(\alpha)$.  But the intersection form on $\pi^{-1}(\alpha)$ is negative definite whereas a compact principle divisor has vanishing self-intersection, so $\{f_1=0\}\cap W$ must be empty.  Shrinking $\epsilon$ further therefore guarantees that $1/|\tilde f_1|$ is bounded on $\tilde W$.  Hence $1/|f_1|$ is bounded on $W\setminus\{\alpha\}$, and the proof of \eqref{EQN:DESIRED_CONCLUSION} is complete.
\end{proof}

\section{Irrational rotation and volume}
\label{sec:volumes}

We saw in Theorem \ref{thm:tropmap} that the tropicalization $A_f:N_\R\to N_\R$ of a toric map $f:\torus\tto\torus$ always restricts to a self-cover of the punctured plane $N_\R\setminus\{0\}$ with degree $\frac{\dtop(f)}{|\cov(f)|}$.  In the semigroup generated by birational toric and monomial maps (i.e. in all examples that we know) it always happens that $\cov(f) = \pm\dtop(f)$, i.e. that $A_f$ is actually a homeomorphism.  In this case, we obtain an induced homeomorphism $\tilde A_f:\nvec \mapsto \frac{A_f(\nvec)}{\norm{A_f(\nvec)}}$ of the unit circle $\{\norm{\nvec}=1\}$.  Abusing terminology slightly, we refer to the rotation number of $\tilde A_f$ as the rotation number of $A_f$.  We recall the following result from \cite{DiLi16}.

\begin{thm}[\cite{DiLi16}, Theorem F]
\label{thm:rotationcriterion}
Suppose $f:\torus\tto\torus$ is a toric surface map, and that the
tropicalization $A_f$ of $f$ is a homeomorphism.  Then some iterate of $f$ is birationally
conjugate to an algebraically stable rational map $f_X:X\tto X$ on some (not necessarily toric)
rational surface $X$ if and only if the rotation number of $A_f$ is rational.

\end{thm}

In particular, when $A_f$ is a homeomorphism with rational rotation number, the
equidistribution results Corollaries 2.11 and 3.5 from \cite{DDG10} imply that (after replacing $f$ by an iterate)
normalized forward and backward images of curves are asymptotic to
$f$-invariant currents $T^*$ and $T_*$ on $X$.
Theorems~\ref{thm:mainA} and \ref{thm:mainD} therefore address
the complementary case in which no equidistribution has been previously
established.   Here we take a key step toward proving these theorems by showing
that iterates of a toric surface map that satisfies the hypothesis of Theorem \ref{thm:mainA}
cannot shrink or expand volume too quickly.    
It will be helpful to briefly consider a broader context and to introduce some ad hoc terminology.  

\begin{defn}
Let $X$ be any complex projective surface endowed with a smooth volume form,  and $R:X\tto X$ be a dominant rational map.  Given $\mu\geq 1$, we say that
\begin{itemize}
 \item $R$ has {\em lower volume exponent} $\mu\geq 1$ if there exists $a>0$ such that $\vol R(S) \geq  (a \vol S)^\mu$ for any measurable $S \subset X$;
 \item $R$ has {\em dynamical lower volume exponent} $\mu\geq 1$ if there exist $a,b > 0$ such that  
\begin{align}
 \vol R^n(S) \geq (a \Vol S)^{b \mu^n}.
\end{align}
for any $n\geq 0$ and measurable $S\subset X$.
\end{itemize}
Similarly, $R$ has \em{upper volume exponent} $\mu\leq 1$ if there exists $a > 0$ such that $\vol R(S) \leq  a (\vol S)^\mu$ and \em{dynamical upper volume exponent} $\mu\leq 1$ if there exists $a,b> 0$ such that $\vol R^n(S) \leq a(\vol S)^{b\mu^n}$ for any measurable $S\subset X$.
\end{defn}

The definitions of dynamical upper and lower volume exponents are tailored to our needs.  The role of the constant $a$ is therefore not very symmetric between them.  Note that upper volume exponents are \emph{smaller} than lower exponents, and in no case does a volume exponent need to be optimal.  Since different volume forms on $X$ are uniformly multiplicatively comparable, the choice of volume form does not affect which $\mu$ are (dynamical) upper or lower volume exponents.

\begin{prop}\label{PROP:BASIC_PROPERTIES_VOL_EXPONENTS}
Let $X$ be a complex projective surface with a smooth volume form and $R: X \dashrightarrow X$ be a dominant rational mapping.  Then
\begin{itemize}
\item[(i)] $R$ has both an upper and a lower volume exponent;
\item[(ii)] $R: X \dashrightarrow X$ has dynamical lower (respectively, upper) volume exponent $\mu$ if and only if some iterate $R^k: X \dashrightarrow X$ has dynamical lower (respectively, upper) volume exponent $\mu^k$.
\item[(iii)]  If $R_1: X \dashrightarrow X$ has dynamical lower (respectively, upper) volume exponent
$\mu > 0$ and $R_2: Y \dashrightarrow Y$ is birationally conjugate to $R_1$,
then $R_2$ also has dynamical lower (respectively, upper) volume exponent $\mu$.
\end{itemize}
\end{prop}

\begin{proof}
Let $\pi:\tilde X\to X$ be a birational morphism lifting $R$ to a holomorphic map $\tilde R:\tilde X\to X$.  Then $\vol\pi(\tilde S) \leq C\vol(\tilde S)$ for any measurable $\tilde S\subset \tilde X$.  Hence, existence of a lower volume exponent for $R$ follows from applying well-known estimates for holomorphic maps $\tilde R$. See, e.g. \cite[Section 5]{BLR},
\cite[Sections 6,7]{FAVRE_JONSSON}) for more details.  Reversing the roles of $\pi$ and $\tilde R$ gives an upper volume exponent.
We establish the two remaining conclusions only for dynamical lower volume exponents since the arguments for upper exponents are similar.

If $R: X \dashrightarrow X$ has dynamical lower volume exponent $\mu \geq 1$, then by definition $R^k$ has dynamical volume exponent $\mu^k$ for any $k \geq 1$.   
Conversely, suppose $R^k$ has dynamical lower volume exponent $\mu^k$.  I.e. for some $a,b> 0$, any measurable $S \subset X$ and any $m \geq 0$,
\begin{align*}
\vol R^{mk}(S) \geq (a \Vol S)^{b \mu^{mk}}.
\end{align*}
Having already proved Part (i), we apply it to find 
$\gamma \geq 1$ and $a' > 0$ such that for any measurable $S \subset X$ we have $\Vol(R^r(S)) \geq (a' \Vol S)^{\gamma}$ for any $0 \leq r < k$. 
Given any integer $n = km+r\geq 0$ with $0 \leq r < k$, we then have
\begin{align}
\vol R^n(S) \geq (a \Vol R^r(S))^{b \mu^{mk}} \geq (a(a' \Vol S)^\gamma)^{b \mu^{mk}} \leq (a'' \Vol S)^{B\mu^n},
\end{align}
where $B = b\gamma$ and $a'' = a^{1/\gamma} a'$.  So $\mu$ is a dynamical lower volume exponent for $R$.

Suppose finally that $\psi : X \dashrightarrow Y$ is birational and conjugates
$R_1$ to $R_2$.  Suppose $\mu$ is a dynamical lower volume exponent for $R_1$ with
corresponding constants $a,b > 0$ and let $\gamma_1,\gamma_2$ be lower volume
exponents for $\psi$ and $\psi^{-1}$ with multiplicative constants $a_1, a_2 >
0$, respectively.   They exist because we have already proved Part (i).  Then,
$$
\vol R_2^n(S) \geq  \vol(\psi(R_1^n(\psi^{-1}(S)))) \geq (a_1 (a (a_2 \Vol S)^{\gamma_2})^{b \mu^n})^{\gamma_1} \geq (a' \Vol S)^{b' \mu^n}
$$
for any measurable $S\subset X$ and any $n \geq 0$.  Here, $b' = b \gamma_1 \gamma_2$ and $a' = a_2 a^{1/\gamma_2} a_1^{1/(\gamma_2 b)} \leq a_2 a^{1/\gamma_2} a_1^{1/(\gamma_2 b \mu^n)}$ for all $n \geq 0$ since $\mu \geq 1$.  Hence $\mu$ is also a dynamical lower volume exponent for $R_2$.
\end{proof}

The main theorem of this section concerns dynamical upper and lower volume
exponents for toric surface maps as in Theorem \ref{thm:mainA}.  Note that by (iii) of
Proposition \ref{PROP:BASIC_PROPERTIES_VOL_EXPONENTS}, for purposes of determining dynamical volume exponents, 
it does not matter which toric surface $X$ a given toric map acts on.

\begin{thm}
\label{thm:volshrink}
Suppose $f:\torus\tto \torus$ is an internally stable toric map and its
tropicalization $A_f:N_\R\to N_\R$ is a homeomorphism with irrational rotation
number.  Then 
\begin{itemize}
 \item any $\mu > \sqrt{\dtop(f)}$ is a dynamical lower volume exponent for $f$;
 \item any $\mu < 1$ is an dynamical upper volume exponent for $f$.  
\end{itemize}
\end{thm}

The following result and item (ii) from
Proposition \ref{PROP:BASIC_PROPERTIES_VOL_EXPONENTS} allow us to replace $f$ in Theorem~\ref{thm:volshrink} 
by an iterate in order to assume, for any given
$\mu>\sqrt{\dtop(f)}$, that 
\begin{align*} %\label{EQN:TROP_VS_MU}
\frac{\dtop(f)}{\mu}\|\nvec\| < \|A_f(\nvec)\| < \mu\|\nvec\|
\end{align*}
 for all non-zero $\nvec\in N$.  Moreover, we remark that (see \cite{DiLi16}
Corollary 8.3 and the surrounding discussion) the tropicalization $A_f$ of a
\emph{birational} toric map always has a rational rotation
number.  Hence any map satisfying the hypothesis of Theorem \ref{thm:volshrink}
has topological degree $\dtop(f) \geq 2$.  So by
choosing $\mu>\sqrt{\dtop(f)}$ sufficiently close to $\sqrt{\dtop(f)}$ we can assume that
\begin{align}\label{EQN:TROP_VS_MU2}
\|\nvec\| \leq \|A_f(\nvec)\| < \mu\|\nvec\|
\end{align}
for all non-zero $\nvec\in N$.

\begin{thm}\label{THM:GOOD_ITERATES} Suppose that $f:\torus\tto \torus$ is toric, and its tropicalization $A_f:N_\R\to N_\R$ is a homeomorphism with irrational rotation number. Then 
$$
\lim_{n\to\infty} \norm{A_f^n(\nvec)}^{1/n} = \sqrt{\dtop(f)}
$$
uniformly on the unit circle $\{\norm{\nvec}=1\}$.  Consequently, for any $\mu>\sqrt{\dtop}$ and any curve $C_\tau\subset\rzt\setminus\torus$, we
have
$$
\lim_{n\to\infty} \frac{\ram(\rzf^n,C_\tau)}{\mu^n} = 0.
$$
\end{thm}

\begin{proof}
Since the induced circle homeomorphism $\tilde A_f$ has irrational rotation number, a slight generalization (see \cite[Theorem 2.5]{ADM} ) of Denjoy's Theorem tells us that it is topologically conjugate to an actual irrational rotation. In particular $\tilde A_f$ is uniquely ergodic.  Hence if $\varphi:N_\R \to \R$ is the continuous function $\varphi(\nvec) := \log\frac{\norm{A_f(\nvec)}}{\norm{\nvec}}$, we obtain that
$$
\log\norm{A_f^n(\nvec)}^{1/n} = \frac1n \sum_{j=0}^{n-1} \varphi\circ \tilde A_f^j(\nvec)
$$
converges uniformly on the unit circle to some constant $L\in \R$.  So for any $\epsilon>0$ there is an index $N$ such that $n\geq N$ implies
$$
(L-\epsilon)^n < \norm{A_f^n(\nvec)} < (L+\epsilon)^n
$$
for all $\nvec$ on the unit circle.  Recall that $A_f$ is $1$-homogeneous and has Jacobian a.e. equal to $\cov(f)$.  Moreover $\cov(f) = \dtop(f)$ since $A_f$ is a homeomorphism.  So we infer that
$$
(L-\epsilon)^{2n} \leq \dtop^n = \pi^{-1}\vol A_f^n(\{\norm{\nvec}\leq 1\}) \leq (L+\epsilon)^{2n}.
$$
Taking $n$th roots and letting $\epsilon\to 0$ concludes the proof of the first assertion.  

To deduce the second assertion, we let $\nvec_n\in N$ be the primitive vector generating $\rzf^n(C_\tau)$.  Note $\norm{v_n}\geq 1$ for all $n$.  So the fourth conclusion in Theorem \ref{thm:tropmap} and the previous paragraph tell us that for any $\tilde\mu\in (\sqrt{\dtop(f)},\mu)$, there exists $C = C(\tilde\mu)$ such that 
$$
\ram(\rzf^n,C_\tau) = \frac{\norm{A_f^n\nvec_0}}{\norm{\nvec_n}} \leq \norm{A_f^n\nvec_0} \leq C\norm{\nvec_0}\tilde\mu^n.
$$
Hence
$$
\lim_{n\to\infty} \frac{\ram(\rzf^n,C_\tau)}{\mu^n} \leq C\norm{\nvec_0} \lim_{n\to\infty} \left(\frac{\tilde\mu}{\mu}\right)^n = 0.
$$
\end{proof}

The next result holds without the assumption that $A_f$ is a homeomorphism.

\begin{lem}\label{lem:xy_new} 
Suppose that $f: \torus \tto \torus$ is toric and that $A_f$ satisfies (\ref{EQN:TROP_VS_MU2}) 
for some constant $\mu > 1$ and for all $\nvec\in N_\R$.  Let $X$ be a toric surface and $P_\sigma$ be bounded distinguished coordinate
bidisks about the $\torus$-invariant points $p_\sigma\in X$.
\begin{enumerate} 
\item There exists $\beta_1>0$ and for any neighborhood $U$of $\rzf(\exc(\rzf))$ a constant $\alpha_1(U)>0$ such that if $p\in P_\sigma\cap\torus$ satisfies $f^j(p)\notin U$ for $1\leq j\leq n$ and $f^n(p)\in P_{\sigma'}\cap \torus$, then
\begin{align}
|x_1'x_2'| \geq (\alpha_1 |x_1y_1|)^{\beta_1 \mu^n},
\end{align}
where $(x_1,x_2)$ and $(x_1',x_2')$ are distinguished coordinates for $p$ and $f^n(p)$, respectively.
\item Similarly, there exists $\beta_2>0$ and for any neighborhood $V$of $\ind(\rzf)$ a constant $\alpha_2(V)>0$ such that if $p\in P_\sigma$ satisfies $f^j(p)\notin V$ for $0\leq j\leq n-1$, and $f^n(p) \in P_{\sigma'}$, then 
$$
|x_1' x_2'| \leq \alpha_2^n |x_1 x_2|^{\beta_2}.
$$
\end{enumerate}
\end{lem}

% 
% , there exists
% a constants $\beta_1, \beta_2 > 0$, and, for any $U$ and $V$ are
% neighborhoods of $\rzf(\exc(\rzf))$ and $\ind(\rzf)$ respectively, constants
% $\alpha_1(U)$ and $\alpha_2(V)$ such that the following holds.
% 
% Suppose $p
% \in P_\sigma\cap\torus$ has distinguished coordinates $(x_1,x_2)$ and
% $\rzf^n(p)\in P_{\sigma'}\cap\torus$ has distinguished coordinates
% $(x_1',x_2')$.
% 
% \begin{enumerate} 
% \item If $f^j(p)\notin U$ for $1\leq j\leq n$, then 
% 
% \item If $f^j(p)\notin V$ for $0\leq j\leq n-1$, then
% 
% \end{enumerate}
% \end{lem}

\begin{proof}
If $P_\sigma$ is any of the finitely many distinguished coordinate bidisks covering $X$ and $p = (x_1,x_2)\in P_\sigma$ is expressed in distinguished coordinates about $p_\sigma$, then we see from \eqref{eqn:twoform} that 
$$
c\|\trop(p)\| \leq -\log|x_1| - \log|x_2| \leq C\|\trop(p)\|
$$
for some constants $c,C>0$ depending only on the bidisks in the cover.  By exponentiating this and iterating the estimate \eqref{EQN:DESIRED_CONCLUSION}, we obtain the following bound for any $p\in\torus$ such that $f^j(p)\notin U$, $1\leq j \leq n$:
$$
|x'_1x'_2| \geq e^{-c\|\trop(f^n(p))\|} \geq e^{-\mu^n (c\|\trop(p)\| + A)} \geq (e^A |x_1x_2|^{c/C})^{\mu^n},
$$
where $A = R \sum_{j=0}^\infty \mu^{-j}$.  So the first conclusion holds. The second conclusion follows from a similar estimate using \eqref{EQN:DESIRED_CONCLUSION2} instead of \eqref{EQN:DESIRED_CONCLUSION} and the hypothesis that $\|A_f(\nvec)\| \geq \|\nvec\|$ for all $\nvec\in N_\R$:
\begin{align*}
|x'_1x'_2| \leq e^{-C\|\trop(f^n(p))\|} \leq e^{-C(\|\trop(p)\| - R n)} \leq (e^{CR})^n |x_1x_2|^{C/c}.
\end{align*}
\end{proof}

The next lemma and its proof are similar to that of Lemmas 5.2 and 5.6 from \cite{BLR} and Proposition 6.3 from \cite{FAVRE_JONSSON}. 

\begin{lem} \label{LEM:INTEGRALS}
There exists $c > 0$ such that for any measurable subset $S$ of the unit bidisk $\mathbb{D}^2\subset\C^2$,
and any $\gamma \geq 0$, we have
\begin{align*}
\int_S |x_1 x_2|^\gamma dV_{{\eucl}} \geq  (c \Vol S)^{1+\gamma}.
\end{align*}
Similarly, for $-2 < \gamma' < \gamma < 0$, there exists $c(\gamma',\gamma)>0$ such that for any measurable $S\subset \mathbb{D}^2$,
\begin{align*}
\int_S |x_1 x_2|^\gamma dV_{{\eucl}} \leq  c(\gamma',\gamma)(\Vol S)^{1+\gamma'/2}.
\end{align*}
\end{lem}

\begin{proof}
For any $\tau \in [0,1]$ let $Y_\tau = \{(x_1,x_2) \in \mathbb{D}^2 \, : \, |x_1 x_2| \leq \tau\}$.   A direct polar coordinate calculation yields for any $\gamma > -2$ that
$$
\int_{Y_\tau} |x_1x_2|^\gamma\, dV_{\eucl} = \frac{4\pi^2\tau^{\gamma+2}}{(\gamma+2)^2} \left(1-(\gamma+2)\log\tau\right) = \frac{4\tau^\gamma}{(\gamma+2)^2} \left(\vol Y_\tau- \pi^2\tau^2\gamma\log\tau\right),
$$
since setting $\gamma = 0$ gives in particular that $\vol Y_\tau = \pi^2\tau^2(1-2\log\tau)$.
Now given a measurable set $S\subset \mathbb{D}^2$, choose $\tau$ so that $\vol Y_\tau = \vol S$.  Then for $\gamma \geq 0$ we have
$$
\int_S |x_1x_2|^\gamma\, dV_{\eucl} 
\geq 
\int_{Y_\tau} |x_1x_2|^\gamma\, dV_{\eucl} 
\geq 
\frac{4\tau^\gamma \vol Y_\tau}{(\gamma+2)^2} 
\geq
\frac{4(C\vol Y_\tau)^{1+\gamma}}{(\gamma+2)^2} 
\geq
(c\vol S)^{1+\gamma}
$$
for some $c>0$ independent of $\gamma$.  In the third inequality, we used that $\vol Y_\tau \leq 2\pi^2 e^{-1/2}\tau$.  In the last, we used $\vol Y_\tau = \vol S$ and the fact that $\underset{\gamma\to\infty}\lim\left(\frac{4}{(\gamma+2)^2}\right)^{\frac{1}{\gamma + 1}}=1$.

If on the other hand $\gamma\in(-2, 0)$, then the first two estimates reverse:
$$
\int_S |x_1x_2|^\gamma\, dV_{\eucl} 
\leq 
\int_{Y_\tau} |x_1x_2|^\gamma\, dV_{\eucl} 
\leq 
\frac{4\tau^\gamma\,\vol Y_\tau}{(\gamma+2)^2}. 
% \leq
% \frac{4\,(C\vol Y_\tau)^{1+\gamma/2}}{(\gamma+2)^2} 
% = c(\gamma) (\vol S)^{1+\gamma/2}.
$$
Moreover, for any $\sigma\in(0,1)$, one checks that $\vol Y_\tau \leq C(\sigma)\tau^{2\sigma}$ for e.g. $C(\sigma) :=\frac{\pi^2 e^{-\sigma}}{1-\sigma}$.  So given $\gamma'\in(-2,\gamma)$, we take $\sigma = \gamma/\gamma'$ to get 
$$
\int_S |x_1x_2|^\gamma\, dV_{\eucl} 
\leq 
\frac{4(\vol Y_\tau)^{1+\gamma'/2}}{C(\sigma)^{\gamma/2\sigma}(\gamma+2)^2}
= 
c(\gamma,\gamma')(\vol S)^{1+\gamma'/2}.
$$
\end{proof}

Next we apply Lemmas \ref{lem:xy_new} and \ref{LEM:INTEGRALS} to estimate volumes of forward images $f^n(S)$ of a measurable subset $S\subset X$.  Since points and curves in a surface have measure zero, we may assume here and below that $S\subset \torus\setminus\bigcup_{n\geq 1}\exc(\rzf^n)$.  Hence $f^n(S)\subset \torus$ for all $n\in\N$.  

\begin{lem}\label{LEM:DYNAMICAL_VOL_EST_GOOD_ITERATES}
Suppose in Lemma \ref{lem:xy_new} that $X$ is endowed with a smooth volume form.  
\begin{enumerate}
 \item\label{item:lower} There exists $b > 1$ and, for any neighborhood $U$ of $\rzf(\exc(\rzf))$ a constant $a = a(U) > 0$ such that if 
 $S\subset X$ is measurable and $f^j(S)\cap U = \emptyset$ for all $1\leq j \leq n$, then
 \begin{align*}
 \vol f^n(S) \geq (a\vol S)^{b\mu^{n}}.
 \end{align*}
 \item\label{item:upper} There exists $b < 1$ and for any neighborhood $V$ of $\ind(\rzf)$ a constant $a = a(V)\geq 1$ such that if $S\subset X$ is measurable and $f^j(S)\cap U = \emptyset$ for $0\leq j < n$, then
 $$
 \vol f^n(S) \leq a^n (\vol S)^b.
 $$
\end{enumerate}
\end{lem}
Our proof combines that of Lemma 3.12 in \cite{Dil96} with the estimate in
Lemma \ref{lem:xy_new} and the fact that $f$ has constant Jacobian relative to
$dV_\torus := \twoform\wedge\bar\twoform$. 
Note that for the proof of Theorem \ref{thm:volshrink}, it is important in both conclusions of Lemma \ref{LEM:DYNAMICAL_VOL_EST_GOOD_ITERATES} that the constant $b$ does not depend on the choice of neighborhood $U$ or $V$.

\begin{proof}
If $p_\sigma\in X$ is $\torus$-invariant and $P_\sigma\subset X$ is the unit coordinate bidisk about $p_\sigma$, then the given volume form on $X$ will be uniformly comparable on $\overline{P_\sigma}$ to Euclidean volume $dV_{eucl} := |x_1x_2|^2 dV_\torus$ in distinguished coordinates.  
Hence we will proceed as if the two volumes are the same.  This will only affect the value of the constant $a$ in the conclusions.

We first prove item (\ref{item:lower}). As $X = \bigcup_{\sigma\in\Sigma_2(X)} \overline{P_\sigma}$, we can partition $S$ into finitely many subsets $S_{\sigma,\sigma'} := S \cap P_\sigma\cap f^{-n}(P_\sigma')$, and focus only on the one for which $\vol f^n(S_{\sigma,\sigma'})$ is maximal.  That is, we may assume without loss of generality that $S = S_{\sigma,\sigma'}\subset \overline{P_\sigma}$ and $f^n(S)\subset \overline{P_{\sigma'}}$. Let $(x_1,x_2)$ and $(x_1',x_2') = f^n(x_1,x_2)$ be distinguished coordinates on $P_\sigma$ and $P_\sigma'$.  Then we have on $S$ that
\begin{align*}
f^{n*} dV_{{\eucl}} = f^{n*}(|x_1x_2|^2\,dV_\torus) = \cov(f)^{2n} |x_1'x_2'|^2\,dV_\torus
= \cov(f)^{2n} \frac{|x_1'x_2'|^2}{|x_1x_2|^2}\, dV_{{\eucl}},
\end{align*}
Lemmas \ref{lem:xy_new} and \ref{LEM:INTEGRALS} therefore give
\begin{eqnarray*}
\vol f^n(S)  & \geq & \frac{1}{\dtop(f)^n} \int_S f^{n*} dV_{{\eucl}} = \left(\frac{\cov(f)^2}{\dtop(f)}\right)^{n}   \int_S \frac{|x_1'x_2'|^2}{|x_1x_2|^2}\, dV_{{\eucl}} \\
& \geq & \int_S \alpha^{2 \beta \mu^n} |x_1x_2|^{2 \beta \mu^n-2} \,dV_{{\eucl}} \geq  (\alpha c \Vol S)^{2 \beta \mu^n}.
\end{eqnarray*}
So statement (1) holds with $a = \alpha c$ and $b = 2\beta$.

The proof of item (\ref{item:upper}) is similar.  As before, we can assume $S \subset \overline{P_\sigma}$ and $f^n(S)\subset\overline{P_{\sigma'}}$.  Since $f^j(S)\cap V = \emptyset$ for $0\leq j < n$, Lemma \ref{lem:xy_new} gives
\begin{eqnarray*}
\vol f^n(S) 
& \leq & 
\int_S f^{n*}dV_{eucl} 
=
\cov(f)^{2n}\int_S \frac{|x_1'x_2'|^2}{|x_1x_2|^2}\, dV_{eucl} \\
& \leq & 
(\alpha\cov(f))^{2n} \int_S |x_1x_2|^{2(\beta-1)}\, dV_{eucl} 
\leq 
a^n (\Vol S)^b,
\end{eqnarray*}
where we choose $b \in (\beta,1)$ and then $a$ satisfying $a^n \geq C (\alpha\cov(f))^{2n}$ for all $n \neq 0$, where $C = C(2b-2,2\beta-2)$ is the multiplicative constant
appearing in the second conclusion of Lemma \ref{LEM:INTEGRALS}.
\end{proof}

\begin{lem}\label{LEM:BOUNDED_NUMBER_OF_RETURNS}
If $f:\torus\tto\torus$ is as in Theorem \ref{thm:volshrink}, then for any integer $m\geq 0$, there exist neighborhoods $U$ of $\rzf(\exc(\rzf))$ and $V$ of $\ind(\rzf)\setminus\torus$ such that for any $p \in \torus \setminus \ind(\rzf^m)$ 
\begin{enumerate}
 \item $\rzf^n(p) \in U$ for at most $\#\rzf(\exc(\rzf))$ iterates $0\leq n \leq m$; and 
 \item $\rzf^n(p)\in V$ for at most $\#\ind(\rzf)$ iterates $0\leq n\leq m$.
\end{enumerate}
\end{lem}

\begin{proof}
The hypothesis that $f$ is internally stable implies that each iterate $\rzf^n$ of $\rzf$ is well-defined and continuous on an $n$-dependent neighborhood of $\rzf(\exc(\rzf))$.  The hypothesis that the rotation number of $A_f$ is irrational implies that for any two poles $C_\tau, C_\tau'\in \rzt$, there is at most one $n\in\Z$ such that $\rzf^n(C_\tau) = C_\tau'$.  Since each point of $\rzf(\exc(\rzf))$ is contained in a unique pole, it follows for any $q,q'\in\rzf(\exc(\rzf))$ that $\rzf^n(q) = q'$ for at most one $n\in\Z$.

It follows that we can choose neighborhoods $U(q)$ of the points $q\in \exc(\rzf)$ so that $\rzf^n|_{U(q)}$, $0\leq n< m$ is well-defined and continuous, and for any  given $q,q'\in\rzf(\exc(\rzf))$, the intersection $\rzf^n(U(q))\cap U(q')$ is non-empty for at most one $0\leq n <m$.  The lemma then follows on taking $U$ to be the union of the neighborhoods $U(q)$.

The construction of $V$ is similar but also depends on
\eqref{EQN:IMAGE_OF_INDET_PT_ON_POLE}.  That is, if $\alpha\in\ind(\rzf)\cap
C_\tau$ is an external persistently indeterminate point, then
\eqref{EQN:IMAGE_OF_INDET_PT_ON_POLE} implies that $\rzf^n(\alpha)
\setminus\torus \subset \rzf^n(\exc(\rzf)) \cup C_{A_f^n(\tau)}$.  So internal
stability of $f$ implies that $\rzf^n(\alpha) \cap \left(\ind(\rzf)\setminus\torus\right)
\subset C_{A_f^n(\tau)}$.  And again, since $A_f$ has irrational rotation
number, $C_{A_f^n(\tau)} \cap \ind(\rzf) \neq \emptyset$ for only finitely many $n$.
\end{proof}

\begin{lem}
\label{lem:volshrinkrevisited}
Suppose that $f:\torus\tto\torus$ satisfies the hypotheses of Theorem \ref{thm:volshrink}, 
that (\ref{EQN:TROP_VS_MU2}) holds with constant $\mu>1$, and that $X$ is a toric surface with a smooth volume form. 
\begin{enumerate}
\item There exists $b>0$ and for any $m\in\N$ a constant $a=a(m)>0$ such that for any measurable $S\subset X$, 
$$
\vol(f^m(S)) \geq (a\vol S)^{b \mu^{m}}.
$$
\item Similarly, there exists $0<\nu<1$ and for any $m\in\N$ a constant $A = A(m)\geq 1$ such that 
$$
\vol(f^m(S)) \leq A(\vol S)^\nu.
$$
\end{enumerate}
\end{lem}

Note that Theorem \ref{thm:volshrink} amounts to Lemma \ref{lem:volshrinkrevisited} with $a$ and $A$ independent of $m$.

\begin{proof}
Part (i) of Proposition \ref{PROP:BASIC_PROPERTIES_VOL_EXPONENTS} gives $a, \gamma > 0$ such that $\Vol f(S) \geq (a \Vol S)^{\gamma}$.  Fix $m$, and let $U\supset \rzf(\exc(\rzf))$ satisfy the conclusion of Lemma \ref{LEM:BOUNDED_NUMBER_OF_RETURNS}.  Let $e = \#\rzf(\exc(\rzf))$.  By partitioning $S$ into finitely many pieces $S_1,\dots, S_{\ell(m)}$, we may assume for each $S_j$ that there are times $0\leq n_1(j)<\dots < n_e(j)<m$ such that $f^n(S_j) \cap U = \emptyset$ for all \emph{other} $0\leq n<m$.  We focus on the piece $S_j$ for which $\vol S_j \geq \frac1\ell \vol S$ is maximal.

Let $a',b' > 0$ be the constants from Lemma \ref{LEM:DYNAMICAL_VOL_EST_GOOD_ITERATES}.  Then 
$$
\Vol f^{n_1}(S_j) \geq (a' \Vol S_j)^{b' \mu^{n_1}}. 
$$
For $1\leq k < e$, we further estimate
\begin{eqnarray*}
\Vol f^{n_{k+1}}(S_j) & \geq & (a' \Vol f^{n_k+1}(S_j))^{b' \mu^{n_{k+1}-n_k-1}} \geq (a'(a\Vol f^{n_k}(S_j))^\gamma)^{b' \mu^{n_{k+1}-n_k-1}} \\
& = &
(a'' \Vol f^{n_k}(S_j))^{b'' \mu^{n_{k+1}-n_k}}.
\end{eqnarray*}
where $a'' = (a')^{1/\gamma} a$ and $b'' = \gamma b' \mu^{-1}$.  Similarly,
$
\vol f^m(S) = (a'' \vol f^{n_e}(S_j))^{b'' \mu^{m-n_e}}
$
Combining these inequalities we find 
\begin{align*}
\Vol f^m(S) \geq \Vol f^m(S_j) & \geq (\tilde{a} \Vol S_j)^{b'(b'')^e \mu^{m}} \geq \left(\frac{\tilde{a}}{\ell}\vol S\right)^{\tilde{b}\mu^m},
\end{align*}
for some constant $\tilde{a}$ depending on $m$ but $\tilde{b} = b'(b'')^e$ independent of $m$.

The upper bound for $\vol f^n(S)$ similarly follows from Lemma \ref{LEM:BOUNDED_NUMBER_OF_RETURNS} and item (\ref{item:upper}) in Lemma~\ref{LEM:DYNAMICAL_VOL_EST_GOOD_ITERATES}
\end{proof}

\begin{proof}[Proof of Theorem \ref{thm:volshrink}]
We give the proof for the dynamical lower exponent only.  The arguments for the dynamical upper exponent are nearly identical.

Fix $\mu>\sqrt{\dtop(f)}$ and $\sigma \in (\sqrt{\dtop(f)},\mu)$.  As noted after the statement of Theorem \ref{thm:volshrink}, we may assume that $\norm{A_f(\nvec)} <  \sigma\norm{\nvec}$ for all non-zero $\nvec\in N_\R$.  Let $b>0$ be the ($m$-independent) constant obtained by applying Lemma \ref{lem:volshrinkrevisited} with $\sigma$ in place of $\mu$.  Choose $m\in\N$ large enough that $b \sigma^m \leq \mu^m$.   Lemma \ref{lem:volshrinkrevisited} then
gives a constant $a=a(m) > 0$ such that for any measurable $S \subset X$ we have 
\begin{align}\label{EQN:CONCLUSION_SIGMA_MU}
\vol(f^m(S)) \geq (a\vol S)^{b \sigma^{m}} \geq (a \vol S)^{\mu^m}.
\end{align}

Given $n\in\N$, we write $n=mq + r$ for $q,r\in\N$ with $0 \leq r < m$.   By Part (i) of 
Proposition \ref{PROP:BASIC_PROPERTIES_VOL_EXPONENTS} there exists $\gamma = \gamma(m)>1$ and $\hat{a} = \hat{a}(m) > 0$ such that 
\begin{align*}
\Vol f^r(S) \geq (\hat{a} \Vol S)^\gamma.
\end{align*}
for any $0 \leq r < m$.  Iterating (\ref{EQN:CONCLUSION_SIGMA_MU}) $q$ times, we then find that
\begin{align*}
\vol f^n(S) \geq (a' \Vol f^r(S))^{(\mu^{m})^{q}} \geq (a' (\hat{a} \Vol(S))^\gamma)^{\mu^n} \geq (a'' \Vol S)^{b'\mu^{n}},
\end{align*}
where $a' \equiv a'(m) = a(m)^\tau$ for $\tau = \sum_{k=0}^\infty (\mu^{m})^{-k}$, $a'' = (a')^{1/\gamma} \hat{a}$ and $b' = \gamma$.  In particular, neither $a''$ nor $b'$ depend on $n$.  

%\textcolor{blue}{(Proof of the upper volume estimate, to be commented out later.)
%As noted after the statement of Theorem \ref{thm:volshrink} we can assume $\norm{A_f(\nvec)} \geq \norm{\nvec}$ for all non-zero $\nvec\in N_\R$.
%Let $0 < \mu < 1$ be the desired dynamical upper volume exponent and let $0 < \nu < 1$ be the exponent in the conclusion of Lemma \ref{lem:volshrinkrevisited}.
%Fix some $m\in\N$ large enough that $\mu ^m < \nu$.   Lemma \ref{lem:volshrinkrevisited} then
%gives a constant $A=A(m) > 0$ such that for any measurable $S \subset X$ we have
%\begin{align}\label{EQN:CONCLUSION_NU_MU}
%\vol(f^m(S)) \leq A (\vol S)^{\nu} \leq  A(\vol S)^{\mu^m}.
%\end{align}
%}

%\textcolor{blue}{
%Given $n\in\N$, we write $n=mq + r$ for $q,r\in\N$ with $0 \leq r < m$.   By Part (i) of
%Proposition \ref{PROP:BASIC_PROPERTIES_VOL_EXPONENTS} there exists $\gamma = \gamma(m) \leq 1$ and $\hat{a} = \hat{a}(m) > 0$ such that
%\begin{align*}
%\Vol f^r(S) \leq \hat{a} (\Vol S)^\gamma
%\end{align*}
%for any $0 \leq r < m$.   (Without loss of generality, we can assume $\hat{a} \geq 1$.)
%Iterating (\ref{EQN:CONCLUSION_NU_MU}) $q$ times, we then find that
%\begin{align*}
%\vol f^n(S) \leq a' (\Vol f^r(S))^{(\mu^{m})^{q}} \leq a' (\hat{a} (\Vol S)^\gamma)^{\mu^n} \leq a'' (\Vol S)^{b'\mu^{n}},
%\end{align*}
%where
%$b' = \gamma$, $a' \equiv a'(m) = A(m)^\tau$ for $\tau = \sum_{k=0}^\infty (\mu^{m})^{k}$, and $a'' = a' \hat{a} \geq a' \hat{a}^{\mu^n}$ for all $n \geq 0$ since
%$\hat{a} \geq 1$ and $0 < \mu < 1$.
%In particular, neither $a''$ nor $b'$ depend on $n$.    
%}
\end{proof}

\section{Positive closed $(1,1)$ currents on toric surfaces}
\label{sec:currents} In this section we consider the relationship between the
set $\pcc(\torus)$ of currents on the algebraic torus, and the subset of
$\pcc(\torus)$ obtained by restricting currents from a toric surface $X$
compactifying $\torus$.  By \cite{Siu74}, a current in $\pcc(X)$ is supported
entirely on poles of $X$ if and only if it belongs to the set $\pcce(X)$ of
external divisors on $X$.  At the other extreme, we call a current in $\pcc(X)$
\emph{internal} if it does not charge any of the poles of $X$.  Alternatively,
$T\in\pcc(X)$ is internal if it is equal to (the trivial extension to $X$ of)
its restriction to $\torus$.  Internal currents on $X$ constitute a dense, but
not closed, subspace $\pcci(X)\subset \pcc(X)$.  

\begin{prop}
\label{prop:eidecomp1}
Any current $T \in \pcc(X)$ decomposes uniquely as a sum $T=D + T'$ where $D\in \pcce(X)$ is an external divisor and $T' \in \pcci(X)$ is an internal current.
\end{prop}

\begin{proof}

Since $\pcc(X)$ consists of differences of positive currents, we may suppose
that $T$ is positive.  Since $X\setminus \torus$ is a finite union of
poles, the Skoda-El Mir Theorem implies that the
restriction $T|_{\torus}\in \pcc(\torus)$ to the internal subset of $X$ extends
trivially to a positive closed current $T'$ on $X$ that does not charge any
pole.  
\end{proof}
  
The set of internal currents is independent of the underlying toric surface.

\begin{prop}
\label{prop:internalcurrents}
If $X\succ Y$ are toric surfaces, then $\pi_{XY*}:\pcci(X)\to \pcci(Y)$ is an isomorphism. 
\end{prop}

\begin{proof}
Set $\pi = \pi_{XY}$ and let $T\in\pcci(X)$ be an internal current.  Since
$\pi(\torus)=\torus$, we have that $\pi_*T$ is an internal current on $Y$.  The
map $\pi_*:\pcci(X)\to \pcci(Y)$ is injective since $\pi$ contracts only poles
of $X$.  In the other direction we define $\pi^\sharp:\pcci(Y) \to \pcci(X)$ by
declaring $\pi^\sharp T$ to be the internal component of $\pi^* T$.  To
conclude the proof it suffices to show that $\pi_*\pi^\sharp T= T$.  But since
$\pi$ is a morphism we have that $\pi_*\pi^* T = T$, and since $T$ itself does
not charge poles of $X$, we have that $\pi^* T - \pi^\sharp T$ is supported
only on poles contracted by $\pi$.  Hence $\pi_*\pi^\sharp T = \pi_*\pi^* T =
T$.
\end{proof}

Another way to read Proposition \ref{prop:internalcurrents} is that the image
of $\pcci(X)$ under the restriction map $\pcci(X)\to \pcc(\torus)$ is
independent of $X$.  We denote this image by
$\pcci(\rzt)$ and refer to its elements as \emph{internal currents} without
reference to any particular surface.  When we want to emphasize the surface, we
will write $T_{X}$ for the trivial extension of $T\in\pcci(\rzt)$ to a positive
closed current on $X$.  In \S \ref{sec:tcurrents} we will amplify this notation
into a broader discussion of what we call toric currents on $\rzt$.  For now,
we note that the discussion above implies a canonical isomorphism $\pcc(X)
\cong \pcce(X)\oplus \pcci(\rzt)$.  The rest of this section is devoted to
better understanding the relationships among $\pcc(\torus)$, $\pcci(\rzt)$ and
cohomology classes on toric surfaces.

\subsection{Currents on $\torus$ and convex functions.}  
% Recall from \S\ref{sec:rztoric} the tropicalization map $\trop:\torus \to N_\R$ given in any distinguished coordinate system by $\trop(x_1,x_2) = (-\log|x_1|,-\log|x_2|)$.  
If $\sfn:N_\R\to\R$ is convex, then $\sfn\circ\trop:\torus \to \R$ is plurisubharmonic.  Hence $\sfn$ determines a current $T=dd^c(\sfn\circ\trop)\in \pcc^+(\torus)$ that is invariant by the `rotational' action of the maximal compact subgroup $\torus_\R := \trop^{-1}(0) \subset~\torus$.  

Here we will partially reverse the association of positive currents to convex
functions by assigning to each $T\in\pcc^+(\torus)$ a succession of three
increasing simple rotationally symmetric approximations $T_{ave}$, $\bar T$,
$\bar T(X)\in\pcc^+(\torus)$.  The current $T_{ave}$ will exist for {\em any}
current $T~\in~\pcc^+(\torus)$, but existence of $\bar T$ and $\bar{T}(X)$
requires imposing a growth condition (see \eqref{EQN:GROWTH} below) on
potentials for $T$.  For our purposes, $\bar T$ will be the most important
approximation.  The simplest on the other hand will be $\bar T(X)$. As the
notation implies, it will depend on a choice of toric surface $X$ as well as
the given current $T$.

\vspace{0.1in}
\noindent
{\bf Convention:} 
From here forward and for the sake of clarity, we will typically specify the surface on which $dd^c$ is operating by using a subscript: e.g. $dd^c_X : \dpsh(X) \rightarrow \mathcal{D}_{1,1}(X)$, where $\dpsh(X)\subset L^1(X)$ denotes integrable functions locally equal to the difference between two plurisubharmonic functions.
When we omit the subscript, the implied domain is $\torus$.  

Since $H^1(\torus,\mathcal{O})$ is trivial, positive closed $(1,1)$ currents on $\torus$ admit global potentials.

\begin{prop}
\label{prop:potentialexists}
For any $T\in \pcc^+(\torus)$, there exists $u\in\psh(\torus)$ such that $T=dd^c u$.  
\end{prop}

Note that for each $\nvec\in N_\R$, the set $\{x \in \torus \, : \, \trop(x) = \nvec\}$ is a real $2$-torus equivalent by translation to $\torus_\R$.  Given $T = dd^c u \in\pcc^+(\torus)$, we define a function $\sfn_u:N_\R\to\R$ by averaging over translations of $\torus_\R$.  That is,
$$
\sfn_u(\nvec) := \int_{\trop(x) = \nvec} u(x)\,\eta.
$$
{When $u = \log|P|$ for $P:\torus\to \C$ a Laurent polynomial, then $\sfn_u$ is the well-known `Ronkin function' from tropical geometry.  In general, $\sfn_u\circ\trop$ is the coordinate-wise average of $u$ with respect to some/any choice of distinguished coordinates, hence also plurisubharmonic.}  It follows that $\sfn_u$ is convex (in particular continuous).  Moreover, $\sfn_u$ is affine if and only if $u$ is pluriharmonic. So $\sfn_u$ is determined by $T$ up to addition of affine functions of $\nvec$, and the current 
\begin{align*}
T_{ave} := dd^c(\sfn_u\circ\trop) \in\pcc^+(\torus)
\end{align*}
 is completely
determined by $T$.  

For any convex $\sfn:N_\R\to\R$, the quantity 
\begin{align}\label{EQN:GROWTH}
\growth{\sfn} := \sup_{\nvec\neq 0}\frac{\sfn(\nvec)-\sfn(0)}{\norm{\nvec}} = \limsup_{\norm{\nvec}\to \infty} \frac{\sfn(\nvec)}{\norm{\nvec}}
\end{align}
is non-negative, though possibly infinite.  It follows for all $\nvec\in N_\R$ that
\begin{align}\label{EQN:GROWTH2}
|\sfn(\nvec)| \leq \growth{\sfn}\cdot \norm{\nvec} + |\sfn(0)|.
\end{align}
As we will see in Theorem \ref{thm:lineargrowth} below, finiteness of $\growth{\sfn_u}$ is necessary and sufficient for extension of $T$ to a positive closed current on arbitrary toric surfaces.

Now if $\sfn = \sfn_u$ is the convex function associated to a potential $u$ for $T\in\pcc^+(\torus)$, then $\growth{\sfn_u}$ depends on $u$ as a well as $T$, but since $\growth{\sfn}$ is finite when $\sfn$ is affine, it follows that $\growth{\sfn_u}$ is finite for one potential $u$ if and only if it is finite for all others.  In particular, one can always add an affine function of $\trop(x)$ to $u$ to ensure that $\sfn_u(\nvec)\geq 0$ for all $\nvec$ and that $\sfn(0) = 0$.  In this case, $\growth{\sfn_u}$ will be nearly minimal, equal to no more than twice the smallest value possible among all potentials for $T$.  In a similar vein, $\growth{\sfn_u}$ depends on the choice of norm on $N_\R$ but only up to a uniform multiplicative non-zero constant.  Likewise, pulling $u$ back by a monomial map changes $\growth{\sfn_u}$ by a non-zero multiplicative factor that is independent of $u$.

\begin{prop}
\label{prop:homogenization}
Suppose $\sfn:N_\R\to\R$ is convex with $\growth{\sfn}<\infty$.  Then $\lim_{t\to\infty} \frac{\sfn(t\nvec)}{t}$ converges normally on $N_\R$ to the minimal convex function $\bar\sfn:N_\R\to\R$ satisfying
\begin{itemize}
 \item $\bar\sfn(t\nvec) = t\bar\sfn(\nvec)$;
 \item $\bar\sfn(\nvec) \geq \sfn(\nvec) - \sfn(0)$.
\end{itemize}
for all $\nvec\in N_\R$ and $t\geq 0$.  Also, $\bar{\bar\sfn} = \bar\sfn$ and 
$
\growth{\sfn} = \growth{\bar\sfn} = \max_{\norm{\nvec}=1}\bar\sfn(\nvec).
$
\end{prop}

\begin{proof}
Exercise.
\end{proof}

If $T = dd^c u \in \pcc^+(\torus)$ and $\growth{\sfn_u} < \infty$, then we call $\bar\sfn_u$ a \emph{support function} for~$T$.  It is, again, determined by $T$ up to addition of affine functions.  Convexity implies we can choose the affine function (non-uniquely) to obtain a non-negative support function.  Alternatively, one can choose it so that the support function takes equal values on each of the three primitive vectors in $N$ that generate the rays in $\Sigma_1(\cp^2)$.  The latter support function might be negative along some rays, but it is uniquely, continuously and linearly determined by $T$.  We denote it by $\sfn_T$.

We call 
\begin{align*}
\bar T := dd^c (\bar\sfn_u\circ\trop) = dd^c(\bar\sfn_T\circ\trop)\in\pcc^+(\torus)
\end{align*}
 the \emph{homogenization} of $T$.  Like $T_{ave}$, the current $\bar T$ is rotationally invariant and determined by $T$, but as will become clearer below, it is more canonical and varies in a more stable fashion than $T_{ave}$.  We call $T$ itself \emph{homogeneous} if $T=\bar T$.  In any case, $\bar T = \bar T_{ave} = \bar{\bar T}$.  {In their recent work disproving Demailly's strengthened Hodge conjecture, Babaee and Huh \cite{BaHu17} gave a different construction of currents like $\bar T$, which when restricted to our context, arrives at $\bar T$ by starting with a tropical curve supported on a finite union of rational rays in $N_\R$.}

For each toric surface $X$, we also associate to the potential $u$ for $T$ the
divisor $D_{u,X}\in\pcce(X)$ with support function $\bar\sfn_{u,X}$
determined by the condition that it coincides with $\bar\sfn_u$ along any ray
$\tau\in\Sigma_1(X)$.  By construction $\bar\sfn_{u,X}$ is convex, so $D_{u,X}$
is nef, satisfying $\pi_{Y,X}^* D_{u,X}\geq D_{u,Y}$ for any $Y\succ X$.  The
homogeneous current 
\begin{align*}
\bar{T}(X) := dd^c(\bar\sfn_{u,X} \circ \trop) \in\pcc^+(\torus)
\end{align*}
 with support function $\bar\sfn_{u,X}$ is determined completely $T$ and the choice of $X$.  Example \ref{eg:assoccurrents} illustrates this below in a specific case.

\begin{prop}
\label{prop:sfnpinch}
For any toric surface $X$, there exists a constant $C(X)>1$ such that if $u\in\psh(\torus)$ has $\growth{\sfn_u}<\infty$, then
$$
\bar\sfn_u\leq \bar\sfn_{u,X} \leq C(X)\,\bar\sfn_u.
$$
In particular $\growth{\bar\sfn_{u,X}} \leq C(X)\growth{\sfn_u}$, and the optimal constant $C(X)$ is decreasing in $X$.  Finally, there exist toric surfaces $Y\succ X$ for which the optimal constant $C(Y)$ is arbitrarily close to $1$.
\end{prop}

\begin{proof}
The assertions follow more or less directly from convexity of $\sfn_u$ and the fact that rays in $\Sigma_1(X)$ generate $N_\R$ as a convex set.  The last assertion holds because for any $\epsilon>0$, one can choose $X$ so that each sector $\sigma\in\Sigma_2$ has angular width smaller than $\epsilon$. 
\end{proof}

\subsection{Finite growth and trivial extension to toric surfaces.}

\begin{prop}
\label{prop:integrability}
Let $X$ be a toric surface with volume form $dV$.  Then there exists a  constant $C$ such that for any convex $\sfn:N_\R\to\R$, we have 
$$
\int_X |\sfn\circ\trop|\,dV \leq C(\growth{\sfn} + |\sfn(0)|).
$$
In particular $\sfn\circ\trop$ is integrable on $X$ if $\growth{\sfn}<\infty$.
Given (additionally) $M>0$ there are constants $a,b>0$ such that if $\growth{\sfn} \leq M$, we have
$$
\vol\{p\in X:t\leq |\sfn\circ\trop(p)|\} \leq ae^{-bt}.
$$
\end{prop}

\begin{proof}
Given a sector $\sigma \in \Sigma_2(X)$, let $(x_1,x_2)$ be distinguished coordinates about the $\torus$-invariant point $p_\sigma$.
It follows immediately from (\ref{EQN:GROWTH2}) that in these coordinates we have 
$$
\sfn\circ\trop(x_1,x_2) \leq |\sfn(0)| + C\growth{\sfn}\sqrt{(\log|x_1|)^2 + (\log|x_2|)^2},
$$
for some $C$ depending only on the choice of coordinate.  Replacing $dV$ with Euclidean volume, one sees that the right side is integrable and satisfies the desired volume estimates on any bounded polydisk $\{|x_1|,|x_2|\leq R\}$.  Since $dV$ is uniformly comparable to Euclidean volume on such polydisks, finitely many of which cover $X$, the proposition follows.
\end{proof}

\begin{thm}
\label{thm:canonicalreps}
Suppose $T = dd^c u \in \pcc^+(\torus)$ satisfies $\growth{\sfn_u}<\infty$, and let $X$ be a toric surface.  Then $\sfn_u\circ\trop$ and $\bar\sfn_u\circ\trop$ are integrable on $X$.  Moreover, 
\begin{align}
dd^c_X(\sfn_u\circ\trop) &= T_{ave} - D_{u,X} \quad \mbox{and} \quad  dd^c_X(\bar\sfn_u\circ\trop) = \bar T - D_{u,X}  \label{EQN:THM_CANONICAL_REPS}.
\end{align}
In particular 
\begin{itemize}
 \item $\bar T$ and $T_{ave}$ extend trivially to elements of $\pcc^+(X)$, both cohomologous in $X$ to $D_{u,X}$ and to (the trivial extension of) $\bar T(X)$;  
 \item the trivial extension of $\bar T$ to $\pcc^+(X)$ has continuous local potentials everywhere on $X$ except possibly about $\torus$-invariant points;
 \item the trivial extension of $\bar T(X)$ to $\pcc^+(X)$ has continuous local potentials everywhere on $X$, and $\supp\bar T(X)$ omits all $\torus$-invariant points.
\end{itemize}
\end{thm}

It follows from the first item in Theorem \ref{thm:canonicalreps} that (the
trivial extensions of) $\bar T$ and $T_{ave}$ are cohomologous in \emph{any}
toric surface.  Later we will signify this by saying the two currents are
\emph{completely cohomologous}.  By way of contrast, however, the current $\bar
T(X)$ is typically only cohomologous to $\bar T$ and $T_{ave}$ on $X$; see
Example \ref{eg:assoccurrents}, below.

\begin{proof}[Proof of Theorem \ref{thm:canonicalreps}]
We first prove the claims about $\bar T$ and $\bar \sfn_u$, noting that it suffices to verify each in a neighborhood of any pole $C_\tau\subset X$.  Given $\tau\in\Sigma_1(X)$, we choose an adjacent sector $\sigma\in\Sigma_2(X)$ and distinguished coordinates $(x_1,x_2)$ about $p_\sigma$ such that 
$C_\tau = \{x_1=0\}$.  That is, we identify $\tau\subset N_\R\cong\R^2$ with $\R\cdot (1,0)$.  By convexity and $1$-homogeneity of $\bar \sfn_u$ 
$$
\lim_{x_1\to 0} \bar\sfn_u\circ\trop(x_1,x_2) - \bar\sfn_u(1,0)\log|x_1|
$$ 
is a well-defined convex function of $\log|x_2|$.  Hence $\bar\sfn_u\circ\trop
- \bar\sfn_u(1,0)\log|x_1|$ extends to a continuous local potential for $\bar
T$ on $\C\times\C^*$.   Thus, $\bar \sfn_u \circ \trop$ is locally integrable on $\C\times\C^*$, and since $\log|x_1|$ is pluriharmonic on $\C^*\times\C^*$ it follows that $\bar T$ extends trivially across $C_\tau = \{x_1 = 0\}$ with a  potential that is continuous except at $(0,0)$, i.e. except at the $\torus$-invariant point $p_\sigma$.  Since $\bar\sfn_u(1,0)$ is the weight of
$C_\tau$ in $D_{u,X}$, it follows that 
the right hand equation in~(\ref{EQN:THM_CANONICAL_REPS}) holds in the neighborhood $\C\times\C^*$ of $C_\tau$.  As $\tau\in\Sigma_1(X)$ was arbitrary, the conclusions concerning $\bar T$ and $\bar\sfn_u$ hold on all of $X$.

Turning to $\sfn_u$ and working in the system of distinguished coordinates, we note that by construction of $\bar\sfn_u$ we have for any fixed $x_2$ that 
$$
\sfn_u\circ\trop(x_1,x_2) - \bar\sfn_u\circ\trop(x_1,x_2) = o(-\log|x_1|)
$$
as $|x_1|\to 0$.  Hence $\sfn_u\circ\trop-\bar\sfn_u\circ\trop$ extends as a locally integrable function on $\C\times\C^*$ and  $dd^c(\sfn_u\circ\trop - \bar\sfn_u\circ\trop)$ has no support on $\{x_1=0\}$.  The left hand equation from (\ref{EQN:THM_CANONICAL_REPS}) follows immediately.

Since $\bar T(X)$ is homogeneous, the first paragraph of the proof applies with $\bar\sfn_{u,X}$ and $\bar T(X)$ in place of $\bar\sfn_u$ and $\bar T$.  Hence $\bar T(X)$ extends trivially to $X$ with continuous local potentials everywhere except possibly at $\torus$-invariant points, and $\bar T(X)$ is cohomologous to $D_{u,X}$ via $dd^c_X(\bar\sfn_{u,X}\circ\trop) = \bar T(X) - D_{u,X}$.  It remains to prove the third assertion of Theorem \ref{thm:canonicalreps}.  For this we again work in distinguished coordinates $(x_1,x_2)$ about some $\torus$-invariant point $p_\sigma\in X$.  The axes $\{x_j=0\}$ are the poles corresponding to the rays $\tau_j\in\Sigma_1(X)$ that bound $\sigma$, and in the implied identification $N\cong \Z^2$, the primitive vectors $(1,0)$ and $(0,1)$ are the generators for $\tau_1,\tau_2$.  Hence by definition we have
$$
\sfn_{u,X}(v) = \sfn_u(1,0)v_1 + \sfn_u(0,1) v_2.
$$
for all $v\in \R^2$ with non-negative coordinates.  So for $(x_1,x_2)$ in the unit bidisk, we have that $\sfn_{u,X}\circ\trop(x_1,x_2) =
-c_{\tau_1}\log|x_1|-c_{\tau_2}\log|x_2|$ is exactly a local potential for
$-D_{u,X}$.   I.e. $\supp\bar T(X)$ does not intersect the interior of the distinguished unit bidisc about $p_\sigma$, which implies the final assertion of Theorem \ref{thm:canonicalreps}.  
\end{proof}

\begin{eg}
\label{eg:assoccurrents}
Let $\sfn(\nvec)~=~\norm{\nvec}$ and set $u:=\sfn\circ\trop$. Then the current $\bar
T = dd^c u \in \pcc^+(\torus)$ is homogeneous with support equal to $\torus$.
Here we describe the ``associated currents'' 
\begin{align*}
T_1:=\bar T(\cp^2) \qquad \mbox{and} \qquad  T_2:=\bar T(\cp^1 \times \cp^1),
\end{align*}
 on $\torus$ and also their trivial extensions to $\cp^2$.
The support functions $\psi_1$ and $\psi_2$ for $T_1$ and $T_2$, respectively, are obtained by restricting $\sfn$ to
the rays in $\Sigma_1(\cp^2)$ and $\Sigma_1(\cp^1 \times \cp^1)$ generated by the sets of primitive vectors
\begin{align*}
\left\{(1,0), (0,1), (-1,-1) \right\} \quad \mbox{and} \quad \left\{(\pm 1,0), (0, \pm 1)\right\}
\end{align*}
and then extending linearly across each sector in $\Sigma_2(\cp^2)$ and $\Sigma_2(\cp^1 \times \cp^1)$.

Let $(x_1,x_2)$ denote the distinguished coordinates on $\cp^2$ or on $\cp^1\times\cp^1$ that are associated to the rays 
generated by primitive vectors $(1,0)$ and $(0,1)$.
One computes that
$\bar T(\cp^2) \in \pcc^+(\torus)$ has nowhere dense support
equal to
\begin{align*}
\{|x_1| = |x_2| \leq 1\}\cup\{|x_1| \geq |x_2|=1\} \cup \{|x_2| \geq |x_1| = 1\}
\end{align*}
and that 
  $\bar T(\cp^1\times\cp^1) \in \pcc^+(\torus)$ has (different) nowhere dense support equal to
\begin{align*}
\{|x_1| = 1\} \cup \{|x_2| = 1\}.
\end{align*}
More generally, if $p_\sigma$ is any $\torus$-invariant point in any toric surface $X$, and $\bar T$ is as above, then the intersection between $\supp\bar T(X)$ and the closed distinguished coordinate bidisc about $p_\sigma$, is exactly the (full) boundary of the bidisk.

By Theorem \ref{thm:canonicalreps} both homogeneous currents $T_1$ and $T_2$ extend trivially to the toric surface~$\cp^2$.  The extensions are not, however cohomologous to each other.  Let $[\tilde{x}_0,\tilde{x}_1,\tilde{x}_2]$ be homogeneous coordinates for $\cp^2$ associated to distinguished coordinates $(x_1,x_2) =
(\tilde{x}_1/\tilde{x}_0,\tilde{x}_2/\tilde{x}_0)$ 
on $\torus$.  Then the external divisors whose support functions are obtained by restricting
$\psi_1$ and $\psi_2$ to $\Sigma_1(\cp^2)$ are 
\begin{align*}
D_1 = (\tilde{x}_1=0) + (\tilde{x}_2=0) + (\tilde{x}_0=0)  \quad \mbox{and} \quad D_2 = (\tilde{x}_1=0) + (\tilde{x}_2=0) + \sqrt{2} (\tilde{x}_0=0).
\end{align*}
Since $\deg D_1 = 3\neq 2 + \sqrt{2} = \deg D_2$, the two divisors are cohomologically inequivalent.  But from Theorem \ref{thm:canonicalreps} we have that $[T_1] = [D_1]$ and $[T_2] = [D_2]$ in $H^2(\cp^2,\R)$, so $T_1$ and $T_2$ are also different in cohomology.
\end{eg}

We can now finish characterizing currents in $\pcc^+(\torus)$ that extend trivially to toric compactifications.

\begin{thm}
\label{thm:lineargrowth}
Let $T\in\pcc^+(\torus)$ be a positive closed $(1,1)$ current on $\torus$ and $u\in\psh(\torus)$ be a potential for $T$.  Then $T$ extends trivially to some/any toric surface $X$, i.e. $T\in\pcci(\rzt)$, if and only if $\growth{\sfn_u}$ is finite.  
\end{thm}

\begin{proof}
For one direction, suppose that $\growth{\sfn_u}$ is finite.   If $\omega$ is a K\"ahler form invariant under the action of $\torus_\R$, then
\begin{align*}
\|T \wedge \omega \| = \|T_{ave} \wedge \omega\|.
\end{align*}
By Theorem \ref{thm:canonicalreps} $T_{ave}$ extends trivially to $X$.  That is, the right side of the equation is finite, and we infer using the Skoda-El Mir criterion that $T$ also extends trivially to $X$.

The other direction is more substantial.  Supposing $\growth{\sfn_u}$ is infinite, it will suffice to exhibit a smooth positive form $\omega$ on $X$ such that 
$T\wedge\omega$ has infinite mass on $\torus$.
Note that by adding an affine function of $\trop$ to the potential $u$
for $T$, we may assume that the convex function $\sfn = \sfn_u$ is non-negative
on $N_\R$ and satisfies $\sfn(0) = 0$.  Since the rays in $\Sigma_1(X)$ generate
$N_\R$ as a convex set and since $\sfn_u$ is convex, it follows that
$\growth{\sfn}<\infty$ if and only if $\frac{\sfn(\nvec)}{\norm{\nvec}}$ is
uniformly bounded on each $\tau\in\Sigma_1(X)$.   Therefore, we can assume
$\frac{\sfn(\nvec)}{\norm{\nvec}}$ is unbounded along some $\tau\in
\Sigma_1(X)$.

Let $\sigma \in \Sigma_2(X)$ be a sector adjacent to $\tau$ and $(x_1,x_2)$ be distinguished coordinates about $p_\sigma$ such that $C_\tau = \{x_1=0\}$.
In effect $\tau = \R\cdot (1,0)$, and our unboundedness assumption means $\lim_{\nvec_1\to\infty}\frac{\sfn(\nvec_1,0)}{\nvec_1} = \infty$.  Since $\sfn$ is convex and non-negative, we have more generally that $\lim_{\nvec_1\to\infty}\frac{\sfn(\nvec_1,\nvec_2)}{\nvec_1}= \infty$ uniformly in a neighborhood of $\nvec_2=0$.  
We will show that $T_{ave}\wedge dx_2\wedge
d\bar x_2$ has infinite mass near $(x_1,x_2) = (0,1)$.

By construction,
$$
T_{ave}\wedge \frac{dx_2\wedge d\bar{x_2}}{2i} = \frac{1}{\pi}\Delta_1(\sfn_u\circ\trop) \wedge \left(\frac{dx_2\wedge d\bar{x_2}}{2i}\right),
$$
where $\Delta_1$ denotes the distributional Laplacian with respect to $x_1$ (regarded as a measure in the $x_1$ coordinate).  Hence it suffices for each fixed value of $x_2$ to show that
$$
\int_{-\nvec_1<\log|x_1|<0} \Delta_1(\sfn_u\circ\trop)(x_1,x_2) 
$$
tends to infinity with $\nvec_1$.   This can be shown using the Poisson-Jensen formula (see e.g. \cite[Theorem A.1.3]{Sib99}) or by a polar coordinate computation, as follows.  The integral above is equal to 
$$
\int_0^{\nvec_1} \sfn_u''(t_1,\log|x_2|)\,dt_1 = \sfn_u'(\nvec^-_1,\log|x_2|) - \sfn_u'(0^+,\log|x_2|),
$$
where all derivatives are with respect to the first argument of $\sfn_u$ and the $\pm$ superscripts distinguish between left and righthand derivatives.  Since $\sfn_u$ is convex, we have the lower bound 
$$
\sfn_u'(\nvec_1^-,\log|x_2|) > \frac{\sfn_u(\nvec_1,\log|x_2|) - \sfn_u(0,\log|x_2|)}{\nvec_1} \
$$
which tends to $\infty$ with $\nvec_1$ by assumption.  Therefore the measure $T\wedge \frac{d x_2 \wedge d\bar x_2}{2i}$ has unbounded mass near $\{x_1=0\}$ as claimed.
\end{proof}

\subsection{Internal currents on toric surfaces revisited.}
  
The cone of positive internal currents $\pcci^+(\rzt)$ is not closed in $\pcc(\torus)$ since a sequence $(T_j)\subset\pcci^+(\rzt)$ can have a limit in $\pcc^+(\torus)$ with unbounded mass on a compactification $X$.    Even when the sequence converges on both $\torus$ and $X$, the limit in $X$ might dominate an external divisor and therefore exceed the limit in $\pcc^+(\torus)$.  We will see at the end of this section that things are better when one restricts attention to \emph{homogeneous} currents.

Theorems \ref{thm:canonicalreps} and \ref{thm:lineargrowth} make clear that $T_{ave},\bar T\in\pcci^+(\rzt)$ if and only if $T\in\pcci^+(\rzt)$.  Recall that a class $\alpha\in\hoo(X)$ is \emph{nef} if it has non-negative intersection with any curve $C\subset X$.

\begin{cor}
\label{cor:internalclass}
If $T\in\pcci^+(\rzt)$, then the cohomology class $\ch[X]{T}$ in any toric surface $X$ is nef, equal to that of $T_{ave}$ and $\bar T$.  There is, moreover, a unique function $\varphi_T\in\dpsh(\torus)$ such that
\begin{itemize}
\item $\int_{\mathbb{T}_\R} \varphi_T\,\twoform = 0$; and
\item $\varphi_T$ is integrable and satisfies $T = \bar T + dd^c_X \varphi_T$ on any toric surface $X$.
\end{itemize}
\end{cor}

\noindent In Corollary \ref{cor:internalclass} and elsewhere we implicitly identify functions $\varphi$ on $X$ with their pullbacks 
$\varphi\circ \pi_{YX}$ to any other toric surface $Y$.   %We denote this identification by $\varphi \equiv \varphi\circ \pi_{XY}$.

\begin{proof} 
Fix a toric surface $X$.  Since the action of $\torus$ on itself extends to an action on $X$ by automorphisms isotopic to the identity, we have that $T$ and $T_{ave}$ are cohomologous on $X$.  Meanwhile, Theorem \ref{thm:canonicalreps} implies that $T_{ave}$ and $\bar T$ are both cohomologous in $X$ to the class of the divisor $D_{u,X}$, which is nef since it has convex support function $\bar \sfn_{u,X}$.  The $dd^c$-lemma gives us $\varphi_T\in L^1(X)$ satisfying $dd^c\varphi_X = T-\bar T$.  Constants are the only pluriharmonic functions on the compact manifold $X$, so $\varphi_T$ is unique once normalized as in this corollary.

To see that $\varphi_T$ is independent of $X$, suppose $Y\succ X$.  Since $T - \bar T = dd^c_X \varphi_T$ on $X$, we have $\pi_{YX}^*
T-\pi_{YX}^*\bar T = dd^c_Y (\varphi_T \circ \pi_{YX})$ on $Y$.  On the other hand, if we identify $T$ and $\bar T$ with their trivial extensions to $Y$, we have that $T - \pi_{YX}^* T$ and $\bar T - \pi_{YX}^*\bar T$ are currents of integration supported on the poles of $Y$ that are contracted by $\pi_{YX}$.  As $T$ and $\bar{T}$ are also cohomologous in $Y$, we see that $(T-\bar T) - dd^c_Y (\varphi_T\circ \pi_{YX})\in\pcce(Y)$ is cohomologous to $0$ and supported on the exceptional set of $\pi_{YX}$.  However, any non-trivial divisor supported on $\exc(\pi_{XY})$ has negative self-intersection.  So in fact $T-\bar T = dd^c_Y
(\varphi_T \circ \pi_{YX}) \equiv dd^c_Y \varphi_T$.
\end{proof}

The homogenization $\bar T$ of a current $T\in\pcci^+(\torus)$ will serve as a
canonical, surface-independent representative for $[T]$.  However, local
potentials for $\bar T$ are typically unbounded near $\torus$-invariant points
of a given toric surface, so it is a little tricky to apply standard
compactness results from pluripotential theory to control the relative
potential $\varphi_T$ for $T-\bar T$.  The following result in this direction
will suffice for our purposes below.

\begin{thm}
\label{thm:volbd}
Given $M>0$ and a toric surface $X$ endowed with a smooth volume form, there exist constants $a,b>0$ such that for any current $T\in\pcci^+(\rzt)$ with a potential $u\in\psh(\torus)$ satisfying $\growth{\sfn_u}\leq M$, we have
$$ 
\vol_X\{|\varphi_T|\geq t\} \leq ae^{-bt}
$$
for all $t\geq 0$.
\end{thm}

\begin{lem}
\label{lem:mdom}
For any toric surface $X$ and $M\geq 0$, there exists a homogeneous internal
current $\boundingcurrent(M,X)$ with continuous local potentials everywhere on $X$ such
that $\boundingcurrent(M,X)~\geq~\bar T(X)$ for any $T\in\pcci^+(\rzt)$ with a support function
satisfying $\growth{\sfn_u}\leq M$. 
\end{lem}

\begin{proof}  Suppose $T\in\pcci^+(\rzt)$ is a current with a potential $u$ such that $\growth{\sfn_u} \leq M$.  Fix a ray $\tau\in\Sigma_1(X)$ and let $\sigma_1,\sigma_2\in\Sigma_2(X)$ denote the sectors adjacent to $\tau$ on either side.  Then we can choose unit vectors $\nvec_j\in\sigma_j$, $j=1,2$, whose average $\nvec := \frac{\nvec_1+\nvec_2}{2}$ generates $\tau$.  Convexity of $\sfn_{T,X}$ about $\tau$ means that the quantity 
$$
\Delta_{T,\tau}:=\sfn_{u,X}(\nvec_1) + \sfn_{u,X}(\nvec_2) - 2\sfn_{u,X}(\nvec)
$$
is non-negative.  Given $T$ and $X$, the convex function $\sfn_{u,X}$ is unique up to addition of affine functions.  One checks that this affine function has no effect on $\Delta_{T,\tau}$, which is therefore completely independent of the choice of potential $u$.  By convexity, we can therefore assume $\sfn_{u,X}\geq 0$ on $N_\R$, allowing that the upper bound on $\growth{\sfn_u}$ might increase to $2M$.  The same bound holds for $\growth{\sfn_{u,X}}$, so we obtain that
$$
\Delta_{T,\tau} \leq 2M + 2M - 0 = 4M
$$
independent of $T$ and $\tau$.  Hence there exists $T_\tau\in\pcci^+(\rzt)$ with potential $u_\tau\in\psh(\torus)$ satisfying $\growth{\sfn_{u_\tau}}\leq M$ such that 
$\Delta_{T_\tau,\tau}\leq 4M$ is maximal.  

We set $S = \sum_{\tau\in\Sigma_1(X)} T_\tau = dd^c (\sfn_v\circ\trop)$, where $v = \sum_{\tau\in\Sigma_1(X)} u_\tau$.  Note that $\growth{\sfn_v} \leq \#\Sigma_1(X)\cdot 4M <\infty$ so that $S\in\pcci^+(\rzt)$.  We claim that $\boundingcurrent(M,X) := \bar S(X)$ satisfies the conclusion of the lemma.  Indeed, by definition $\Delta_{\bar T(X),\tau} = \Delta_{T,\tau}$ for any $T\in\pcci^+(\rzt)$.  So if $T$ has a potential $u$ satisfying $\growth{u}\leq M$, we have by construction that for each $\tau\in\Sigma_1(X)$
$$
\Delta_{\bar S(X),\tau} = \Delta_{S,\tau} \geq \Delta_{T_\tau,\tau} \geq \Delta_{T,\tau} = \Delta_{\bar T(X),\tau},
$$
where the first inequality holds because $\Delta_{T_{\tau'},\tau}\geq 0$ even for rays $\tau'\in\Sigma_1(X)$ different from $\tau$.  We infer that $\bar\sfn_{v,X}-\bar\sfn_{u,X}$ is convex along every ray $\tau\in\Sigma_1(X)$ and therefore convex on all of $N_\R$, i.e. $\bar S(X) - \bar T(X)\geq 0$.
\end{proof}

\begin{proof}[Proof of Theorem \ref{thm:volbd}]
Given $u,T$ as in the theorem, the $dd^c$-lemma gives us a unique function $\varphi_{T,X}\in L^1(X)$ such that $T = \bar T(X) + dd^c_X \varphi_{T,X}$ and $\int_{\trop(x) = 0} \varphi_{T,X} \,\twoform = 0$.  
Since 
$$
T - \bar T(X) = (T-\bar T) + (\bar T - \bar T(X)) = dd^c_X (\varphi_T + \bar\sfn_u\circ\trop - \bar\sfn_{u,X}\circ\trop),
$$
the normalizations of $\varphi_T$ and $\varphi_{T,X}$ imply that 
$$
\varphi_T = \varphi_{T,X} + \bar\sfn_{u,X}\circ\trop - \bar\sfn_u\circ\trop. 
$$
everywhere on $X$.  By Propositions \ref{prop:sfnpinch} and
\ref{prop:integrability} the last two terms satisfy the desired volume
estimates.  Therefore it will suffice to establish the desired volume estimates
for $\varphi_{T,X}$ instead of $\varphi_T$.  The advantage to this is that
Lemma \ref{lem:mdom} tells us $dd^c_X  \varphi_{T,X} +
\boundingcurrent(M,X)\geq 0$; i.e. $\varphi_{T,X}$ is
$\boundingcurrent(M,X)$-plurisubharmonic, where $\boundingcurrent(M,X)$ is
independent of $T$ and has continuous local potentials on $X$.  

The following lemma plays a key role in completing the proof.

\begin{lem}
\label{lem:supbd}
$0\leq \sup \varphi_{T,X} \leq C$ for some constant $C$ depending on $M$ and $X$ but not $T$.
\end{lem}

\begin{proof}
The normalization of $\varphi_{T,X}$ immediately gives $0\leq \sup\varphi_{T,X}$.  For any fixed $T$, compactness of $X$ and the fact that local potentials for $\bar T(X)$ are continuous (Theorem \ref{thm:canonicalreps}) implies that $\sup\varphi_{T,X}<\infty$.  

Fix a $\torus_\R$-invariant volume form $dV$ on $X$ with $\int_X dV = 1$.  
A standard argument (see \cite[Chapter 8]{GZ_book}) using Hartog's compactness theorem for
families of subharmonic functions gives that for any fixed choice of $S \in \pcc^+(X)$ that has bounded local
potentials the set \begin{align*}
\{\phi \in L^1(X) \, : \, dd^c_X \phi \geq - S \quad \mbox{and} \quad \sup \phi = 0\}
\end{align*}
is compact.

Applying this in the case $S = \boundingcurrent(M,X)$ we find that $\norm[L^1(X)]{\varphi_{T,X}-\sup_X\varphi_{T,X}}~\leq~C$ for some constant $C$ independent of $T$, implying that
$$
\sup_X \varphi_{T,X} \leq C + \int_X \varphi_{T,X}\,dV.
$$
It suffices to show therefore that the integral on the right is non-positive.  Since the volume form is $\torus_\R$ invariant, it suffices to show for every $\nvec\in N_\R$ that the average value
$$
h(\nvec) := \int_{\trop(x) = \nvec} \varphi_{T,X} \,\twoform
$$
of $\varphi_{T,X}$ on the torus $\{\trop(x) = \nvec\}$ is non-positive.  But 
\begin{align*}
dd^c_X (h\circ \trop) = T_{ave} - \bar T(X) = dd^c_X(\sfn_u\circ\trop - \bar\sfn_{u,X}\circ\trop),
\end{align*}
giving that the functions $h$ and $\sfn_u - \bar \sfn_{u,X}$ differ by a constant.   Since $h(0) = \bar\sfn_u(0) = \bar\sfn_{u,X}(0)$, we obtain:
$$
h = \sfn_u - \sfn_u(0) - \bar\sfn_{u,X} \leq \bar\sfn_u-\bar\sfn_{u,X} \leq 0.
$$
\end{proof}

Using Lemma \ref{lem:supbd} we now complete the proof of Theorem \ref{thm:volbd}. 
It tells us that $\varphi_{T,X}$ ranges within a compact
family of $\boundingcurrent(M,X)$-plurisubharmonic functions.  Theorem \ref{thm:volbd}
now follows from work of Zeriahi (see \cite[Corollary 4.3]{Zer01}) which gives
such estimates uniformly for any compact family of plurisubharmonic functions
on a ball in~$\C^n$.
\end{proof}

We conclude this subsection with some further facts about homogeneous elements of $\pcc^+(\torus)$ and their support functions.

\begin{thm}
\label{thm:continuity}
Let $(\bar T_j)\subset \pcc^+(\torus)$ be a sequence of homogeneous currents and $\bar\sfn_j$ be support functions for $\bar T_j$ that converge pointwise to some function $\bar\sfn:N_\R\to\R$. Then 
\begin{itemize}
 \item the convergence is actually uniform on compact subsets;
 \item the limit $\bar\sfn$ is convex and positively homogeneous;
 \item $(\bar T_j)$ converges weakly on any toric surface $X$ to the homogeneous current $\bar T\in\pcc^+(\torus)$ with support function $\bar\sfn$.
\end{itemize}
Conversely, if $(\bar T_j)$ converges weakly to some current $\bar T\in\pcc^+(\torus)$, then $\bar T$ is internal and homogeneous, and the support functions $\sfn_{\bar T_j}$ converge uniformly on compact sets to $\sfn_{\bar T}$.
\end{thm}

\begin{proof}
Assume that the support functions $\bar\sfn_j$ converge pointwise to some limit function $\bar\sfn$.
Since $\bar\sfn_j(\nvec)\to \bar\sfn (\nvec)$ on any finite set of vectors $\nvec$ generating $N_\R$ as convex set, it follows from convexity and homogeneity that the sequence $(\bar\sfn_j)$ is bounded above uniformly on compact subsets of $N_\R$.  Since the limit $\bar\sfn$ is finite, it follows again from convexity of $\bar\sfn_j$ that the convergence is uniform and $\bar\sfn$ is convex and positively homogeneous.  Thus $\bar\sfn$ is a support function for some homogeneous current $\bar T \in \pcc^+(\torus)$.  On any fixed toric surface $X$, Theorem \ref{thm:canonicalreps} tells us that $dd^c_X(\bar\sfn_j\circ\trop) = \bar T_j - D_j$ and $dd^c_X(\bar\sfn\circ\trop) = \bar T - D$, where $D_j,D\in\pcce(X)$ have support functions determined by the restrictions of $\bar\sfn_j,\sfn$ to rays in $\Sigma_1(X)$.  Uniform local convergence $\bar\sfn_j\to\bar\sfn$ and homogeneity imply that $D_j\to D$.  Homogeneity of support functions further implies that $\bar\sfn_j\circ\trop\to\bar\sfn\circ\trop$ in $L^1(X)$.  Weak continuity of $dd^c_X$ then gives $\bar T_j-D_j \to \bar T-D$.  Hence $\bar T_j\to \bar T$.

Now assume instead that $\bar T_j\to \bar T$ weakly in $\pcc(\torus)$.  Then there exist plurisubharmonic potentials $u_j,u\in\psh(\torus)$ for $\bar T_j$ and $\bar T$ such that $u_j\to u$ in $L^1_{loc}(\torus)$.  Since $\bar T_j$ and $\bar T$ are invariant under the action of $\torus_\R$, we may assume that $u_j$ and $u$ are, too.  That is, $u_j =\tilde\sfn_j\circ\trop$ and $u=\tilde\sfn\circ\trop$ for convex functions $\tilde\sfn_j,\tilde\sfn:N_\R\to \R$.  Convergence $u_j\to u$ in $L^1_{loc}$ implies that $\tilde\sfn_j\to\tilde\sfn$ uniformly locally.  Since two support functions for the same internal current differ by an affine function, it follows that the $\tilde\sfn_j$ are all positively homogeneous.  Hence we are back in the context of the arguments in the first paragraph of this proof.
\end{proof}

\begin{cor} If $(\bar T_j)\subset\pcc^+(\torus)$ is a weakly convergent sequence of homogeneous currents, then $\growth{\bar\sfn_{T_j}}$ is bounded uniformly in $j$.
\end{cor}

\begin{cor}\label{COR:CLOSED_SETS_POSITIVE_HOMOG_CURR} For any toric surface $X$, the set of positive homogeneous currents is closed in $\pcc^+(X)$ and therefore also in $\pcc^+(\torus)$.
\end{cor}

\section{Toric currents and cohomology classes}
\label{sec:tcurrents}

The papers \cite{BFJ08} and \cite{Can11} pioneered the idea that to better understand dynamics of a rational surface map $f:X\tto X$, one should consider not just the pullback action $f^*$ on $\hoo(X)$ but simultaneously also the actions on $\hoo(Y)$ for all surfaces $Y$ obtained by blowing up sequences of points in $X$.  In this section and the next, we imitate the inverse limit constructions from those papers for both $(1,1)$ cohomology classes and currents, but instead of considering all possible surfaces obtained by blowups from e.g. $\cp^2$, we restrict attention to toric surfaces only.

If $Z\succ Y\succ X$ are all toric surfaces, then we have $\pi_{Z X*} = \pi_{Y X*}\pi_{Z Y*}$ and $\pi_{Z X}^* = \pi_{Z Y}^*\pi_{Y X}^*$ for both cohomology classes and currents.  Hence the following makes sense.

\begin{defn}
A \emph{toric (Weil) class} is a collection 
$$
\alpha = \{\alpha_X \in \hoo(X):X\text{ is a toric surface}\}
$$ 
satisfying $\pi_{YX*} \alpha_Y = \alpha_X$ whenever $Y\succ X$.
Similarly, a \emph{toric current} is a collection $T =\{ T_X\in\pcc(X)\}$, also indexed by toric surfaces and subject to the same compatibility.  
\end{defn}

\noindent
We call $\alpha_X$ and $T_X$ the \emph{incarnations} of $\alpha$ and $T$ on $X$, letting $\hoo(\rzt)$ denote the $\R$-vector space of all toric classes and $\pcc(\rzt)$ the space of all toric currents.
To each $T\in\pcc(\rzt)$ we associate the class $\ch{T}\in\hoo(\rzt)$ whose incarnation on any toric surface $X$ is the class $\ch[X]{T_X}\in\hoo(X)$.  This is well-defined since pushforward by rational maps preserves cohomological equivalence.  We will say that a toric current $T$ (and by extension its class $\ch{T}$) is \emph{positive} if all its incarnations $T_X$ are positive (respectively), writing $T\geq S$ if $T-S$ is positive.  Similarly a class $\alpha\in \hoo(\rzt)$ (and by extension, any $T\in\pcc(\rzt)$ representing $\alpha$) is \emph{nef} if all its incarnations are nef.

Recall that $\pcci(\rzt)$ was defined in the paragraph after the proof of
Proposition \ref{prop:internalcurrents}.  As explained there, any
internal current $T\in\pcci(\rzt)$ is a toric current with incarnations $T_X
\in~\pcci(X)$ obtained by trivial extension of $T$.  By Corollary
\ref{cor:internalclass} positive internal currents are nef.
A pole $C_\tau\subset\rzt$ is a toric current, with incarnations $C_{\tau,X}$
equal to either $C_\tau$ or~$0$ depending on whether or not
$\tau\in\Sigma_1(X)$.  In particular $C_\tau$ is positive but not nef, since in a sufficiently dominant $X$ one has $C_\tau^2 < 0$.  

More generally, we call a toric current $D$ an \emph{external divisor} if $D_X\in \pcce(X)$ for each toric surface $X$.  Formally, we have
$$
D = \sum c_\tau C_\tau,
$$
where the sum is over all rational rays $\tau\subset N_\R$, and then on $X$ we have $D_X = \sum_{\tau\in\Sigma_1(X)} c_\tau C_\tau$.  Let $\pcce(\rzt)$ denote the set of all external divisors.  
The external/internal decomposition in Proposition \ref{prop:eidecomp1} is invariant under pushforward by birational morphisms $\pi_{X Y}$, so  we obtain the same decomposition for toric currents, i.e. 
$
\pcc(\rzt) = \pcce(\rzt)\oplus \pcci(\rzt).
$

We extend the notion of support function $\sfn_D$ from external divisors on a toric surface $X$ to any external divisor $D\in\pcce(\rzt)$ by declaring that $\sfn_D$ is the pointwise limit of the support functions $\sfn_{D_X}$ for incarnations.  In general, $\sfn_D$ is defined only on rational rays and not all of $N_\R$, but things are better when $D$ is nef.  

\begin{prop}
\label{prop:nefsupport}
Any $\alpha\in\hoo(\rzt)$ is uniquely represented by an external divisor $D = \sum c_\tau C_\tau$, normalized by the condition that the coefficients $c_\tau$ are the same for all $\tau\in\Sigma_1(\cp^2)$.  Moreover, for any $D\in \pcce(\rzt)$, we have
\begin{itemize}
 \item $D$ is nef if and only if $\sfn_D$ extends continuously to a convex function on all of $N_\R$;
 \item $\ch{D} = 0$ if and only if $\sfn_D:N\to \R$ is linear;
\end{itemize} 
\end{prop}

\begin{proof}
For each toric surface $X$, let $D_X\in\pcce(X)$ be an external divisor representing $\alpha_X$.  If $X\succ \cp^2$, then we can add a linear function to $\sfn_{D_X}$ to normalize so that $\sfn_{D_X}(v_\tau)$ is the same for each of the three primitive vectors $v_\tau$ that generate rays in $\Sigma_1(\cp^2)$.  Then $D_X$ is uniquely determined by its class in $\hoo(X)$.  If 
$Y\succ X$ is another toric surface, then $\pi_{YX*} D_Y$ is also a normalized representative of $\alpha_X$ and therefore equal to $D_X$.  Hence the external divisor $D = \{D_X\}\in \pcce(\rzt)$ represents $\alpha$.  That $\sfn_D$ is linear if and only if $\alpha = 0$ and convex if and only if $\alpha$ is nef follow from the corresponding facts for support functions for external divisors on a fixed toric surface.
\end{proof}

We endow both $\hoo(\rzt)$ and $\pcc(\rzt)$ with the product topology, declaring that $T_j\to T\in\pcc(\rzt)$ if $T_{j,X}\to T_X\in\pcc(X)$ for each toric surface $X$.  The assignment $T\mapsto \ch{T}$ of currents to classes is then continuous.
As is the case on individual surfaces $X$, external divisors constitute a closed subspace of $\pcc(\rzt)$, but internal currents do not.  So when a sequence $(T_j)\subset\pcc(\rzt)$ converges, the internal and external components of the terms $T_j$ need not converge individually.  

\begin{defn} Given $S,T\in\pcc(\rzt)$ and a toric surface $X$, we say that $S$ and $T$ are \emph{cohomologous in $X$} if $\ch[X]{S} = \ch[X]{T} \in\hoo(X)$.  We say that $S$ and $T$ are \emph{completely cohomologous} if this is true for every $X$, i.e. if $\ch{S}=\ch{T}$.
\end{defn}

By the $dd^c$-lemma, $S$ and $T$ are cohomologous in $X$ if and only if $S_X-T_X = dd^c\varphi_X$ for some $\varphi_X\in \dpsh(X)$.  When $S$ and $T$ are completely cohomologous, one can use essentially the same potential for $S-T$ in any other toric surface:

\begin{prop}
\label{prop:tamepotential}
If $T\in\pcc(\rzt)$ is completely cohomologous to $0$, then there is a function $\varphi\in\dpsh(\torus)$ such that for every toric surface $X$, we have $\varphi\in\dpsh(X)$ and $T_X = dd^c_X \varphi$. 
\end{prop}

\noindent Note for purposes of integration that since $X\setminus \torus$ is a measure zero subset of $X$, the extension of $\varphi$ to $X$ is essentially unique. We will call $\varphi$ a potential for $T$.  It is, by compactness, unique up to additive constant. 

\begin{proof}
This is similar to the proof of Corollary \ref{cor:internalclass}.  Suppose $Y\succ X$ and let $\pi = \pi_{Y X}$.  Let $\varphi_X\in L^1(X)$ be a potential for $T_X$.  Then on the one hand $\varphi_X\circ\pi$ is a potential for $\pi^* T_X$.  One the other hand, since $\pi_*\pi^* = \id$, we have that $\pi^* T_X - T_Y$ is a divisor supported on $\exc(\pi)$ and cohomologous to $0$.  The only such divisor is trivial, so $T_Y = \pi^* T_X$.  Hence $T_Y := dd^c_Y \varphi$ where $\varphi := \varphi_X|_\torus$.
\end{proof}

The following reformulation of Theorem \ref{thm:canonicalreps} is now immediate from definitions.

\begin{thm}
\label{thm:canonicalreps2}
Any positive internal current $T\in\pcci(\rzt)$ is completely cohomologous to its
homogenization $\bar T$.  Moreover, if $\bar\sfn$ is a support function for
$\bar T$, and $D\in\pcce(\rzt)$ is the external divisor with support function $\bar\sfn$,
then $T$ and $D$ are also completely cohomologous and $\bar\sfn\circ\trop$ is a
potential for $\bar T - D$. 
\end{thm}

%\begin{cor}
%\label{cor:homoghomeo}
%The map $\pcci(\rzt) \to \hoo(\rzt)$ carrying toric currents to their cohomology classes restricts to a homeomorphism from the set of positive homogeneous internal currents onto the nef cone in $\hoo(\rzt)$.
%\end{cor}

We can now prove Theorem \ref{thm:mainB}, which was stated in the introduction.

\begin{proof}[Proof of Theorem \ref{thm:mainB}]
If $\bar T\in\pcci^+(\rzt)$ is homogeneous, then
$\ch{\bar T}\in\hoo(\rzt)$ is nef by Corollary~\ref{cor:internalclass}.
If instead $\alpha\in\hoo(\rzt)$ is nef, represented by $D\in\pcce(\rzt)$, then the support function $\sfn_D$
for $D$ is convex by Proposition \ref{prop:nefsupport}, and Theorem
\ref{thm:canonicalreps2} tells us that $D$ is completely cohomologous to the
positive homogeneous current $\bar T$ with the same support function.  That is,
$\ch{\bar T} = \alpha$. 

If $\bar{T}'\in\pcci^+(\rzt)$ is a another homogeneous current with $\ch{\bar
T'} = \alpha$, and $\sfn$ is a support function for $\bar{T}'$, then the
external divisor $D'$ with the same support function is cohomologous to
$\bar{T}'$ and therefore also $D$.  It follows that $\sfn$ differs from
$\sfn_D$ by a linear function $h:N_\R\to\R$.  Hence $\bar T - \bar T' =
dd^c_\torus (h\circ\trop) = 0$ on $\torus$, and therefore also on any toric
surface, since they are internal currents.  This proves that the continuous map
$\bar T \mapsto \ch{\bar T}$ from positive homogeneous currents to toric
classes is a bijection.

It remains to see that the inverse, carrying a nef class $\alpha\in\hoo(\rzt)$
to its unique positive homogeneous representative $\bar T_\alpha$, is also
continuous.  The map $\alpha\mapsto D_\alpha\in\pcce(\rzt)$ sending $\alpha$ to
its normalized representative $D_\alpha = \sum c_{\tau,\alpha}
C_\tau\in\pcce(\rzt)$ has the property that for each $\tau \in N$ the
coefficient $c_{\tau,\alpha}$ depends continuously on $\alpha$.  Therefore the
corresponding support functions $\bar \psi_\alpha$ depend continuously on
$\alpha$ and Theorem \ref{thm:continuity} implies that $T_\alpha = dd^c
(\bar \psi_\alpha \circ \trop)$ depends continuously on $\alpha$ as well.
\end{proof}

%But this follows from Theorem \ref{thm:continuity} and
%coefficient-wise continuity \textcolor{red}{(Meaning continuity when restricted
%to any specific surface $X$?)} of the map $\alpha\mapsto
%D_\alpha\in\pcce(\rzt)$ sending $\alpha$ to its normalized representative
%$D_\alpha = \sum c_{\tau,\alpha} C_\tau\in\pcce(\rzt)$.  \end{proof}

\begin{rem}
\label{REM:UNIQUE_HOMOGENEOUS_CURRENT}
To reiterate, Theorem \ref{thm:mainB} says that any nef $\alpha \in \hoo(\rzt)$ is represented by a unique homogeneous current $\bar{T} \in \pcci^+(\rzt)$; whereas for any fixed toric surface $X$, each nef class $\alpha_X \in \hoo(X)$ is represented by many different homogeneous currents.
\end{rem}

Given a nef class $\alpha_X\in\hoo(X)$ on a particular toric surface, the set of $T\in\pcc^+(X)$ representing $\alpha_X$ is compact.  So Theorems \ref{thm:mainB}, \ref{thm:continuity} and \ref{thm:canonicalreps2} yield the following.

\begin{cor} 
\label{cor:compactgrowth}
Let $\mathcal{T}\subset\pcci^+(\rzt)$ be a family of toric currents and $\mathcal{H}\subset\hoo(\rzt)$ be the corresponding family of toric classes.  Then the following are equivalent.
\begin{itemize}
 \item $\mathcal{T}$ is precompact in $\pcc^+(\rzt)$.
 \item $\mathcal{H}$ is precompact in $\hoo(\rzt)$.
 \item There exists $M\geq 0$ such that $\growth{\bar\sfn_T} \leq M$ for all $T\in\mathcal{T}$.
\end{itemize}
\end{cor}

% Tameness stuff.  Almost expendable.

Representatives of a given toric class/current are, by definition, compatible under pushforward $\pi_{YX*}$ for $Y\succ X$.  Compatibility under pullback is a much stronger condition.

\begin{defn}
A class $\alpha\in\hoo(\rzt)$ is \emph{Cartier} if it is \emph{determined} on some particular toric surface $X$, i.e. $\alpha_Y = \pi_{YX}^*\alpha_X$ for all $Y\succ X$.  Likewise, $T\in \pcc(\rzt)$ is \emph{tame}, \emph{determined} on $X$, if $T_Y = \pi_{Y X}^*T_X$, for all $Y\succ X$.  
\end{defn}

By Proposition \ref{prop:tamepotential} a class $\alpha\in\hoo(\rzt)$ is Cartier, determined on $X$, if and only if some/any $T\in\pcc(\rzt)$ representing $\alpha$ is tame and determined on $X$.  Poles $C_\tau\in \rzt$ are never tame.  For internal currents, on the other hand, we have the following criterion. 

\begin{prop}
\label{prop:internaltame}
The following are equivalent for an internal current $T\in\pcci^+(\rzt)$ and toric surface $X$.  
\begin{itemize}
 \item $T$ is tame, determined on $X$.
 \item $\nu(T,p_\sigma) = 0$ for each $\torus$-invariant point $p_\sigma\in X$.
 \item $\pi_{YX}^* T_X$ is internal on $Y$ for all $Y\succ X$.
 \item $\sfn_{\bar T}$ is linear on each sector $\sigma\in\Sigma_2(X)$.
\end{itemize}
\end{prop}

\begin{proof}

Given $Y\succ X$, suppose that $C_\tau\subset Y$ is a pole contracted by
$\pi_{YX}$ to the $\torus$-invariant point $p_\sigma \in X$.  Then
Lemma \ref{lem:favrebd}
tells us that $\pi_{YX}^* T_X$ charges $C_\tau$ if and only if
$\nu(T_X,p_\sigma) \neq 0$.  In particular, $T_Y = \pi_{YX}^* T_X$ if and only if
$\nu(T_X,p_\sigma) = 0$ for the image of each pole contracted by $\pi_{YX}$.
Equivalence of the first three criteria follows.

For equivalence with the fourth criterion, suppose that $T$ is tame, determined by $T_X\in\hoo(X)$.  Since
$T_X$ and $\bar T(X)$ are cohomologous in $\hoo(X)$, it follows that the homogeneous internal toric current determined by $\bar T(X)$ is completely cohomologous to $T$.  On the other hand, there is only one homogeneous internal current completely cohomologous to $T$, so $\bar T = \bar T(X)$.  Reversing this logic, we infer that $\bar T = \bar T(X)$ implies that $T$ is tame and determined in $X$.
\end{proof}

\begin{eg}
By Propositions \ref{prop:internalcurve} and \ref{prop:internaltame} an internal curve $C\subset\rzt$ is always tame, determined in any toric surface $X$ that fully realizes $C$.  On the other hand, the homogeneous current $T\in\pcci^+(\rzt)$ with support function $\norm{\cdot}$ is not tame.  In fact, one can check on any toric surface $X$ that $\nu(T,p_\sigma) > 0$ for all $\sigma\in\Sigma_2(X)$.
\end{eg}

For any toric surface $X$, we have an inclusion $\pcc(X)\hookrightarrow \pcc(\rzt)$ as a set of tame currents.  That is, given $T_X\in \pcc(X)$ and $X'$ is another surface, we can choose a third surface $Y$ dominating both $X$ and $X'$ and obtain the incarnation of $T_X$ on $X'$ from $T_{X'} = \pi_{X' X}^* T_X := \pi_{Y X'*} \pi_{Y X}^* T_X$.   If all currents $T_j$ in a given sequence are tame, determined on the same surface $X$, then $(T_j)$ converges as a sequence of toric currents if and only if $(T_{j,X})$ converges in $\pcc(X)$.  Hence the inclusion $\pcc(X)\hookrightarrow\pcc(\rzt)$ is a homeomorphism onto its image.  One should note, however, that since pullbacks of internal currents need not be internal, this inclusion does not respect the decomposition of a toric current into internal and external components.

\subsection{Intersection and Lelong numbers for toric classes and currents}
Following \cite{BFJ08, Can11} we introduce a quadratic form on much of $\hoo(\rzt)$ that simultaneously extends the intersection form on $\hoo(X)$ for all toric surfaces $X$.  Two key observations make this possible.

First, we have a well-defined intersection form on the subspace of $\hoo(\rzt)$ consisting of Cartier classes: if $\alpha\in\hoo(\rzt)$ is determined on some surface $X$, then $(\alpha_Y,\beta_Y)_Y = (\alpha_X,\beta_X)_X$ for any other toric class $\beta$ and any surface $Y\succ X$.  If $\alpha_0$ is the Cartier class determined by a line in $\cp^2$, then the intersection form is negative definite on the orthogonal complement $\ker \pi_{X\cp^2 *}$ of $\alpha_0$ in any surface $X\succ\cp^2$.  So we may define a norm on Cartier classes 
$\alpha$ by decomposing $\alpha = c\alpha_0 + \alpha_1$ into components parallel and orthogonal to $\alpha_0$
and declaring
$$
\norm{\alpha}^2 = c^2 - \alpha_1^2.
$$

Second, Theorem \ref{thm:pushpull1} implies $\alpha_X^2 = (\pi_{YX*} \alpha_Y)^2 \geq \alpha_Y^2$ for any $\alpha\in\hoo(\rzt)$ and any toric surfaces $Y\succ X$.  Let 
$$
\eltwo(\rzt) := \{\alpha\in\hoo(\rzt):\inf_X \alpha_X^2 > - \infty\}.
$$
In particular $\eltwo(\rzt)$ includes all Cartier and all nef classes.  The following theorem summarizes basic results whose arguments are given in Sections 1.4 and 1.5 of \cite{BFJ08}.

\begin{thm}
\label{thm:intersection}
$\eltwo(\rzt) \subset \hoo(\rzt)$ is (canonically isomorphic to) the Cauchy completion of the Cartier classes with respect to the above norm.  Moreover, the intersection form on Cartier classes has a unique extension to $\eltwo(\rzt)$, and the following hold for all $\alpha,\beta\in\eltwo(\rzt)$.
\begin{itemize}
  \item $\alpha^2 = \inf_X \alpha_X^2$.
  \item if $\alpha$ is Cartier, determined on $X$, then $(\alpha\cdot\beta) = (\alpha_X\cdot \beta_X)_X$.
  \item if $\alpha,\beta\in\eltwo(\rzt)$ are non-zero classes satisfying $\alpha^2 \geq 0$ and $(\beta\cdot\alpha) = 0$ then $\beta^2 \leq 0$; equality holds if and only if $\alpha$ and $\beta$ are multiples of each other.
  \item $\alpha\in \hoo(\rzt)$ is nef if and only if $(\alpha\cdot \beta) \geq 0$ for all positive and Cartier $\beta$.
\end{itemize}
\end{thm}

The norm topology on $\eltwo(\rzt)$ is strictly stronger than the (restriction) of the weak topology on $\hoo(\rzt)$ since, among other things, Cartier classes are dense in the weak topology on $\hoo(\rzt)$, whereas $\eltwo(\rzt)$ is a proper subset of $\hoo(\rzt)$.  On the other hand, if $\alpha = \alpha_X\in \hoo(X)\subset\eltwo(\rzt)$ is Cartier, then the linear functional $\beta\mapsto (\alpha\cdot\beta)$ on $\eltwo(\rzt)$ extends continuously to \emph{all} of $\hoo(\rzt)$ via $(\alpha\cdot\beta) := (\alpha_X\cdot\beta_X)$.  Finally, we point out that unlike the situation in \cite{BFJ08}, the (indefinite) Hilbert space $\eltwo(\rzt)$ is separable, because there are only countably many different toric surfaces.  

% The following comparison result for nef classes is established in \cite{BFJ08}.  The argument uses the Hodge Index Theorem together with the fact that classes with positive self-intersection are (up to sign) effective, and it remains valid in the present context.
% 
% \mpar{Don't think we really need this}
% \begin{thm}
% For any nef $\alpha,\beta\in\hoo(\rzt)$, we have $2(\alpha\cdot\beta)\alpha \geq (\alpha^2) \,\beta$
% \end{thm}

If $T\in\pcc^+(\rzt)$ and $p\in\rzt^\circ$ is a realizable point, then we define the \emph{Lelong number} $\nu(T,p)$ of $T$ at $p$ to be the Lelong number $\nu(T_X,p)$ of the incarnation $T_X$ of $T$ in some (and hence any) toric surface $X$ that realizes $p$.  Lelong numbers are, among other things, always non-negative.  

\begin{lem} 
\label{lem:nefcurrent}
Suppose $X$ is a compact K\"ahler surface, $C\subset X$ is a curve and $T\in\pcc^+(X)$ is a current that does not charge $C$.  Then 
$$
(T\cdot C) \geq \sum_{p\in C} \nu(T,p).
$$
\end{lem}

\begin{proof}
We first show that $(T\cdot C)\geq 0$.  By decomposing, we reduce to two cases: $T$ does not charge curves, and $T$ is the current of integration over a curve $C'\neq C$.  In the first case Demailly (see Corollary 6.4 in \cite{Dem92}), showed that the cohomology class of $T$ is nef, and in the second $C'$ and $C$ meet transversely.  Either way $(T\cdot C) \geq 0$.

Now consider the blowup $\pi:\hat X\to X$ of $X$ at $p\in C$. Then pulling back local potentials gives
$$
\pi^* T = T' + \nu(T,p) E,
$$
where $E$ is the curve contracted by $\pi$ and $T'$ is the `strict transform' of $T$ by $\pi$, i.e.\ the unique positive closed current with no mass on $E$ such that $\pi_* T' = T$.  So if $C'\subset \hat X$ denotes the strict transform of $C$, we get
$$
T\cdot C = (T\cdot \pi_* C') = (\pi^* T \cdot C') 
\geq (T'\cdot C') + \nu(T,p) \geq (T'\cdot C')
$$
since $(E\cdot C')\geq 1$.  The argument from the first paragraph tells us again that $(T'\cdot C') \geq 0$.  Hence repeating the estimate for $T'$, etc, finishes the proof.
\end{proof}

\begin{thm}
\label{thm:lelongbd}
Given $T\in\pcci^+(\rzt)$, there exists $M>0$ depending only on the cohomology class of $T$ such that for sufficiently dominant toric surfaces $Y$,
$$
\sum_{p\in Y\setminus \torus} \nu(T_Y,p) < M.
$$
\end{thm}

\begin{proof}
Fix a toric surface $X$.  The canonical divisor on $X$ is given by
$$
K_X = -\sum_{\tau\in\Sigma_1(X)} C_\tau,
$$
so Lemma \ref{lem:nefcurrent} tells us
$$
(-K_X\cdot T_X) \geq \sum_{\tau\in\Sigma_1(X)}\sum_{p\in C_\tau} \nu(T_X,C_\tau).
$$
If $\pi:Y\to X$ is the blowup of $X$ at some $\torus$-invariant point $p$, and $C = \pi^{-1}(p)$ is the new pole, then we have
$K_Y = \pi^*K_X + C$ (by adjunction) and $T_Y = \pi^* T_X - \nu(T,p) C$ (by pulling back local potentials).  Therefore:
%\textcolor{blue}{When computing $-K_Y\cdot T_Y$ using the above formulae, the cross terms cancel due to the projection formula and we find:}
$$
0 \leq (-K_Y\cdot T_Y) = (-K_X\cdot T_X) -\nu(T_X,p) \leq (-K_X\cdot T_X).
$$
Hence for any $Y\succ X$ we have
$$
\sum_{p\in Y\setminus\torus} \nu(T_Y,p) \leq M:=(-K_X\cdot T_X). 
$$  
\end{proof}

\subsection{Relationship with \cite{BFJ08}}

Besides the fact that we consider currents as well as classes, the main
difference between our context and that of \cite{BFJ08} is that, for us, a
toric (Weil) class is given by an incarnation on every toric surface $X$,
whereas a Weil class in \cite{BFJ08} has an incarnation on every blowup of
some fixed base surface.  Say for the sake of discussion the base surface is
$\cp^2$, and let $W(\cp^2)$ denote the set of all Weil classes in the sense of
\cite{BFJ08}.

One sees easily that any $\tilde\alpha\in W(\cp^2)$ induces a class $\rs(\tilde\alpha) \in \hoo(\rzt)$ obtained by `restriction'.  That is, we set $\rs(\tilde\alpha)_X = \tilde \alpha_X$ whenever $X\succ \cp^2$ is a toric surface dominating $\cp^2$.  For any other toric surface $Y$, we choose a toric surface $X$ that dominates both $Y$ and $\cp^2$ and set $\rs(\alpha)_Y = \pi_{XY*}\tilde\alpha_X$.  The resulting restriction map $\rs:W(\cp^2) \to \hoo(\rzt)$ is linear, surjective and continuous.  It preserves both positive and nef classes.

In order to apply the main results of \cite{BFJ08}, it will be important for us that the restriction map admits a distinguished section $\ex:\hoo(\rzt) \hookrightarrow W(\cp^2)$.  This depends on the following observation that any blowup of $\cp^2$ dominates a maximal toric surface.  

\begin{prop} If $\pi_{\tilde X}:\tilde X\to\cp^2$ is a birational morphism from
a (non-toric) surface ${\tilde X}$ onto $\cp^2$, then there is a toric surface $X$ and a
birational morphism $\pi_{\tilde XX}:\tilde X\to X$ such that any other
birational morphism $\pi_{\tilde X Y}:\tilde X\to Y$ onto a toric surface $Y$
factors through $X$, i.e. $\pi_{\tilde XY} = \pi_{XY}\circ\pi_{\tilde XX}$
where $\pi_{XY}:X\to Y$ is a morphism.
\end{prop}

\begin{proof} This is a special case of the second conclusion of Corollary 5.5 in \cite{DiLi16}.
Factoring $\pi_{\tilde X}=\tau_k\circ\dots\circ\tau_1$ into blowups $\tau_j:X_j\to X_{j-1}$, the existence of $X$ amounts to noting that the factors can be ordered so that first $k'$ among them are `toric', centered at torus invariant points on the surface created by the previous blowups, and the rest are centered at non-torus invariant points appearing in the surface $X=X_{k'}$.
\end{proof}

Given $\alpha\in\hoo(\rzt)$ we may now define the `extension' $\ex(\alpha) \in W(\cp^2)$ by declaring for any blowup $\pi_{\tilde X}:\tilde X\to \cp^2$ that $\ex(\alpha)_{\tilde X} := \pi_{\tilde X X}^*\alpha_X$, where $X\prec \tilde X$ is the maximal toric surface dominated by $\tilde X$.  One checks that since the incarnations $\alpha_X$ of $\alpha$ are consistent across toric surfaces, the incarnations $\ex(\alpha)_{\tilde X}$ are consistent across blowups of $\cp^2$.  It is immediate that $\ex:\hoo(\rzt)\to W(\cp^2)$ is continuous and linear, satisfies $\rs \circ \ex = \id$, and preserves positive and nef classes.  

As with $\eltwo(\rzt)$, we can orthogonally decompose elements $\tilde\alpha = c\tilde \alpha_0 + \tilde\alpha_1\in L^2$ relative to the Cartier class $\tilde\alpha_0$ determined by a line in $\cp^2$, and the strong topology on $L^2$ is then given by the norm defined by
$$
\norm{\tilde\alpha} = c^2 - (\tilde\alpha_1)^2.
$$
Since $\cp^2$ is itself a toric surface, we have $\ex(\alpha_0) = \tilde \alpha_0$ and 
$\rs(\tilde\alpha_0) = \rs(\ex(\alpha_0)) = \alpha_0$.  

\begin{prop}
\label{prop:comparison}
Extension of toric classes restricts to a (strongly) continuous map $\ex:\eltwo(\rzt)\to L^2$ that preserves both the intersection products and the associated norms.  Restriction of Weil classes restricts to a continuous map $\rs:L^2\to\eltwo(\rzt)$ that is non-increasing for both self-intersections and norms.  Finally, we have for any $\alpha\in \eltwo(\rzt)$ and $\tilde\beta \in L^2$ that
$$
(\alpha\cdot \rs(\tilde\beta)) = (\ex(\alpha)\cdot\tilde\beta).
$$
\end{prop}

\begin{proof}
Given classes $\alpha, \beta \in \eltwo(\rzt)$, let $\tilde X\to \cp^2$ be a blowup and $X\prec\tilde X$ the maximal toric surface dominated by $\tilde X$.  Then since pullbacks by birational morphisms preserve intersections, we have 
$$
(\ex(\alpha)_{\tilde X})^2
= 
(\pi_{\tilde X X}^* \alpha_X)^2
=
\alpha_X^2.
$$
Hence $\ex(\alpha) \in L^2$ with $\ex(\alpha)^2 = \alpha^2$.  Similarly, if $\beta\in \eltwo(\rzt)$ we have
$(\ex(\alpha)\cdot\ex(\beta)) = \alpha\cdot\beta$, so that $\ex$ is an isometry with respect to intersection.
It follows that if $\alpha = c\alpha_0 + \alpha_1$ is the orthogonal decomposition of $\alpha$ relative to $\alpha_0$, then $\ex(\alpha) = c\tilde\alpha_0 + \ex(\alpha_1)$ is the orthogonal decomposition relative to $\tilde\alpha_0$.  Hence $\norm{\ex(\alpha)} = \norm{\alpha}$.  In particular $\ex$ is continuous in the norm topology.

Concerning the map $\rs$, we have for any $\tilde\alpha\in L^2$ that 
$$
\rs(\tilde\alpha)^2 = \inf_X \tilde \alpha_X^2 \geq \inf_{\tilde X} \tilde\alpha_{\tilde X}^2 = \tilde \alpha^2
$$
where the first infimum is over all toric surfaces, the second is over all blowups of $\cp^2$, and the inequality holds because any toric surface $X$ is dominated by a toric surface that also dominates $\cp^2$.  In particular, $\rs(L^2)\subset \eltwo(\rzt)$.  

Now if $\alpha\in\eltwo(\rzt)$, $\tilde\beta\in L^2$, $\tilde X$ is a blowup of $\cp^2$ and $X$ is the maximal toric surface dominated by $X$, we obtain from the projection formula that
$$
(\alpha_X\cdot \rs(\tilde\beta)_X) = (\alpha_X\cdot \pi_{\tilde X X*}\tilde\beta_{\tilde X}) = (\pi_{\tilde XX}^*\alpha_X\cdot \tilde\beta_{\tilde X}) = (\ex(\alpha)_{\tilde X}\cdot \tilde\beta_{\tilde X}).
$$
Hence $(\alpha\cdot\rs(\tilde\beta)) = (\ex(\alpha)\cdot \tilde\beta)$.

It follows that if $\tilde\alpha = c\tilde\alpha_0+\tilde\alpha_1\in L^2$ is the orthogonal decomposition relative to $\tilde\alpha_0$, then $\rs(\tilde\alpha) = c\alpha_0 + \rs(\tilde\alpha_1)$ is the orthogonal decomposition relative to $\alpha_0$.  Hence
$$
\norm{\rs(\tilde\alpha)} = c^2 - \rs(\tilde \alpha_1)^2 \leq c^2 - \tilde \alpha_1^2 = \norm{\tilde\alpha}.
$$
\end{proof}

\section{Pushforward and pullback by toric maps}
\label{sec:actions}

In this section, still following \cite{BFJ08}, we define and investigate natural pullback and pushforward actions $\rzf^*,\rzf_*$ on toric currents and classes.  The main idea is that Proposition \ref{prop:stablecomp} allows us to pushforward and pullback the incarnations of toric currents and classes in a manner that is consistent across different toric surfaces.  We will focus initially on currents, leaving details of the analogous discussion of classes to the reader.

\begin{cor}
\label{cor:unambiguousaction} Let $f$ be a toric map,  $T\in\pcc(\rzt)$, and $\tilde X\succ X$, $\tilde Y\succ Y$ be toric surfaces. 
\begin{enumerate}
 \item If $\ind(f_{XY})$ contains no $\torus$-invariant points, then
 $f_{XY*} T_X = \pi_{\tilde YY*} f_{\tilde X\tilde Y*} T_{\tilde X}$.
 \item If $f_{XY}$ does not contract any pole of $X$, then 
 $\pi_{\tilde X X*} f_{\tilde X\tilde Y}^* T_{\tilde Y} = f_{XY}^* T_Y.$
\end{enumerate}
\end{cor} 
 
That $\ind(f_{XY})$ contains no $\torus$-invariant points amounts (by Theorem \ref{thm:janliandme2} and the definition of $\ind(\rzf)$) to saying that $X$ realizes all points in $\ind(\rzf)$ and $\ind(f_{XY}) = \ind(\rzf)$.  That $f_{XY}$ contracts no poles likewise amounts (by Theorem \ref{thm:janliandme1}) to saying $\exc(f_{XY}) = \exc(\rzf)$.
 
\begin{proof}
Suppose $\ind(f_{XY})$ contains no $\torus$-invariant points, hence no points that are images of curves contracted by $\pi_{\tilde XX}$.   
So from Proposition \ref{prop:stablecomp} and $\ind(\pi_{\tilde YY}) = \emptyset$, we infer that 
$$
\pi_{\tilde Y Y*}f_{\tilde X\tilde Y*} T_{\tilde X} = f_{\tilde XY*}T_{\tilde X} = f_{XY*}\pi_{\tilde X X*} T_{\tilde X} = f_{XY*} T_X,
$$
which gives the first assertion. For the second assertion, we have similarly that since $f_{XY}$ contracts no poles,
$$
\pi_{\tilde XX*} f_{\tilde X\tilde Y}^* T_{\tilde Y} = f_{X\tilde Y}^* T_{\tilde Y} = f_{XY}^*(\pi_{\tilde Y Y}^{-1})^* T_{\tilde Y} = f_{XY}^* \pi_{\tilde Y Y*} T_{\tilde Y} = f_{XY}^*T_Y.
$$
\end{proof}

Corollary \ref{cor:unambiguousaction} and the last conclusions of Theorems \ref{thm:janliandme1} and \ref{thm:janliandme2} make the following definition unambiguous.

\begin{defn}
\label{defn:pushandpull}
Let $f$ be a toric map and
let $T\in\pcc(\rzt)$.
\begin{itemize}
 \item The \emph{pullback} $\rzf^*T\in\pcc(\rzt)$ is given in any toric surface $X$ by choosing a toric surface $Y$ sufficiently dominant that $f_{XY}$ contracts no poles of $X$ and setting $(\rzf^*T)_X := f_{XY}^* T_Y$;
 \item The \emph{pushforward} $\rzf_* T\in\pcc(\rzt)$ is given in any toric surface $Y$ by choosing a toric surface $X$ sufficiently dominant that $\ind(f_{XY})$ contains no $\torus$-invariant points and setting
 $(\rzf_* T)_Y = f_{XY*} T_X$.
\end{itemize}
\end{defn}

% \mpar{I don't see this is needed anywhere. -Jeff}
%It is useful to note that the pushforward and pullback operators $\rzf_*,\rzf^*$ we have just defined are compatible with the corresponding operators on any particular toric surface.
% 
% \begin{prop}
% \label{prop:pushandpullonX}
% Let $X$ be a toric surface and $T\in\pcc(\rzt)$ be a tame current determined on $X$.  Then $(\rzf^* T)_X = f_{XX}^* T_X$ and $(\rzf_* T)_X = f_{XX*} T_X$.
% \end{prop}
% 
% \begin{proof}
% Choose $Y\succ X$ sufficiently dominant that $f_{XY}$ does not contract poles of $X$.  Since $T$ is determined on $X$, we have $T_Y = \pi_{YX}^* T_X$.  Then Proposition \ref{prop:stablecomp} and the fact that $\ind(\pi_{YX}) = \emptyset$ imply that 
% $$
% f_{XX}^* T_X = f_{XY}^* \pi_{YX}^* T_X = f_{XY}^* T_Y = (\rzf^* T)_X.
% $$
% The argument for pushforward is identical.
% \end{proof}

\begin{prop} If $f:\torus\tto\torus$ is a toric map and $T\in\pcc(\rzt)$ is internal, then $\rzf^*T$ and $\rzf_*T$ are also internal.
\end{prop}

\begin{proof}
Choose toric surfaces $X$ and $Y$ so that $f_{XY}$ contracts no poles of $X$.  Corollary \ref{cor:curveimages} then tells us that the image $f_{XY}(C_\tau)$ of any pole in $X$ is a pole in $Y$.  So if $T$ is internal, then $T_Y$ does not charge poles, and therefore neither does $(\rzf^*T)_X = f_{XY}^* T_Y$.  Hence $\rzf^*T$ is internal.  The argument that $\rzf_*T$ is internal is similar.
\end{proof}

We remark that if $D\in\pcc(\rzt)$ is an external divisor, then $\rzf^*D$ and $\rzf_*D$ need not also be external, as the divisor $\rzf^*D$ might have support on $\exc(\rzf)$, and $\rzf_*D$ might have support on $\rzf(\ind(\rzf))$. 

\begin{cor}
\label{cor:pushandpullalt}
If $f:\torus\tto\torus$ is toric then both $\rzf^*$ and $\rzf_*$ preserve the set of tame currents.  
%Let $X$ and $Y$ be toric surfaces and $T\in\pcc(\rzt)$.
\begin{itemize}
 \item If $T \in \pcc(\rzt)$ is tame and determined in $X$ and $f_{XY}$ contracts no poles of $X$, then $\rzf_*T$ is tame and  determined in $Y$, satisfying $(\rzf_*T)_Y = f_{XY*} T_X$.
 \item If $T \in \pcc(\rzt)$ is tame and determined in $Y$ and $\ind(f_{XY})$ contains no $\torus$-invariant points, then $\rzf^*T$ is tame and determined in $X$, satisfying $(\rzf^* T)_X = (f_{XY})^* T_Y$.
\end{itemize}
\end{cor}

\begin{proof}
The arguments for the two assertions are similar, so we only give details for the second.  Given $\tilde X\succ X$, choose $\tilde Y\succ Y$ so that
$f_{\tilde X\tilde Y}$ contracts no poles of $\tilde X$.  
%Then by definition
%$(\rzf^*T)_{\tilde X} = f_{\tilde X\tilde Y}^* T_{\tilde Y}$.  

Since $f_{\tilde X\tilde Y}$ doesn't contract poles, neither does $f_{X\tilde Y}$.  So we further have
$$
(\rzf^*T)_X = f_{X\tilde Y}^* T_{\tilde Y} = f_{X\tilde Y}^* \pi_{\tilde Y Y}^* T_Y = f_{XY}^* T_Y,
$$
where the first equality holds by definition of $\rzf^*T$, the second because $T$ is determined in $Y$, and the third by Proposition \ref{prop:stablecomp} and the fact that $\ind(\pi_{\tilde YY}) = \emptyset$.  

To see that $\rzf^*T$ is tame and determined in $X$ note that 
$$
(\rzf^*T)_{\tilde X} = f_{\tilde X\tilde Y}^* T_{\tilde Y} = f_{\tilde X \tilde Y}^* \pi_{\tilde Y Y}^* T_Y    = f_{\tilde X Y}^* T_Y = \pi_{\tilde X X}^* f_{XY}^* T_Y = \pi_{\tilde X X}^* (\rzf^*T)_{X}.
$$
Here the first equality holds by definition of $\hat f^* T$ and choice of $\tilde Y$, and 
the second holds because $T$ is determined in $Y$.  The third and fourth equalities use Proposition \ref{prop:stablecomp} and the facts that $\ind(\pi_{\tilde YY})=\emptyset$, whereas $\ind(f_{XY})$ contains no
$\torus$-invariant points.  The final equality uses the previous displayed equation.   Since $(\rzf^*T)_{\tilde X} = \pi_{\tilde X
X}^* (\rzf^*T)_{X}$ for arbitrary $\tilde X\succ X$, we see that $(\rzf^*T)$ is determined in~$X$.
\end{proof}

Much of the discussion from \S~\ref{subsec:ratmaps} concerning pushforward and pullback of currents and classes on surfaces generalizes more or less immediately to toric currents and classes.

\begin{prop}
\label{prop:pushpullbasics}
If $f:\torus\tto \torus$ is a toric map, then
\begin{itemize}
 \item $\rzf^*,\rzf_*:\pcc(\rzt) \to \pcc(\rzt)$ are weakly continuous
 and preserve the cone $\pcc^+(\rzt)$.
 \item If $T = dd^c \varphi\in\pcc(\rzt)$ is completely cohomologous to $0$, then so are $\rzf^* T = dd^c(\varphi\circ f)$ and $\rzf_*T = dd^c (f_*\varphi)$.
 \item Hence $\rzf^*,\rzf_*$ descend to continuous linear operators $\rzf^*,\rzf_*:\hoo(\rzt)\to\hoo(\rzt)$.
 \item $\rzf^*$ and $\rzf_*$ preserve the set of nef classes in $\hoo(\rzt)$.
 \item If at least one of the classes $\alpha,\beta\in\hoo(\rzt)$ is Cartier, then we have the Projection Formula $(\rzf^*\alpha\cdot \beta) = (\alpha\cdot \rzf_*\beta)$.
\end{itemize}
\end{prop}

Proposition \ref{prop:stablecomp} also translates easily to toric classes and currents.

\begin{prop}
\label{prop:stablecomp2}
The following are equivalent for toric maps $f,g:\torus\tto \torus$.
\begin{itemize}
 \item $(\rzf \circ \hat g)^* = \hat g^* \circ \rzf^*$ (on $\pcc(\rzt)$ and/or $\hoo(\rzt)$);
 \item $(\rzf\circ \hat g)_* = \rzf_* \circ \hat g_*$;
 \item $\hat g(\exc(\hat g))\cap \ind(\rzf) = \emptyset$.
\end{itemize}
Hence if $\rzf$ is internally stable, then $(\rzf^n)^* = (\rzf^*)^n$ and $(\rzf^n)_* = (\rzf_*)^n$ for all $n\in\N$.
\end{prop}

Finally, we generalize Theorem \ref{thm:pushpull1} to toric maps

\begin{thm}
\label{thm:pushpull2}
For any toric map $f:\torus\tto\torus$, the continuous linear operator $\extra:\hoo(\rzt)\to\hoo(\rzt)$ given by $\extra(\alpha) = \rzf_*\rzf^*\alpha - \dtop\alpha$ satisfies the following for any $\alpha\in\hoo(\rzt)$.
\begin{enumerate}
 \item $\extra(\alpha)$ is Cartier, represented by a divisor supported on $\rzf(\ind(\rzf))$.
 \item $\extra(\alpha)$ is nef whenever $\alpha$ is effective.  
 \item $\extra(\alpha) = 0$ if and only if $(\alpha\cdot C) = 0$ for each curve $C\subset\rzf(\ind(\rzf))$.
 \item $(\extra(\alpha)\cdot\alpha) \geq 0$ with equality if and only if $\extra(\alpha) = 0$.
\end{enumerate}
\end{thm}

We stress that the intersection in the final two items make sense for
\emph{any} toric class $\alpha$ because curves $C \subset \rzf(\ind(\rzf))$
represent Cartier classes, and so for sufficiently dominant surfaces $X$, we have
$(\alpha\cdot C) := (\alpha_X\cdot C_X)_X$ is independent of $X$.

\begin{proof}
By continuity of $\rzf^*$ and $\rzf_*$, it suffices to establish the first two conclusions in the case where $\alpha$ is Cartier.  Then since $\extra(\alpha)$ is Cartier and all components of $\rzf(\ind(\rzf))$ are tame, we can again assume $\alpha$ is Cartier when proving the last two conclusions.  

So assuming that $\alpha$ is Cartier, we also have that $\rzf^*\alpha$ and $\rzf_*\rzf^*\alpha$ are Cartier.  Choose a toric surface $Y$ sufficiently dominant that $\alpha$, $\rzf_*\rzf^*\alpha$ are determined and all components of $\rzf(\ind(\rzf))$ are fully realized on $Y$.  Then choose a toric surface $X$ sufficiently dominant that $\ind(f_{XY}) = \ind(\rzf)\subset X^\circ$.  It follows from Corollary \ref{cor:pushandpullalt} that $\rzf^*\alpha$ is determined in $X$ with $(\rzf^*\alpha)_X = f_{XY}^*\alpha_Y$.  Hence by Theorem \ref{thm:pushpull1}
\begin{align}\label{EQN:EXTRA_XY}
(\rzf_*\rzf^* \alpha)_Y - \dtop\alpha_Y = (f_{XY*}f_{XY}^* -\dtop)\alpha_Y = \extra_{XY}(\alpha_Y) 
\end{align}
for some class $\extra_{XY}(\alpha_Y)$ depending linearly on $\alpha_Y$ and
represented by a divisor on $f_{XY}(\ind(f_{XY})) = \rzf(\ind(\rzf))$.  
Since the left hand side of (\ref{EQN:EXTRA_XY}) depends only on $Y$, 
$\extra_Y(\alpha_Y) \equiv \extra_{XY}(\alpha_Y)$ is independent of the choice of $X$. Since
$\rzf_*\rzf^* \alpha - \dtop\alpha$ is determined in $Y$, we further have for any $Z\succ Y$
that the corresponding class is given by $\extra_Z(\alpha_Z) =
\pi_{ZY}^*\extra_Y(\alpha_Y)$.  So taking $\extra(\alpha)\in\hoo(\rzt)$ to be
the Cartier class determined by $\extra_Y(\alpha_Y)$, we obtain 
$$
\rzf_*\rzf^*\alpha - \dtop(f)\alpha = \extra(\alpha).  
$$
All remaining assertions about $\extra(\alpha)$ proceed from the corresponding facts about $\extra_Y(\alpha_Y)$, as given by Theorem \ref{thm:pushpull1}.  
We note, though, that the the second conclusion of Theorem~\ref{thm:pushpull2} is stronger than its counterpart in Theorem~\ref{thm:pushpull1} because each irreducible component $C\subset\rzf(\ind(\rzf))$ is an internal curve, representing a nef class, so that $(\alpha\cdot C) \geq 0$ whenever $\alpha$ is (merely) effective and $\extra(\alpha)$ is nef as soon as it is effective.  
\end{proof}

\subsection{Comparison with \cite{BFJ08}}

Our definitions of pushforward and pullback are again inspired by the definitions of $f^*,f_*:W(\cp^2)\to W(\cp^2)$ given in \cite{BFJ08}.  Let us focus here on $f^*$.  Starting with any rational map $f:\cp^2\tto\cp^2$ and a class $\tilde\alpha\in W(\cp^2)$, they define the incarnation $(f^*\tilde\alpha)_X$ on a blowup $X\to\cp^2$ by choosing $Y$ so that $\exc(f_{XY}) = \emptyset$.  and declaring $(f^*\tilde\alpha)_X = f_{XY}^* \tilde\alpha_Y$.  In our toric setting, the best we can do is achieve that $\exc(f_{XY}) = \exc(\rzf)$ consists of only the persistently exceptional curves.  This means among other things that while $(f^n)^* = (f^*)^n$ automatically holds for $f^*:W(\cp^2)\tto W(\cp^2)$, it only holds for $\rzf^*:\hoo(\rzt)\tto\hoo(\rzt)$ when $\rzf$ is internally stable.  Despite this, we have some compatibility between the two notions of pullback.

\begin{prop}
\label{prop:pullbackcomparison}
Given a toric map $f:\torus\tto\torus$, we have $\rzf^* = \rs\circ f^*\circ \ex$ and $\rzf_* = \rs \circ f_* \circ ex$, where $f^*,f_*$ denote pushforward and pullback on $W(\cp^2)$.  
\end{prop}

Note that the reverse relationship $f^* = \ex\circ\rzf^*\circ \rs$ does \emph{not} generally hold. 

\begin{proof}
We give the details for pullback only.  Fix a class $\alpha\in\hoo(\rzt)$.  Given a toric surface $X$, let $Y$ be another toric surface such that $f_{XY}$ contracts only persistently exceptional curves.  Blowing up still further we can also choose a (non-toric) surface $\tilde Y\to Y$ such that $f_{X\tilde Y}$ doesn't contract any curves at all.  Replacing $Y$ with the maximal toric surface dominated by $\tilde Y$, we still have that $f_{XY}$ contracts only persistently exceptional curves.  By (our) definition $(\rzf^*\alpha)_X = f_{XY}^*\alpha_Y$; whereas by the definition of pullback from \cite{BFJ08} and the discussion at the end of \S\ref{sec:tcurrents},
\begin{align*}
\rs(f^* \ex(\alpha))_X = (f^*\ex(\alpha))_X =& f_{X\tilde Y}^*\ex(\alpha)_{\tilde Y}  \\ 
&= f_{X\tilde Y}^*\pi_{\tilde Y Y}^*\alpha_Y = (\pi_{\tilde Y Y}\circ f_{X\tilde Y})^*\alpha_Y = f_{XY}^*\alpha_Y = (\hat f^* \alpha )_X.
\end{align*}
The fourth equality follows from Proposition \ref{prop:stablecomp} and the fact that $\ind(\pi_{\tilde YY})=\emptyset$.  So the outcome is the same by either definition.
\end{proof}

\begin{cor}\label{COR:BOUNDED_ON_L2_and_PROJECTION_FORMULA} If $f:\torus\tto\torus$ is toric, then the operators
$\rzf^*,\rzf_*$ on $\hoo(\rzt)$ restrict to bounded linear operators on
$\eltwo(\torus)$ satisfying 
$$
(\rzf^*\alpha\cdot \beta) = (\alpha\cdot f_*\beta) 
$$
for all $\alpha,\beta\in\eltwo(\rzt)$.
\end{cor}

\begin{proof}
That $\rzf^*,\rzf_*$ restrict to continuous operators on $\eltwo(\rzt)$ follows immediately from Propositions \ref{prop:comparison} and \ref{prop:pullbackcomparison} and continuity of the operators $f_*,f^*:L^2\to L^2$ (see Proposition 2.3 in \cite{BFJ08}).  The displayed formula follows from continuity of $\rzf^*$ and $\rzf_*$ and Proposition \ref{prop:pushpullbasics}, i.e. the fact that it holds if either of the classes is Cartier.
\end{proof}

We can now apply the main results from \cite{BFJ08} to obtain the starting point for our construction of forward and backward invariant currents associated to a toric map.  

\begin{thm}
\label{thm:invariantclasses}
Let $f:\torus\tto\torus$ be an internally stable toric map.  Then there exist non-zero nef classes $\alpha^*,\alpha_*\in\eltwo(\rzt)$ such that $\rzf^*\alpha^* = \ddeg(f)\alpha^*$ and $\rzf_*\alpha_* = \ddeg\alpha_*$.  If additionally $\ddeg^2(f) > \dtop(f)$, then up to constant multiple these classes are unique, satisfying for any $\alpha\in\eltwo(\rzt)$ that 
\begin{itemize}
 \item $\lim_{n\to\infty} \ddeg^{-n}(\rzf^n)^* \alpha = (\alpha\cdot\alpha_*)\,\alpha^*$;
 \item $\lim_{n\to\infty} \ddeg^{-n}(\rzf^n)_*\alpha = (\alpha\cdot\alpha^*)\,\alpha_*$.
\end{itemize}
\end{thm}

\begin{proof}
We give details only for the class $\alpha^*$ here.  The arguments for $\alpha_*$ are identical.  Theorem 3.2 in \cite{BFJ08} gives a non-zero nef class $\tilde\alpha^*\in W(\cp^2)$ such that $f^*\tilde\alpha^* = \ddeg(f)\tilde\alpha^*$.  The argument from \cite{FaDa21} Theorem 4.12 shows in fact that one may construct $\tilde\alpha^*$ as follows.
Let $\ell\in\cp^2$ be a line.  Then for any nef class $\tilde\alpha \in W(\cp^2)$ and $t>\ddeg$, the series
$$
\tilde\alpha^*(t) = \sum_{n=0}^\infty t^{-n} f^{n*}\tilde\alpha
$$
converges, and any limit $\tilde\alpha^*$ of the bounded family $\tilde\alpha^*(t)/(\tilde\alpha^*(t)\cdot \ell)$ as $t$ decreases to $\ddeg(f)$ is a non-zero nef class with the desired invariance.  So if we begin with a non-zero nef class $\alpha\in\eltwo(\rzt)$ and let $\tilde\alpha = \ex(\alpha)\in W(\cp^2)$ be the extension of $\alpha$, Proposition \ref{prop:pullbackcomparison} and internal stability give us on any toric surface $X$ that 
$$
\alpha^*(t) := \rs(\tilde\alpha^*(t)) = \sum_{n=0}^\infty t^{-n}(\rzf^n)^*\alpha = \sum_{n=0}^\infty t^{-n} (\rzf^*)^n \alpha.
$$
In particular, the series on the right converges.  Applying $\rzf^*$ to both sides yields
$$
t^{-1}\rzf^*\alpha^*(t) = \alpha^*(t) - \alpha.
$$
Moreover, since $\ell$ determines a class in the toric surface $\cp^2$, we have
$(\tilde\alpha^*(t)\cdot \ell) = (\alpha^*(t)\cdot\ell)$, which as is observed
in \cite{FaDa21}, increases to infinity as $t$ decreases to $\ddeg(f)$.  So if
$\tilde \alpha^*$ is any limit of $\tilde\alpha^*(t) /(\tilde\alpha^*(t)\cdot\ell)$ as $t\searrow\ddeg(f)$,
and we let $\alpha^* := \rs(\tilde\alpha^*)\in\hoo(\hat \torus)$ be its
restriction, we obtain from continuity of $\rzf^*$ that
$$
\ddeg^{-1}\rzf^*\alpha^*  = \alpha^*.
$$
This proves the first assertion in the theorem.  

Under the additional assumption that $\ddeg(f)^2 > \dtop(f)$, Theorem 3.5 in \cite{BFJ08} gives for any class $\tilde\alpha\in L^2$ that
$$
\lim_{n\to\infty} f^{n*}\tilde\alpha/\ddeg(f)^n = (\tilde\alpha\cdot\tilde\alpha_*)\,\tilde\alpha^*.
$$
Taking $\tilde\alpha = \ex(\alpha)$ for some class $\alpha\in\eltwo(\rzt)$, we again invoke internal stability and Proposition \ref{prop:pullbackcomparison} to see that
$$
\frac{(\rzf^*)^n\alpha}{\ddeg(f)^n} = \frac{(\rzf^n)^*\alpha}{\ddeg(f)^n} = \frac{\rs(f^{n*}\tilde\alpha)}{\ddeg(f)^n} \to (\tilde\alpha\cdot\tilde\alpha_*)\,\rs(\tilde\alpha^*) = (\alpha\cdot\alpha_*)\,\alpha^*,
$$
since $\tilde\alpha = \ex(\alpha)$ is the extension of a class in $\hoo(\rzt)$.
\end{proof}

We remark that under the stronger hypothesis that $\ddeg^2(f) > \dtop(f)$, the above proof can be simplified. One obtains the invariant class~$\alpha^*$ directly using the last paragraph of the proof (and \cite[Theorem 3.5]{BFJ08}).

\begin{prop}
Let $f:\torus\tto\torus$ and $\alpha^*,\alpha_*\in\hoo(\rzt)$ be as in Theorem \ref{thm:invariantclasses}.  If the tropicalization $A_f$ of $f$ is a homeomorphism with irrational rotation number then neither class $\alpha^*$ or $\alpha_*$ is Cartier.
\end{prop}

\begin{proof}
Suppose for contradiction that $\alpha^*$ is Cartier and determined in some toric surface~$X$.  Since $\alpha^*$ is non-trivial and nef and every effective divisor on $X$ is linearly equivalent to an external divisor, there is a pole $C_\tau$ of $X$ such that $(\alpha^*_X \cdot C_\tau)_X > 0$.  Since $A_f$ has irrational rotation number, there further exists $k \in \Z_{> 0}$ such that $\tau' := A_f^{-k}(\tau) \not \in \Sigma_1(X)$.  Then on the one hand the intersection $(\alpha^*\cdot \beta)$ with any other class $\beta\in\hoo(\rzt)$ is given by the intersection $(\alpha_X\cdot\beta_X)$ of incarnations in $X$.  Hence $(\alpha^*\cdot C_{\tau'}) = (\alpha_X\cdot 0) = 0$.

But on the other hand $\rzf^k_* C_{\tau'} \geq \deg(\rzf^k|_{C_{\tau'}})\,C_\tau$ (strict inequality holds precisely when $\ind(\rzf^k)\cap C_{\tau'}\neq \emptyset$), so
$$
(\alpha^* \cdot  C_{\tau'}) = \ddeg^{-k}(\rzf^{k*}\alpha^*\cdot C_{\tau'}) \geq \ddeg^{-k}\deg(f^k|_{C_{\tau'}})(\alpha^*\cdot C_{\tau}) > 0,
$$
where the first inequality holds because $\alpha^*$ is nef.  The gives us our contradiction.

% 
% 
% Let $Y \succ X$ be chosen so that $A_f^{-k}(\tau)\in\Sigma_1(Y)$ and 
% $\ind(f^k_{YX})$ contains no torus invariant points.   Then:
% \begin{itemize}
% \item[(i)]  $\begin{displaystyle}\alpha^*_Y = \frac{1}{\lambda_1^k}((\rzf^k)^* \alpha^*)_Y = \frac{1}{\lambda_1^k}(f^k_{YX})^* \alpha^*_X\end{displaystyle}$, by Corollary \ref{cor:pushandpullalt}, and 
% \item[(ii)]
% $(f^k_{YX})_* C_{\tau,Y} = 0 \in \hoo(X)$ since $A_f^k(\tau) \not \in \Sigma_1(X)$.
% \end{itemize}
% One on hand, since $\alpha^*$ is Cartier and determined in $X$ and we have
% \begin{align*}
% (\alpha^*_Y \cdot C_{\tau,Y})_Y = (\pi_{YX}^* \alpha^*_X \cdot C_{\tau,Y})_Y = (\alpha^*_X \cdot \pi_{YX*} C_{\tau,Y})_X =  (\alpha^*_X \cdot C_{\tau,X})  > 0.
% \end{align*}
% On the other hand,
% \begin{align*}
% (\alpha^*_Y \cdot C_{\tau,Y})_Y = \frac{1}{\lambda_1^k} ((f^k_{YX})^* \alpha^*_X \cdot C_{\tau,Y})_Y =
% \frac{1}{\lambda_1^k} (\alpha^*_X \cdot (f^k_{YX})_* C_{\tau,Y})_X = 0.
% \end{align*}
% This provides the desired contradiction.
The proof for the forward invariant class $\alpha_*$ is quite similar and the details are left to the reader.
\end{proof}

\section{Monomial maps and equidistribution}
\label{sec:monomial}

The following facts about monomial maps were first observed by Favre.

\begin{thm}[\cite{Fav99}, also \cite{JoWu11}]
\label{thm:monomialstuff}
Let $h = h_A:\torus\to\torus$ be the monomial map associated to a non-singular linear operator $A:N\to N$.  Then 
\begin{itemize}
 \item $\dtop(h) = \det A$ and $\ddeg(h)$ is the maximum of the absolute values of the eigenvalues of $A$; 
 \item in particular, $\ddeg(h)^2 = \dtop(h)$ if and only if the eigenvalues of $A$ are a complex conjugate pair;  
 \item there exists an iterate $h^n$, $n\leq 6$, and a toric surface $X$ on which $h^n$ extends to an algebraically stable rational map if and only if some power of $A$ has real eigenvalues.
\end{itemize}
\end{thm}

The equidistribution results from \cite{DDG11} apply to monomial maps
satisfying the third conclusion of this Theorem \ref{thm:monomialstuff}.  In
this section, we prove a complementary equidistribution result for the case
when the third conclusion fails, i.e. those for which the underlying matrix $A$
has eigenvalues $\xi,\bar\xi = |\xi|e^{\pm 2\pi i\theta}\in\C$ with irrational
arguments $\pm\theta\in\R$.  First however, we prove Theorem \ref{thm:mainC} from the introduction.

%\begin{thm}
%\label{thm:minimalddeg}
%If $f:\torus\tto\torus$ is an internally stable toric map such that $\ddeg(f)^2 = \dtop(f)$, then either 
%\begin{itemize}
% \item $f$ is a shifted monomial map; or
% \item $\ddeg(f) \in \Z_{> 0}$ and there exists a distinguished coordinate system in which $f$ has the skew product form 
%$$
%f(x_1,x_2) = (t x_1^{\pm\ddeg(f)},g(x_1)x_2^{\pm\ddeg(f)})
%$$
%for some $t\in\C^*$ and rational function $g:\cp^1\to\cp^1$.  
%\end{itemize}
%In the second case the tropicalization $A_f$ of $f$ has rational rotation number.
%\end{thm}

\begin{proof}[Proof of Theorem \ref{thm:mainC}]
Since $f$ is internally stable, Theorem \ref{thm:invariantclasses} gives that there is a nef class $\alpha^*\in\hoo(\torus)$ satisfying $\rzf^*\alpha^*=\ddeg(f)\alpha^*$.  Corollary \ref{COR:BOUNDED_ON_L2_and_PROJECTION_FORMULA} and Theorem \ref{thm:pushpull2} tells us that
$$
\ddeg(f)^2(\alpha^*\cdot\alpha^*) = (\rzf^*\alpha^*\cdot\rzf^*\alpha^*) = \dtop(f) (\alpha^*\cdot\alpha^*) + (\extra(\alpha^*)\cdot\alpha^*).
$$
Since $\ddeg(f)^2 = \dtop(f)$, we see that $(\extra(\alpha^*)\cdot\alpha^*) = 0$ and therefore $(C\cdot\alpha^*) = 0$ for every curve $C\subset\rzf(\ind(\rzf))$.  But any such curve is internal and therefore nef. So by Theorem \ref{thm:intersection} $C^2 = 0$ and the cohomology class of $C$ is a multiple of $\alpha^*$.  

Suppose now that $f$ is not a shifted monomial map.  Then by Corollary \ref{cor:noexceptions} there is at least one persistently exceptional curve $E$ for $\rzf$.
Since $\rzf(E)$ is a point, the pushforward $\rzf_*E$ is supported on $\rzf(\ind(\rzf))$:
$$
\ddeg(f) (E\cdot \alpha^*) = (E\cdot\rzf^*\alpha^*) = (\rzf_* E\cdot\alpha^*) \leq M\sum_{C\subset\rzf(\ind(\rzf))} (C\cdot\alpha^*) = 0.
$$
Again because $E$ is internal, it follows that the class of $E$ is a positive multiple of $\alpha^*$ and $E^2 = \alpha^{*2} = 0$.  

Now let $X$ be a toric surface that fully realizes all curves in $\exc(\rzf)$ and $\rzf(\ind(\rzf))$ and $K_X = -\sum_{\tau \in\Sigma_1(X)} C_\tau$ be the canonical class of $X$.  Then
the genus formula tells us that the arithmetic genus of $E$ is a non-negative integer given by
$$
1 + \frac12 E\cdot (E+K_X) = 1 + \frac12 E\cdot K_X = 1 - \frac12\sum_{\tau\in\Sigma_1(X)} E\cdot C_\tau \leq 0,
$$
since by Corollary \ref{cor:balanceformula}, $E$ meets at least two external curves in $X$.  So in fact $E\cdot K_X = -2$, and the arithmetic genus of $E$ vanishes.  This implies that $E$ is a smooth rational curve that meets exactly two poles $C_{\tau_2},C_{-\tau_2}\subset\rzt$, both realized in $X$, and $(E\cdot C_{\pm\tau_2}) = 1$.  
% In fact the same is true (with the same $\tau_1$) for every curve in $\exc(\rzf)$ and $\rzf(\ind(\rzf))$.  Rescaling, we may assume that $\alpha^*$ is the common class of all such curves.

Let $\tau_1\subset N_\R$ denote another rational ray, chosen so that the
generators of $\tau_j\cap N$, $j=1,2$ form a basis for $N$.  Then
$\{0\},\pm\tau_1,\pm\tau_2$ and the four complementary sectors constitute a fan
in $N_\R$, and this fan determines a toric surface $X$ that fully realizes the
persistently exceptional curve $E$.  The distinguished coordinate system
$(x_1,x_2)$ identifying $C_{\tau_j}$ with $\{x_j = 0\}$  extends to an
isomorphism $X\to \cp^1\times\cp^1$ identifying $C_{-\tau_j}$ with $\{x_j =
\infty\}$.  
Let us write $f$ with domain and codomain in these distinguished coordinates:
$$
(x_1,x_2) \mapsto (f_1(x_1,x_2),f_2(x_1,x_2)).
$$
The curve $E$ is a vertical line $\{x_1 = c\}$ for some $c\in\C^*$,
and since $E$ is cohomologous to a multiple of $\alpha^*$, the invariance
$\rzf^*\alpha^* = \ddeg\alpha^*$ implies that $f_1(x_1,x_2) \equiv f_1(x_1)$ is a function of $x_1$ only.   I.e. 
$f$ is a skew product preserving the vertical
fibration.  

Write $f_1(x_1) = x_1^a g_1(x_1)$, where $g_1(0)\neq 0,\infty$ and $f_2(x_1,x_2) = x_1^b x_2^c g_2(x_1,x_2)$ where $g_2$ has no zero or pole along either axis $\{x_j=0\}$.  Writing out $f^*\eta = \cov(f)\eta$ in coordinates then gives
$$
\cov(f) = ac + x_1 c\frac{g'_1}{g_1} + ax_2\frac{\partial g_2/\partial x_2}{g_2} +  x_1x_2 \frac{g'_1}{g_1} \frac{\partial g_2/\partial x_2}{g_2}.
$$
Setting $x_2 = 0$, we see that $g'_1\equiv 0$ so that $\cov(f) = ac$ and $g_1(x_1) \equiv t\in\C^*$ is constant.  It further follows then that $\frac{\partial g_2}{\partial x_2} \equiv 0$ so that $g_2 = g_2(x_1)$ is a function of $x_1$ only.  In short, the distinguished coordinate representation of $f$ reduces to 
$$
(x_1,x_2) \mapsto (t x_1^a,x_2^c g(x_1)),
$$
where $a$ and $b$ are integers satisfying $|ab| = \cov(f)$, $t\in\C^*$, and $g(x_1) = x_1^b g_2(x_1)$ is a rational function of $x_1$.  In particular, $f(\{x_1=0\})$ is either $\{x_1=0\}$ or $\{x_1=\infty\}$.  

Since $f^*\{x_1 = 0\} = |a|\{x_1=0\}$, we infer $a = \pm\ddeg(f)$.  One computes directly that typical points $(x_1,x_2)$ have $|ac|$ preimages under $f$.  Hence $\ddeg(f)^2 = \dtop(f) = |ac| = \ddeg(f)|c|$, i.e. $c=\pm\ddeg(f)$.
\end{proof}

The following statement was noted in the proof of Theorem \ref{thm:mainC}.

\begin{rem}\label{COR:RATIONAL_ROT_NUMBER}
In either conclusion of Theorem \ref{thm:mainC} one sees that the tropicalization $A_f$ of $f$ is a homeomorphism.  In the case when $f$ is a skew product, it maps the pole $\{x_1=0\}$ to itself, so $A_f$ fixes the ray $\tau_1$ that indexes this pole.   Hence under the additional hypothesis that $A_f$ has irrational rotation number, the conclusion narrows: $f$ must be a shifted monomial map.
\end{rem}

For the remainder of this section, we focus on monomial maps $h = h_A$ that do not admit algebraically stable models.  That is, we assume that $A$ has eigenvalues $\xi,\bar\xi = |\xi|e^{\pm 2\pi i\theta}$ for $\theta\in\R\setminus\Q$.  For any such $A$ we have a norm $\|\cdot\|_A$ on $N_\R$, the Euclidean norm relative to an appropriately chosen basis, such that $\|A\nvec\|_A = |\xi|\|\nvec\|_A$ for all $\nvec\in N_\R$.  Since this norm is a convex, homogeneous function on $N_\R$, we have an associated homogeneous current $\bar T_A\in\pcci^+(\rzt)$ given on $\torus$ by $\bar T_A := dd^c\|\trop\|_A$.  Since $\trop\circ h = A\circ\trop$, pullbacks of homogeneous currents by $\rzh$ remain homogeneous.  In particular, 
$$
\rzh^* \bar T_A = dd^c\|\trop \circ h\|_A = dd^c\|A\circ\trop\|_A = |\xi|dd^c\|\trop\|_A = \ddeg(h)\bar T_A. 
$$
Hence the class $\alpha^*\in\hoo(\rzt)$ of $\bar T_A$ is a leading eigenvector for pullback, the existence of which is guaranteed by Theorem \ref{thm:invariantclasses}.   In fact, since $\ind(\rzh) = \emptyset$, we further have from Theorem \ref{thm:pushpull2} that
$$
\ddeg(h) \rzh_* \bar T_A = \rzh_*\rzh^* \bar T_A = \dtop(h) \bar T_A = \ddeg(h)^2 \bar T_A.  
$$
Hence $\rzh_* \bar T_A = \ddeg(h) \bar T_A$, too.  So $\alpha_* = \alpha^*$ is also a leading eigenvector for pushforward.

We can now state our equidistribution theorem for monomial maps.   
% We phrase it in the level of generality of toric
% currents, but it can be adapted to statements about equidistribution on $\cp^2$
% (and other specific toric surfaces) in a similar way to how derive Theorem
% \ref{thm:mainA} from Theorem \ref{thm:invcurrents} in the next section.

\begin{thm}
\label{thm:monomialequidistribution}
Let $A:N\to N$ be a linear operator with eigenvalues $\xi,\bar\xi\in\C$ such that $\xi^n \notin\R$ for any $n\in\N$.  Let $\bar T_A\in\pcc^+(\torus)$ be the associated homogeneous current. Then any non-zero tame current $T\in\pcci^+(\rzt)$, there exists $c^* > 0$ such that
$$
\lim_{n\to\infty} \frac{1}{n}\sum_{j=0}^{n-1} \frac{\rzh^{j*} T}{|\xi|^j} = c^* \bar T_A.
$$
\end{thm}

Since monomial maps preserve external divisors, the requirement that the current $T$ be internal is necessary in this theorem.  We do not know if the tameness requirement is also necessary.  Practically speaking, we depend on it at the conclusion of the proof to ensure $T-\bar T$ has a potential that is bounded above on $\torus$.  It is not used at all in the preliminary results Lemmas \ref{lem:homogcase} and \ref{lem:symcase} below.

\begin{lem}
\label{lem:homogcase}
Theorem \ref{thm:monomialequidistribution} is true for homogeneous currents $T = \bar T\in\pcci^+(\rzt)$.
\end{lem}

Note that without the C\'esaro averaging in Theorem \ref{thm:monomialequidistribution}, this lemma fails for any homogeneous current $\bar T$ that is not already a multiple of $\bar T_A$.  Since homogeneous currents uniquely represent nef classes in $\hoo(\rzt)$, one sees that C\'esaro averaging is needed for Theorem \ref{thm:monomialequidistribution} to hold even on the cohomological level.  

\begin{proof}
On $\torus$, we have $\bar T = dd^c(\bar\sfn\circ\trop)$ for some non-negative homogeneous convex function $\bar\sfn:N_\R\to \R$, and therefore
$$
\rzh^{j*} \bar T = dd^c(\bar\sfn\circ\trop\circ h^j) = dd^c(\bar\sfn\circ A^j\circ\trop).
$$
In particular $\rzh^{j*}\bar T$ is also homogeneous, and we have for any $\nvec\in N_\R$ that 
$$
|\xi|^{-j} \bar\sfn(A^j\nvec) = \norm[A]{\nvec}\cdot \bar\sfn\left(
\frac{A^j\nvec}{\norm[A]{A^j\nvec}}\right). 
$$
Since $A$ induces a homeomorphism with irrational rotation number on $\{\norm[A]{\nvec} = 1\}$, which is therefore uniquely ergodic, we also see that
$$
\lim_{n\to\infty} \frac{1}{n}\sum_{j=0}^{n-1} \bar\sfn\left(
\frac{A^j\nvec}{\norm[A]{A^j\nvec}}\right) = c^* > 0
$$
uniformly in $\nvec$ for some constant $c^* > 0$.  We conclude that
$$
\lim_{n\to\infty} \frac{1}{n}\sum_{j=0}^{n-1} \frac{\rzh^{j*}\bar T}{|\xi^j|} 
= 
dd^c \left(\norm[A]{\trop} \cdot 
\lim_{n\to\infty} 
\bar\sfn\left(
\frac{A^j \circ \trop}{\norm[A]{A^j\circ\trop}}
\right)\right)
= c^* dd^c \norm[A]{\trop} = c^* \bar T_A.
$$
\end{proof}

The analog of Lemma \ref{lem:homogcase} for normalized \emph{pushforwards} of a homogeneous current holds with a similar proof.  However, the remainder of the proof of Theorem \ref{thm:monomialequidistribution}, including the next result, is specific to normalized pullbacks.

\begin{lem}
\label{lem:symcase}
Let $\sfn:N_\R\to \R$ be a convex function with $\growth{\sfn}<\infty$ and $\bar\sfn$ be its homogenization.  Then
$$
\lim_{n\to\infty} \frac{\sfn\circ A^n - \bar\sfn\circ A^n}{|\xi|^n} = 0
$$
uniformly on compact subsets of $N_\R$.  
\end{lem}

\begin{proof}
Recall that the homogenization of $\sfn$ is given for any $\nvec\in N_\R$ by $\bar\sfn(\nvec)= \lim_{t\to\infty} t^{-1}\sfn(t\nvec)$.  So if we set $t_n := \norm[A]{A^n\nvec} = |\xi|^n\norm[A]{\nvec}$ and $u_n := t_n^{-1} A^n\nvec$, then $t_n\nearrow \infty$ and Proposition \ref{prop:homogenization} therefore yields
$$
\norm[A]{\nvec}^{-1} \lim_{n\to\infty} \frac{\sfn(A^n\nvec) -\bar\sfn(A^n\nvec)}{|\xi|^n} = 
\lim_{n\to\infty}\left(\frac{\sfn(t_n u_n)}{t_n} - \bar\sfn(u_n)\right) = 0. 
$$
uniformly on $N_\R$.  Multiplying through by $\norm[A]{\nvec}$, we get the same limit, albeit with uniform convergence only on compact subsets of $N_\R$.
%
% Now if $T\in\pcci(\rzt)$ is rotationally symmetric, given on $\torus$ by $T = dd^c(\sfn\circ\trop)$, then 
% $$
% \frac{\rzh^{*n} T - \rzh^{*n} \bar T}{|\xi|^n}
% =
% dd_X^c \frac{\sfn\circ\trop\circ h^n - \bar\sfn\circ\trop \circ h^n}{|\xi|^n}
% = 
% dd_X^c \frac{\sfn\circ A^n\circ\trop - \bar\sfn\circ A^n\circ \trop}{|\xi|^n}
% \to 0
% $$
% \emph{as currents on $\torus$}.  
% 
% To see that the limit remains zero in $\pcci(\rzt)$, note that the previous computation shows that any limit point $S$ of $|\xi|^{-n} \rzh^{n*} T$ in $\pcc(\rzt)$ has the form $S = T_A + S_{ext}$ where $S_{ext}$ is an effective external divisor.  On the other hand, $\bar T_A = \lim |\xi|^{-n} \rzh^{n*} \bar T$ by Lemma \ref{lem:homogcase}, so since 
% since $T$ and $\bar T$ are completely cohomologous, the same is true of $S$ and $\bar T_A$.  So $S_{ext}\geq 0$ is completely cohomologous to zero and must therefore vanish.
\end{proof}

The rest of the proof of Theorem \ref{thm:monomialequidistribution} resembles that of Lemma 1.10.12 in \cite{Sib99}.

\begin{proof}[Proof of Theorem \ref{thm:monomialequidistribution}] 
Fix a tame current $T\in\pcci^+(\rzt)$, a toric surface $X$ on which $T$ is determined, and a $\torus_\R$-invariant K\"ahler form $\omega$ on $X$.  Let $T_{ave}$ be the rotational average of $T$ and $\bar T$ be the homogenization of $T_{ave}$.  By Theorem \ref{thm:canonicalreps} and Proposition \ref{prop:internaltame}, $\bar T = \bar T(X)$ has a continuous local potential in a neighborhood of any point in $X$.  This means we can write $T - \bar T = dd^c\varphi$, where $\varphi:X\to [-\infty,0)$ is integrable and negative.  

It suffice to show that $|\xi|^{-n}\rzh^{n*} (T-\bar T)\to 0$ in $\pcc(\rzt)$.
We will first show that this holds on $\torus$.
Let $\chi$ be a smooth $(1,1)$ form supported on a compact set $\{\norm[A]{\trop(x)}\leq c\} \subset \torus$.  Then
$$
\left|\pair{h^{n*}(T-\bar T)}{\chi}\right| = \left|\int_{\norm[A]{\trop}<c} (\varphi\circ h^n) \,dd^c\chi\right|
    \leq -C_\omega \norm[C^2]{\chi} \int_{\norm[A]{\trop}<c}(\varphi\circ h^n) \,\omega^2,
$$
where $C_\omega>0$ depends only on the K\"ahler form.  On the other hand, $\torus_\R$-invariance of $\omega^2$ and $\trop$ implies for any $y\in\torus_\R$ that
$$
\int_{\norm[A]{\trop}<c} (\varphi\circ h^n) \,\omega^2 = \int_{\norm[A]{\trop}<c} (\varphi_y \circ h^n) \,\omega^2
=
\int_{\norm[A]{\trop}<c} (\varphi_{ave}\circ h^n) \,\omega^2
$$
where $\varphi_y(x) :=\varphi(xy)$ and $\varphi_{ave}(x) := \int_{y\in\torus_\R} \varphi_y(x)\,\eta$.  Now $\varphi_{ave}$ is a potential on $X$ for $T_{ave} - \bar T$ and so equal to $\sfn_{ave}(\trop(x)) - \bar\sfn(\trop(x)) + B$ for some constant $B\in\R$.  Lemma~\ref{lem:symcase} therefore implies
$$
\lim_{n\to\infty} \left|\pair{|\xi|^{-n} h^{n*}(T-\bar T)}{\chi}\right| \leq \lim_{n\to\infty} C_\omega\norm[C^2]{\chi} \int_{\norm[A]{\trop}<c} \frac{\varphi_{ave}\circ h^n}{|\xi|^n} \,\omega^2 = 0.
$$
That is, $\lim |\xi|^{-n} h^{n*}(T-\bar T) = 0$ on $\torus$.

To see that the limit remains zero in $\pcc(\rzt)$, note that the previous
computation shows that any limit point $S$ of $|\xi|^{-n} \rzh^{n*} T$ in
$\pcc(\rzt)$ has the form $S = \bar T_A + S_{ext}$ where $S_{ext}$ is an
effective external divisor.  On the other hand, any limit point of $\lim |\xi|^{-n}
\rzh^{n*} \bar T$ is internal by Corollary \ref{COR:CLOSED_SETS_POSITIVE_HOMOG_CURR}, so since since $T$ and $\bar T$
are completely cohomologous, the same is true of $S$ and $\bar T_A$.  So
$S_{ext}\geq 0$ is completely cohomologous to zero and must therefore vanish.
\end{proof}

To close this section, we make two further remarks about Theorem \ref{thm:monomialequidistribution}.  First of all, one can restate it in a form closer to Theorem \ref{thm:mainA}, as a theorem specifically about equidistribution of preimages of curves on $\cp^2$, though invariance of $T_A$ and convergence of pullbacks then hold only modulo external divisors.  See the discussion in Section \ref{sec:invcurrents} about passing from Theorem \ref{thm:invcurrents} to Theorem \ref{thm:mainA}. 
Secondly, given $h$ as in Theorem \ref{thm:monomialequidistribution} and $\tau\in\torus$, one can check that the shifted monomial map $\tau h$ is conjugate via translation on $\torus$ to $h$ itself.  Hence the theorem applies to \emph{shifted} monomial maps, too, albeit relative to some translate of $T_A$.

\section{Invariant toric currents}
\label{sec:invcurrents}

In this section, we establish the main results of this article, Theorems \ref{thm:mainA} and \ref{thm:mainD} from the introduction.  The tools developed in previous sections allow us to employ arguments that are by now standard for algebraically stable surface maps.  First we construct the current $T^*$ in Theorem \ref{thm:mainA}. 

\begin{thm}
\label{thm:invcurrents}
Let $f:\torus\tto \torus$ be an internally stable toric rational map whose tropicalization $A_f$ is a homeomorphism with irrational rotation number.  Let $\alpha^*,\alpha_*\in\hoo(\rzt)$ be the nef classes associated to $\rzf$ in the conclusion of Theorem \ref{thm:invariantclasses}.  If $f$ is not a shifted monomial map, then there exists $T^*\in\pcc(\rzt)$ such that for any nef $T\in\pcc(\rzt)$, we have
$$
\lim \frac{\rzf^{n*}T}{\ddeg(f)^n} = (T\cdot\alpha_*) T^*.
$$
Hence $T^*$ is nef and positive, satisfies $\rzf^* T^* = \ddeg(f) T^*$ and represents the class $\alpha^*$.
\end{thm}

The condition that $T$ be nef in the conclusion of this theorem guarantees that the intersection $(T\cdot\alpha^*)$ is finite.  It is satisfied e.g. when $T$ is positive and internal, or when $T$ is the tame current determined by an effective divisor in $\cp^2$.
% 
% \textcolor{blue}{Should we put a comment on why $T$ being nef is an appropriate condition, i.e. should we explain why there is not
% much hope of extending the theorem to non-nef $T$?}

\begin{proof}
Fix an internal current $T\in\pcci^+(\rzt)$ such that $(T\cdot\alpha_*) = 1$.
Since $f$ is internally stable, $T_n := \ddeg^{-n}\rzf^{n*} T = \ddeg^{-n}
(\rzf^*)^n T$.  Since $A_f$ has irrational rotation number and $f$ is not a
shifted monomial map, Corollary \ref{COR:RATIONAL_ROT_NUMBER} tells us that $\ddeg(f)^2 >
\dtop(f)$.  Hence by Theorem~\ref{thm:invariantclasses}, $\lim\ch{T_n} =
\alpha^*$.  In particular $([T_n])\subset \hoo(\rzt)$ is a relatively compact
sequence.  So, by Corollary~\ref{cor:compactgrowth}, there exists $M>0$ such that
$\growth{\bar\sfn_{T_n}} \leq M$ for all $n$.  

By Corollary \ref{cor:internalclass} we have $T - \bar T = T_0 - \bar T_0$ is completely cohomologous to zero.  So by Proposition \ref{prop:tamepotential} there exists $\varphi_0 := \varphi_{T}:\torus\to\R$ such that $\varphi$ is integrable on any toric surface $X$ compactifying $\torus$ and $T-\bar T = dd^c\varphi_0$ on $X$.  Similarly, for any $n>1$, the difference
$\lambda_1^{-1}(\rzf^* \bar T_{n-1} - \overline{\rzf^* \bar T_{n-1}})$ is completely cohomologous to zero with potential $\varphi_n :=
\varphi_{\ddeg^{-1}\rzf^*\bar T_{n-1}}\in L^1(X)$ independent of the choice of toric surface $X$.
Since $\bar T_{n-1}$ is completely cohomologous to $T_{n-1}$, it follows from Theorem \ref{thm:mainB} that
$\overline{\rzf^* \bar T_{n-1}} = \overline{\rzf^* T_{n-1}} = \bar T_n$.   Therefore, 
$$
dd^c\varphi_n = \lambda^{-1}\rzf^*\bar T_{n-1} - \bar T_n.
$$
Thus
\begin{eqnarray*}
T_n & = & \frac{\rzf^{n*} T}{\ddeg^n} = \frac{\rzf^{*n}\bar T_0 + dd^c \varphi_0\circ \rzf^n}{\ddeg^n}\\
    & = &\frac{\rzf^{(n-1)*}\bar T_1}{\ddeg^{n-1}} + dd^c\left(\frac{\varphi_1\circ \rzf^{n-1}}{\ddeg^{n-1}} + \frac{\varphi_0\circ \rzf^n}{\ddeg^n}\right) = \dots \\
    & = &\bar T_n + dd^c \sum_{j=0}^n \frac{\varphi_{n-j}\circ \rzf^j}{\ddeg^j}. 
\end{eqnarray*}
Since $\ch{T_n}$ converges to $\alpha^*$, Theorem \ref{thm:mainB} tells us that $\bar
T_n$ converges to the unique homogeneous current representing $\alpha^*$.  Therefore, to prove that the sequence $(T_n)$ converges,
it suffices to show for any fixed toric surface $X$ that the sum on the right
converges in $L^1(X)$ as $n\to\infty$.

Once we fix $X$, Corollary \ref{cor:compactgrowth} and Theorem \ref{thm:volbd} yield constants $a,b>0$ such that 
\begin{align*}
\vol\{|\varphi_n|\geq t\} \leq ae^{-bt}
\end{align*}
for all non-negative $n$ and $t$.  From Theorem \ref{thm:volshrink} we further infer for any $\mu >\sqrt{\dtop(f)}$ that there exist constants $A,B>0$ such that
$$
\vol\{|\varphi_n\circ f^m|\geq t\} \leq A e^{-Bt\mu^{-m}}
$$
for all non-negative $n$, $m$ and $t$.  Since $\ddeg(f)^2 > \dtop(f)$, we may assume $\ddeg> \mu$.  Thus
$$
\int_X \frac{|\varphi_{n-j}\circ f^j|}{\ddeg^j}\,dV = \frac{1}{\ddeg^j}\int_0^\infty \vol\{|\varphi_{n-j}\circ f^j|\geq t\}\,dt \leq \frac{A}{\ddeg^j}\int_0^\infty e^{-Bt\mu^{-j}}\,dt \leq \frac{A}{B}\left(\frac{\mu}{\ddeg}\right)^j.
$$
Hence
$\sum_{j=0}^\infty \frac{\varphi_{n-j}\circ f^j}{\ddeg^j}$ converges in $L^1(X)$ and therefore $T^* := \lim T_n$ exists.  Since $T_n$ is nef and positive for all $n$, so is $T^*$.  The cohomology class
is $\ch{T^*} = \lim \ch{T_n} = \alpha^*$.  Continuity of $f^*$ implies that $f^*T^* = \ddeg T^*$.  

To see that $T^*$ is independent of the choice of initial current $T$, suppose
$T'\in\pcc^+(\rzt)$ is another nef toric current
satisfying $f^* T' = \ddeg T'$.  Since $\ddeg(f)^2 > \dtop(f)$, the class
$\alpha^*$ is unique up to positive multiple.  Hence $T'$ and $T^*$ are
completely cohomologous, and we have $T' - T^* = dd^c \varphi$ for some
relative potential $\varphi$.  Invariance gives 
$$
T'-T^* = \frac{f^{n*}(T'-T^*)}{\ddeg}^n = dd^c \frac{\varphi\circ f^n}{\ddeg^n}
$$ 
for all $n\in\N$.  But one shows as above that $\lim_{n\to\infty}\int_X \frac{|\varphi\circ f^n|}{\ddeg^n}\,dV = 0$ geometrically on any toric surface $X$.
So in fact $T'=T^*$.

The assumptions that our initial current $T$ is internal and positive can be dispensed with similarly.  If $T\in\pcc(\rzt)$ is an arbitrary (in particular,  not
necessarily invariant) nef current with $(T\cdot \alpha_*) = 1$, then there exists a homogeneous current $\bar T\in\pcci^+(\rzt)$ completely cohomologous to $T$.  Writing $T-\bar T = dd^c \varphi$ for some relative potential $\varphi$, we have again for any toric surface $X$ that $\ddeg^{-n}\varphi\circ f^n\to 0$ in $L^1(X)$, and therefore 
$\ddeg^{-n}\rzf^{n*}(T - \bar T) \to 0$.
Hence 
$$
\lim \ddeg^{-n} \rzf^{n*} T = \lim \ddeg^{-n} \rzf^{n*} \bar T = T^*.
$$
\end{proof}

\begin{thm}
\label{thm:zerolelong}
Let $f:\torus\tto\torus$ and $T^*\in\pcc^+(\rzt)$ be as in Theorem \ref{thm:invcurrents}.  Then the Lelong number $\nu(T^*,p)$ vanishes at all points $p\in\rzt^\circ$ that are not persistently indeterminate for any iterate $\rzf^n$ of $\rzf$.  Hence $T^*$ does not charge curves.  
\end{thm}

Note that since the image of each persistent exceptional curve is a free point in some pole, assuming that $A_f:N_\R\to N_\R$ is a homeomorphism with irrational rotation number about $0$ guarantees that no exceptional curve is preperiodic.  

\begin{proof}[Proof of Theorem \ref{thm:zerolelong}]
Recall from Theorem \ref{thm:lelongbd} that there is an absolute upper bound $M$ for the Lelong numbers $\nu(p,T^*)$ of $T^*$ at points $p$ in any/all toric surfaces $X$.  Let $p\in\rzt^\circ$ be a point such that $p$ is not $\torus$-invariant and $\rzf^n(p)\notin \ind(\rzf)$ for any $n\geq 0$.  

First suppose that $p$ and its entire forward orbit lie in $\torus$.  By Theorem \ref{thm:janliandme2} we then have for every $n\in\N$ that $f$ is a local isomorphism about $f^n(p)$.
Hence
$$
\ddeg^n\nu(T^*,p) = \nu(f^{n*}T^*,p) = \nu(T^*,f^n(p)) \leq M.
$$
It follows on letting $n\to\infty$ that $\nu(T^*,p) = 0$.

Now suppose that $p\in C_\tau\subset X$ is a free point in some external curve but $\rzf^n(p)\notin \exc(\rzf)$ for any $n\geq 0$.  Choose an increasing sequence of toric surfaces $X_0 = X\prec X_1\prec \dots$ such that $A_f^n\tau\in \Sigma_1(X_n)$ for every $n$.  Then $p_n := f_{X_{n-1}X_n}\circ\dots\circ f_{X_0X_1}(p)$, $n\in\N$ is a well-defined sequence of free points in poles $C_{A_{f^n}(\tau)}\subset X_n$.  The local degree of $f^n_{X_0X_n}$ about $p$ is just the ramification $\ram(\rzf^n,\tau)$ about the pole containing $p$.  By Theorem \ref{thm:tropmap} this is equal to $\frac{\norm{A_f^n(\nvec)}}{\norm{\nvec'}}$, where $\nvec'\in N$ is the primitive vector generating $A_f^n(\tau)$. Hence Lemma \ref{lem:favrebd} tells us that 
$$
\ddeg^n\nu(T^*,p) = \nu(f^{n*}T^*,p) \leq \frac{\norm{A_f^n(\nvec)}}{\norm{\nvec'}}
\nu(T^*,f^n(p)) \leq M \frac{\norm{A_f^n(\nvec)}}{\norm{\nvec'}}.
$$
Thus $\nu(T^*,p) \leq M\cdot\lim_{n\to\infty} \frac{\norm{A_f^n(\nvec)}}{\ddeg^n \norm{\nvec'}} = 0$, by Theorem \ref{THM:GOOD_ITERATES} and the hypotheses on $f$ and its tropicalization.  Since for each $n\geq 0$, there are only finitely many points in $\exc(\rzf^n) \setminus\torus$, we infer that $T^*$ has no mass on external curves $C_\tau$, i.e. $T^*$ is an internal current.

Consider finally the possibility that the forward orbit of $p$ by $\rzf$ meets $\exc(\rzf)$.  Since $A_f$ has irrational rotation number, and $\rzf(\exc(\rzf))\subset\rzt^\circ\setminus\torus$ consists of points whose forward orbits do not meet $\ind(\rzf)$, there is are no preperiodic persistently exceptional curves, i.e. no curves $E\subset\exc(\rzf)$ such that $\rzf^m(E) \in E$ for some $m > 0$.  Hence there are only finitely many $n\geq 0$ such that $\rzf^n(p) \in \exc(\rzf)$.  If $n = N$ is the largest of these, then Lemma \ref{lem:favrebd} tells us that $\nu(T^*,p) \leq C^{N+1}\nu(T^*,\rzf^{N+1}(p))$, where $\rzf^{N+1}(p)$ is a free point in some pole, and then $\nu(T^*,\rzf^{n+1}(p)) = 0$ by the previous paragraph.
\end{proof}

\begin{cor}
\label{cor:downstairs}
The current $T^*$ in Theorem \ref{thm:invcurrents} is internal.  For any toric surfaces $X$ and $Y$, the difference $f_{XY}^* T_Y^* - \ddeg(f) T_X^*$ is an effective divisor.  
\end{cor}

\begin{proof} Since $T^*$ does not charge (in particular, external) curves, it is internal.  

Given toric surfaces $X$ and $Y$, choose $Y'\succ Y$ so that $f_{XY'}:X\tto Y'$ does not contract external curves of $X$.  We have by definition that $(\rzf^* T^*)_X = f_{XY'}^* T^*_{Y'}$.  Since $T_Y^* = \pi_{Y'Y*} T_{Y'}$,  Theorem \ref{thm:pushpull1} (applied to $f = \pi_{Y'Y}^{-1} = \pi_{YY'}$) implies that 
\begin{align*}
E:=\pi_{Y'Y}^* T^*_Y - T^*_{Y'} = \pi_{Y'Y}^* \pi_{Y'Y*} T^*_{Y'} - T^*_{Y'} = \extra_{\pi_{YY'}}(T_Y^*)
\end{align*}
is an effective divisor.  Hence by Proposition \ref{prop:stablecomp} and the fact that $\ind(\pi_{Y'Y}) = \emptyset$,
$$
f_{XY}^* T^*_Y = f_{XY'}^* \pi_{Y'Y}^* T^*_Y = f_{XY'}^*(T^*_{Y'} + E) = \ddeg(f)T^*_X + f_{XY'}^* E. 
$$
\end{proof}

\begin{cor}
\label{cor:downstairs2}
For any toric surface $X$ and any curve $C_X\subset X$, we have
$$
\lim_{n\to\infty} \frac{f_{XX}^{n*} C_X}{\ddeg(f)^n} \to (\alpha_{*X}\cdot C_X) T_X^*.
$$
\end{cor}

\begin{proof}
Let $C\in\pcc^+(\rzt)$ denote the tame toric current determined in $X$ by $C_X$.  (Note that if $C_X$ is itself external, then $C$ will have most external curves of $\rzt$ in its support.)  Choose an increasing sequence of toric surfaces $X \prec X_1\prec X_2\prec\dots$ such that $\ind(f^n_{X_nX})$ excludes all $\torus$-invariant points of $X_n$.  By Corollary \ref{cor:pushandpullalt} we have that $\rzf^{n*} C$ is determined on $X_n$ with representative
$
(\rzf^{n*} C)_{X_n} = f^{n*}_{X_nX} C_X.
$
Thus Proposition \ref{prop:stablecomp} and the fact that $\pi^{-1}_{X_nX}$ does not contract curves further imply that
$$
f^{n*}_{XX} C_X = (f^n_{X_nX}\circ \pi^{-1}_{X_nX})^* C_X = \pi_{X_n X*} f^{n*}_{X_nX} C_X
= 
\pi_{X_nX*} (\rzf^{n*} C)_{X_n} = (\rzf^{n*} C)_X.
$$
Dividing through by $\ddeg(f)^n$, letting $n\to\infty$ and applying Theorem \ref{thm:invcurrents} on the right side concludes the proof.
\end{proof}

\begin{proof}[Proof of Theorem \ref{thm:mainA}]

The hypotheses of Theorems \ref{thm:mainA} and Theorem \ref{thm:invcurrents}
are the same, so let $T^*$ be as in the conclusion of Theorem
\ref{thm:invcurrents}.  We claim that all conclusions of Theorem
\ref{thm:mainA} hold for the representative $T^*_{\cp^2}$ of $T^*$ in $\cp^2$.
Indeed the first conclusion follows immediately from Theorem \ref{thm:zerolelong}.
The second conclusion follows immediately from $\rzf^* T^* = \ddeg(f) T^*$ and
Corollary \ref{cor:downstairs} applied with $X = Y = \cp^2$. 
The third conclusion
of Theorem \ref{thm:mainA} follows from Corollary \ref{cor:downstairs2} with
$X=\cp^2$ and $C_X = C$. 

\end{proof}

All remaining results in this section are for toric maps that satisfy the hypotheses of Theorem \ref{thm:invcurrents} and also have \emph{small topological degree}, i.e. $\dtop(f) < \ddeg(f)$.  

A positive $(1,1)$ current $T\in\pcc^+(X)$ is laminar and strongly approximable
if, roughly speaking, it can be exhausted from below by sums currents of the
form $S = \int \Delta(x)\,\mu(x)$, where $\Delta(x)$ are integration currents
over pairwise disjoint disks with controlled geometry and $\mu$ is a positive
Borel measure on some disk transversal to all the $\Delta(x)$.  We will not
need the precise definition and so do not give it here, noting only that
laminarity has been an important property for studying the ergodic properties
of rational maps on complex surfaces with small topological degree.  We refer
the interested reader to e.g. \cite{Duj03}, \cite{Duj05}, and \cite{Duj06} for
much more context.

\begin{thm}
\label{thm:laminarity}
If $f$ satisfies the hypotheses of Theorem \ref{thm:invcurrents} and has small topological degree, then the current $T^*$ is laminar and strongly approximable in any toric surface.
\end{thm}

\begin{proof}
Let $X$ be an arbitrary toric surface.  Choosing an embedding $\iota:X\hookrightarrow \cp^N$, we let $C_X = \iota^* H$ be the pullback of a very general hyperplane $H\subset\C^N$.  
Since the complete linear system associated to such curves is basepoint free, we may assume that $C_X$ is smooth and disjoint from the finitely many $\torus$-invariant points of $X$.  
Hence $C_X$ fully realizes an internal curve $C\subset\rzt$.  We can further suppose that $C$ is disjoint from the countably many points in $\bigcup_{n\geq 1} \rzf^n(\exc(\rzf))$.  Then $\rzf^{-n}(C)$ is an (irreducible) curve for all $n\in\N$, and $\supp\rzf^{n*} C = \rzf^{-n}(C)$.  Hence in $X$, we have
$\supp f_{XX}^{n*} C_X = (\rzf^{n*}C)_X = \rzf^{-n}(C)_X$.  Now the laminarity and strong approximability follow from Theorem 2.12 in \cite{DDG10}.  Note that
while the statement of that theorem requires that $f$ is algebraically stable, the hypothesis is only used to guarantee the existence of the current $T^*_X$ and
the fact that it is a weak limit of $f_{X,X}^{n*} C_X/\ddeg(f^n)$.  These facts are
guaranteed by Theorem \ref{thm:invcurrents} and Corollary
\ref{cor:downstairs2}.
\end{proof}

Morally speaking, a rational surface map with small topological degree is close enough to being invertible that one also expects equidistribution for forward iterates of curves.  

\begin{thm}
\label{thm:forward}
If $f$ satisfies the hypotheses of Theorem \ref{thm:invcurrents} and has small topological degree, then there exists a nef and internal toric current $T_*\in\pcc^+(\rzt)$ such that for any other nef current $T\in\pcc^+(\rzt)$ we have
$$
\lim_{n\to\infty} \frac{\rzf^n_* T}{\ddeg(f)^n} = (T\cdot\alpha^*)\,T_*.
$$
In particular $T_*$ is woven and strongly approximable and satisfies $\rzf_* T_* = \ddeg(f) T_*$.
\end{thm}

The definition of woven and strongly approximable is exactly the same as that of laminar and strongly approximable except that the disks $\Delta(x)$ in the lower approximations are allowed to intersect each other.  A significant difference between forward and backward equidistribution currents $T_*$ and $T^*$, at least as far as our results go, is that we do not know whether $T_*$ can charge curves.  Instead we show here only that $T_*$ dominates no \emph{external} curves.

\begin{proof}
The construction of $T_*$ is quite similar to that of $T^*$ in Theorem \ref{thm:invcurrents}.  Hence we only sketch it.  Fixing a toric current $T = T_0 \in\pcci^+(\rzt)$ such that $(T\cdot\alpha^*) = 1$, and a toric surface $X$, we let $T_n = \ddeg(f)^{-n}\rzf^n_* T$ and write $T_0 - \bar T_0 = dd^c\varphi_0$ and, for all $n\geq 1$, 
$$
\ddeg(f)^{-1}f_* \bar T_{n-1} - \bar T_n = dd^c\varphi_n  
$$
where the relative potentials $\varphi_j$ are integrable on $X$.  Then as in the proof of Theorem \ref{thm:invcurrents}
$$
T_n = \bar{T}_n + dd^c \sum_{j=1}^n\frac{f^j_*\varphi_{n-j}}{\ddeg^j}.
$$
Here the pushforward $f^j_*\varphi$ of an integrable function $\varphi$ is given for a.e. $p\in X$ by $f^j_*\varphi(p):= \sum_{f^j(q) = p} \varphi(q)$.  Fixing $\mu < 1$, we use compactness of the sequence $(T_n)$, Theorem \ref{thm:volshrink} and Corollary \ref{cor:compactgrowth} to get constants $A,B>0$ such that
$$
\vol(\{|f^m_*\varphi_n|\geq t\}) \leq Ae^{-Bt(\mu/\dtop)^m}
$$
for all $n,m,t\geq 0$.  Here the $\dtop^{-m}$ in the exponent on the right is needed to account for the number of preimages of a general point $p$ by $f^m$.   With this, one obtains an upper bound
$$
\int_X \frac{f^j_*\varphi_{n-j}}{\ddeg^j}\,dV \leq \frac{1}{\ddeg^j} \int_0^\infty \vol\, \{p\in X: |f^j_*\varphi_{n-j}(p)|\geq t\}\,dt \leq \frac{A}{B} \left(\frac{\dtop}{\mu\ddeg}\right)^j.
$$
Since we assume $f$ has small topological degree, for $\mu < 1$ large enough,
the right side of this estimate converges to zero geometrically as
$j\to\infty$, which suffices to show that $(T_n)$ converges to a current
$T_*\in\pcc^+(\rzt)$.  Necessarily $\ch{T_*} = \alpha_*$ and $\rzf_* T_* =
\ddeg T_*$, and one shows as in the proof of Theorem \ref{thm:invcurrents} that
the limit $T_*$ is independent of the initial current $T$.  Also as before, one
shows that the initial toric current $T$ need not be positive or internal as
long as it represents a nef class in $\hoo(\rzt)$.  By the same reasoning as in the proof of Theorem \ref{thm:laminarity}, Theorem 3.6 in \cite{DDG10}
implies that $T_*$ is woven.

It remains to show that $T_*$ is internal.  For this, let us fix an external curve $C_\tau\subset \rzt$ and suppose $T_*\geq a C_\tau$ for some constant $a > 0$.  Let $\tau' = A_f^{-1}(\tau)$.  Then $C_\tau'\subset\rzt$ is the unique (external) curve such that $\rzf(C_\tau') = C_\tau$.  Let $Y$ be a toric surface realizing $C_\tau$ and $X$ be a toric surface that realizes $C_{\tau'}$ and for which $\ind(f_{XY}) = \ind(\rzf)$ has no $\torus$-invariant points.  Then by definition of pushforward, we have on the one hand that
$$
a\ddeg C_\tau \leq (\ddeg T_*)_Y = (\rzf_* T_*)_Y = f_{XY*} (T_*)_X.
$$
On the other hand $C_\tau$ is not contained in $f_{XY}(\ind(f_{XY}))$, since the latter consists only of internal curves, and $C_{\tau'}$ is the unique curve in $X$ such that $f_{XY}(C_\tau') = C_\tau$.  Since $f_{XY*} C_{\tau'} = (\deg f_{XY}|C_{\tau'}) C_\tau \leq \dtop C_\tau$, we infer that $(T_*)_X \geq a(\ddeg/\dtop) C_{\tau'}$.  That is, $T_*\geq a(\ddeg/\dtop) C_{\tau'}$.  Repeating this estimate with $\rzf^n$ in place of $\rzf$, we find a sequence of external curves $C_n$ such that $T_* \geq a(\ddeg/\dtop)^n C_n$ for all $n\in\N$. In particular, the Lelong numbers of $T_*$ along external curves are unbounded, contradicting Theorem \ref{thm:lelongbd}.  So $T_*$ does not dominate a positive multiple of any external curve and must therefore be internal.
\end{proof}

Theorem \ref{thm:mainD} follows from Theorems \ref{thm:laminarity} and \ref{thm:forward} in the same way Theorem \ref{thm:mainA} followed from the earlier results of this section.

\nocite{}

\end{document}